\documentclass{article}

\usepackage{style}

\usepackage{authblk}
\usepackage[margin=1in]{geometry}

\title{Non-Intrusive~Uncertainty~Quantification using Reduced~Cubature~Rules}

\author[1,2]{L.M.M.~van~den~Bos\footnote{Corresponding author: \texttt{l.m.m.van.den.bos@cwi.nl}}}
\author[1,2]{B.~Koren}
\author[3]{R.P.~Dwight}

\affil[1]{Eindhoven University of Technology, P.O.~Box 513, 5600~MB Eindhoven, Netherlands}
\affil[2]{Centrum Wiskunde \& Informatica, P.O.~Box 94079, 1090~GB Amsterdam, Netherlands}
\affil[3]{Delft University of Technology, P.O.~Box 5, 2600~AA Delft, Netherlands}

\begin{document}
\maketitle

\begin{abstract}
\noindent For the purpose of uncertainty quantification with collocation, a method is proposed for generating families of one-dimensional nested quadrature rules with positive weights and symmetric nodes. This is achieved through a \emph{reduction} procedure: we start with a high-degree quadrature rule with positive weights and remove nodes while preserving symmetry and positivity. This is shown to be always possible, by a lemma depending primarily on Carath\'eodory's theorem. The resulting one-dimensional rules can be used within a Smolyak procedure to produce sparse multi-dimensional rules, but weight positivity is lost then. As a remedy, the reduction procedure is directly applied to multi-dimensional tensor-product cubature rules. This allows to produce a family of sparse cubature rules with positive weights, competitive with Smolyak rules. Finally the positivity constraint is relaxed to allow more flexibility in the removal of nodes. This gives a second family of sparse cubature rules, in which iteratively as many nodes as possible are removed. The new quadrature and cubature rules are applied to test problems from mathematics and fluid dynamics. Their performance is compared with that of the tensor-product and standard Clenshaw--Curtis Smolyak cubature rule.
\end{abstract}

\begingroup
\endgroup

\section{Introduction}
The problem of non-intrusive \emph{uncertainty quantification} (UQ) in expensive computational models is considered, for example computational fluid dynamics (CFD) models. Consider a model with $d$ uncertain parameters having specified distributions. The objective is to obtain statistics on the outputs of the model, while using the model only as a black box (i.e.\ non-intrusively). We wish to obtain accurate statistics with as few evaluations of the model as possible.

The canonical method is Monte Carlo (MC), with the well-known dimension independent convergence rate of $\mathcal{O}(1 / \sqrt{N})$, where $N$ is the number of samples. For sufficiently low dimension $d$ this can be improved to $\mathcal{O}\bigl( (\log N)^d / N\bigr)$ using Quasi Monte Carlo methods, see e.g.~\cite{Ye1998,Caflisch1998}. For $d \lesssim 10$ this can be significantly further improved by using methods based on polynomial approximation of the model output in the parameter space. This case is studied in this paper. For a sufficiently smooth parametrized model, spectral convergence is obtained. Stochastic Collocation (SC) \cite{Najm2009,Eldred2009,Xiu2005} is such a method which uses either tensor products or sparse grids to sample the parameter space. Quadrature weights on these grids allow the evaluation of statistics. We mention also the hybrid techniques of Witteveen et al.~\cite{Witteveen2010,Witteveen2013}, which use piecewise polynomial interpolation on random MC grids. Approaches based on compressed sensing (and therefore not dependent on quadrature rules) have also been studied \cite{Doostan2011,Blatman2011}.

SC methods can be regarded as cubature rules targeted at moderate dimensional spaces. The conventional tensor-product cubature rule introduces a large number of nodes for moderate $d$. Sparse grid strategies are therefore required, e.g. a Smolyak sparse grid \cite{Smolyak63,Novak1999}. These in turn require nested one-dimensional quadrature rules for the optimal result, but no general strategy exists to create these nested quadrature rules with positive weights for arbitrary distributions. Furthermore, a Smolyak procedure does not guarantee that the weights of the resulting multi-dimensional rule are positive (and therefore is not necessarily numerically stable), even if the underlying one-dimensional rule has positive weights.

Sparse grid techniques have been studied thoroughly by various authors. For example, Garcke, Gerstner, and Griebel (see e.g.~\cite{Garcke2013,Gerstner1998}) studied the generation of sparse grids, among others with improvements such as dimension dependent adaptivity. Narayan and Jakeman \cite{Narayan2014} studied the construction of quadrature rules as input for the Smolyak sparse grid. Nobile et al.\ \cite{Nobile2008} studied the effectiveness of sparse grids compared with MC methods. Anisotropic extensions (i.e.\ different quadrature rules in different dimensions) were also studied by Nobile et al.\ \cite{Nobile2008-2}. Pfl\"uger \cite{Pfluger2010} studied adaptive sparse grids, where locally the grid is refined if necessary, yielding a strategy to determine a sparse grid that depends on the specifics of the model.

The present paper has two major contributions. Firstly a method is introduced for constructing a nested family of one-dimensional quadrature rules \emph{with positive weights}, from any single high-order rule (with positive weights). Thus given a quadrature rule for a specific probability distribution, a nested family can be constructed, suitable for use in a Smolyak procedure. If the original rule is symmetric, the symmetry of the entire family is guaranteed. Secondly, in the multi-dimensional case a closely related operation can be performed. Starting from a tensor-product rule, nodes can be removed successively while maintaining positivity of all weights and symmetry of the rule. The result is a new kind of sparse grid with only positive weights. Because the positivity restriction is quite limiting, also the case where negative weights are permitted is considered. This allows to remove more nodes at each step of the reduction procedure. All resulting rules are well suited for UQ, as top level quadrature rules can be chosen separately for each parameter, without any concern about nesting. The new rules are demonstrated on the Genz test functions, two CFD test cases, and compared to the tensor-product and Smolyak rules.

The study is set up as follows. First, in the section hereafter the UQ problem is formulated. In the next section some useful well-known methods are discussed. In Section~\ref{sec:redquadrule} the reduced quadrature rule is introduced, which is extended to a multi-dimensional setting in Section~\ref{sec:redcubrule}. The introduced cubature rule is compared with conventional cubature rules in Section~\ref{sec:numerics}. Firstly a mathematical comparison is made using test functions. Secondly the cubature rules are applied in UQ for the standard lid-driven cavity flow problem computed through a Lattice Boltzmann method with two uncertain parameters. To show the effectiveness in high-dimensional problems, it is finally applied to a three-dimensional aircraft aerodynamics problem, computed through a finite-volume Euler-flow model, considering seven uncertain parameters.

\section{Uncertainty Quantification}
\label{sec:probform}
Consider a discrete computational problem for a quantity of interest
\begin{equation}
	v \coloneqq v\bigl(s(\xib)\bigr),
\end{equation}
where $s: \mathbb{R}^d \rightarrow \mathbb{R}^n$ is the state of some system, satisfying
\begin{equation}
	R(s; \xib) = 0,
\end{equation}
where $R$ is typically a discretization of a continuous PDE, including initial and boundary conditions, and where $n$ is the dimension of the discrete state. The quantity of interest $v: \mathbb{R}^n \rightarrow \mathbb{R}$ is a single quantity derived from the full state. The parameters $\xib$ are $d$ random variables, that is $\xib: \Omega \rightarrow \Xi$, which are assumed to be independent and square-integrable (i.e.\ having finite variance), with respect to the probability space $(\Omega, \mathcal{F}, P)$ with $\Xi \subset \mathbb{R}^d$, $\Omega \subset \mathbb{R}^d$, $\mathcal{F} \subset 2^\Omega$, and $P$ the probability measure. Although infinite dimensional, random fields can be fit into this framework after the application of a truncated Karhunen-Lo\`eve expansion~\cite{Ghanem1991,Xiu2010}.

The problem is now to determine the probability distribution and statistical moments of $u$, with $u(\xib) \coloneqq v\bigl(s(\xib)\bigr)$. The focus is on the latter, i.e.~on determining
\begin{equation}
	\mathbb{E}[u^l(\boldsymbol \xi)] \coloneqq \int_\Xi u^l(\boldsymbol \xi) \dd P(\boldsymbol \xi), \text{ for } l = 1, 2, \dots
\end{equation}
The collocation approach is to approximate this integral using a weighted combination of a finite number of samples $\{\boldsymbol \xi_k\}_{k=1, \dots, N} \in \Xi$ as
\begin{equation}
	\label{eq:quadcubrule}
	\mathbb{E}[u^l(\boldsymbol \xi)] \simeq \sum_{k=1}^N u^l(\boldsymbol \xi_k) w_k, \text{ for } l = 1, 2, \dots,
\end{equation}
where $\{w_k\}_{k=1, \dots, N} \in \mathbb{R}$ are the weights. $u^l(\boldsymbol \xi_k) = (u(\boldsymbol \xi_k))^l$ is determined by solving the (potentially expensive) deterministic discrete problem
\begin{equation}
	R(s(\xib_k); \xib_k) = 0
\end{equation}
for $s(\xib_k)$, and by evaluating $u(\xib_k)$.

In the remainder of this paper the term quadrature rule is used in a one-dimensional setting (i.e.\ $d = 1$) and the term cubature rule otherwise. All properties of cubature rules also apply to quadrature rules (but not vice versa).

\section{Numerical integration - Terminology and basic principles}
\subsection{Quadrature and cubature rules}
\label{subsec:quadcubrule}
Let $\mathbb{P}(K, d)$ be all $d$-variate polynomials of degree equal to or less than $K$. The \emph{degree} of a cubature rule is defined as the number $K$ such that all polynomials $p \in \mathbb{P}(K, d)$ are integrated exactly and at least one polynomial $p \in \mathbb{P}(K+1, d)$ exists that is not integrated exactly.

We consider a set of cubature rules to be \emph{nested} if the nodes of a smaller cubature rule are also nodes of all larger cubature rules. If a cubature rule is nested, error estimates can be naturally constructed by comparing the approximation on two consecutive levels. In addition to nesting, it is desirable that rules are (i) \emph{symmetric}, meaning the nodes and weights have the same symmetry as the underlying probability distribution, and (ii) \emph{positive}, meaning all weights are positive. Symmetric quadrature rules naturally represent the underlying distribution and are necessary in the multi-dimensional case to reduce the number of nodes (which is done in the second part of this paper). Quadrature rules with positive weights are unconditionally numerically stable if they are evaluated and yield an integration operator with norm equal to 1. For example Gaussian quadrature rules are positive, irrespective of the underlying distribution, and symmetric if the distribution is symmetric~\cite{Golub1969}. However they are not nested. The nested Clenshaw--Curtis rule is usually applied with a uniform distribution, in which case weights are positive~\cite{Clenshaw1960} -- but this is not true if weights are constructed for an arbitrary distribution.

\subsection{The generalized Vandermonde-matrix}
\label{subsec:genvanmat}
If $N$ distinct one-dimensional quadrature nodes (denoted by $\{\xi_k\}_{k=1}^N\subset\mathbb{R}$) are specified, the weights of the quadrature rule can be determined such that it is a rule of degree $N-1$ by solving the following linear system:
\begin{equation}
	\label{eq:vandermondesys}
	\sum_{k=1}^N \xi_k^j w_k = \int_\Xi \xi^j \dd P(\xi), \text{ for all } j = 0, \dots, N-1.
\end{equation}

The matrix of this system is a Vandermonde-matrix, hereafter denoted by $V$ and defined by $V_{j,k} = \xi_k^j$. This system has a unique solution for distinct nodes. Hence the quadrature rule is unambiguously specified by the nodes only. For large $N$ the Vandermonde-matrix becomes ill-conditioned, such that for various quadrature rules more efficient algorithms exist to determine both the nodes and the weights, e.g.\ the algorithm of Golub and Welsch \cite{Golub1969} can be used to determine a Gauss quadrature rule and Clenshaw--Curtis rules can be determined efficiently using a Fast Fourier transform \cite{Waldvogel2006}.

Generalizing to a multi-dimensional setting, let $\{\xib_k\}_{k=1}^N \subset \mathbb{R}^d$ be $N$ cubature nodes. Integration conditions result in the system
\begin{equation}
	\label{eq:genvandermondesys}
	\sum_{k=1}^N m_j(\boldsymbol \xi_k) w_k = \int_\Xi m_j(\boldsymbol \xi) \dd P(\boldsymbol \xi), \text{ for all } j = 1, \dots, N,
\end{equation}
where $m_j$ is the $j^\text{th}$ monomial under some ordering. We call the matrix $G_{j,k} = m_j(\boldsymbol \xi_k)$ the \emph{generalized Vandermonde-matrix}. As is well-known $G$ may be singular, but for tensor-product rules $G$ is non-singular, as it can be formed by the Kronecker product of the Vandermonde-matrix of the quadrature rules \cite{Laub2005}. Just as in the one-dimensional case, in general this matrix can become ill-conditioned for high-dimensional polynomial spaces or large $N$.

\subsection{Smolyak cubature rules}
\label{subsec:smolyak}
The Smolyak procedure \cite{Smolyak63} is a method of constructing ``sparse'' cubature rules from a family of (typically nested) quadrature rules indexed by level. Rather than building the tensor product of the one-dimensional rule at the finest level in every direction, Smolyak builds tensor products of fine levels in some directions and coarse levels in others, and combines many such products in a single rule. If the one-dimensional rule is nested, these tensor products have many coincident nodes -- reducing the total cost. The resulting set of nodes is known as a \emph{sparse grid} \cite{Novak1999}. 

A concise formula for the Smolyak rule \cite{Wasilkowski1995} is
\begin{equation}
	\mathcal{S}_K = \sum_{\substack{K-d+1 \leq \|\alpha\|_1 \leq K \\ \alpha \in \mathbb{N}^d}} {(-1)}^{K-\|\alpha\|_1} \binom{d-1}{K-\|\alpha\|_1} \bigotimes_{k=1}^d \mathcal{Q}_{N_{\alpha_k}},
\end{equation}
where $\{N_k\}_{k=1}^N \subset \mathbb{N}$ is an increasing sequence and $\mathcal{Q}_{N_k}$ is an $N_k$-node quadrature rule. $\{N_k\}$ is typically an exponentially growing sequence, because then the Smolyak cubature rule has a relatively high degree, which can be seen in the following lemma \cite{Novak1996,Novak1999}.
\begin{lemma}
	\label{lmm:smoldegree}
	Let $\{N_k\}$ grow exponentially in $k$. Then $S_K$ has at least degree $2(K-d)+1$.
\end{lemma}
In this paper, all Smolyak cubature rules are generated using quadrature rule sets with the following exponentially growing numbers of nodes:
\begin{equation}
	N_k = \begin{cases}
		1 & \text{if $k = 1$,} \\
		2^{k-1}+1 & \text{otherwise.}
	\end{cases}
\end{equation}
This sequence is chosen such that the sequence of $N_k$ nodes of the Clenshaw--Curtis quadrature rule is nested.

\begin{figure}
	\centering
	\begin{minipage}{.3\textwidth}
		\includegraphics[width=\textwidth]{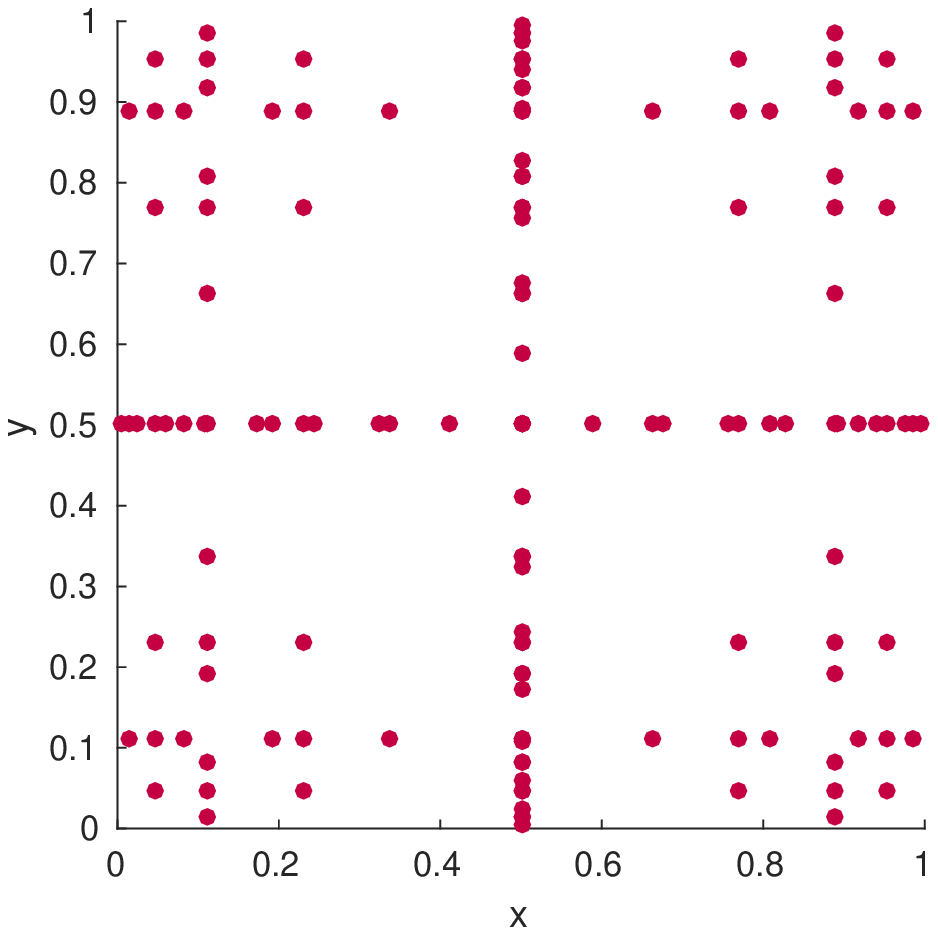}
		\subcaption{Gauss--Legendre \\ (141~nodes)}
		\label{fig:smolyak2dnnested}
	\end{minipage}
	\begin{minipage}{.3\textwidth}
		\includegraphics[width=\textwidth]{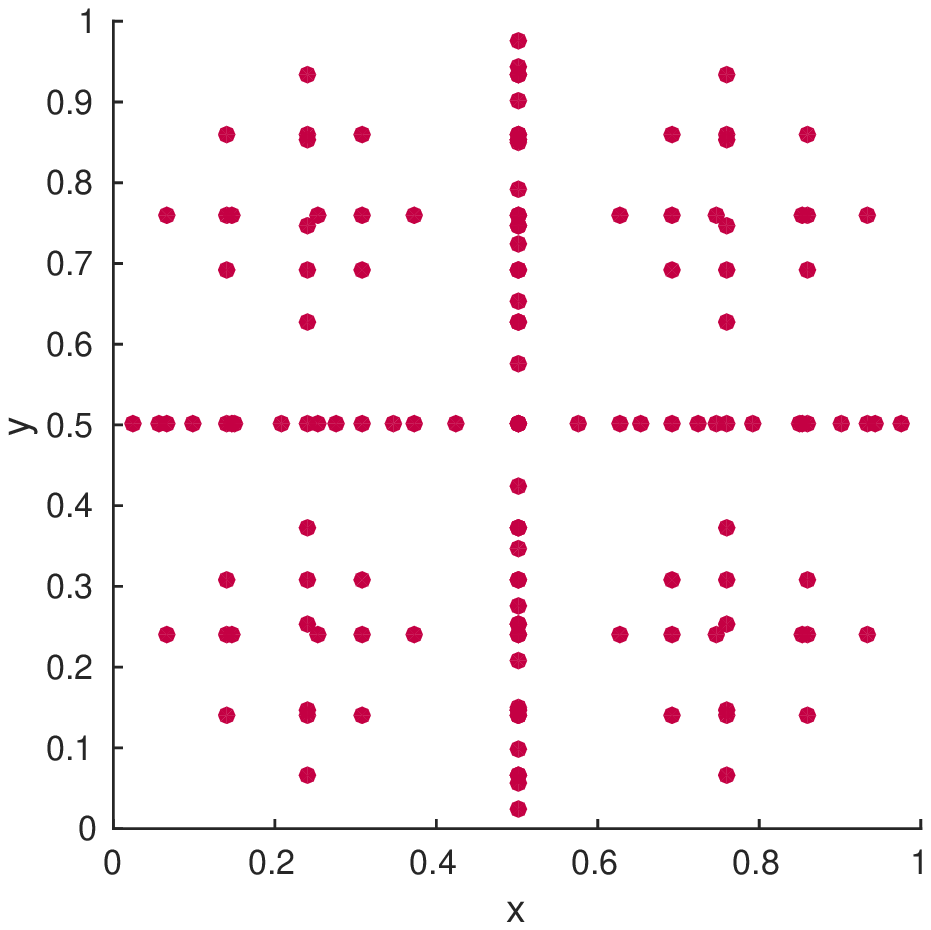}
		\subcaption{Gauss--Jacobi \\ (141~nodes)}
		\label{fig:smolyak2dnnestedb}
	\end{minipage}
	\begin{minipage}{.3\textwidth}
		\includegraphics[width=\textwidth]{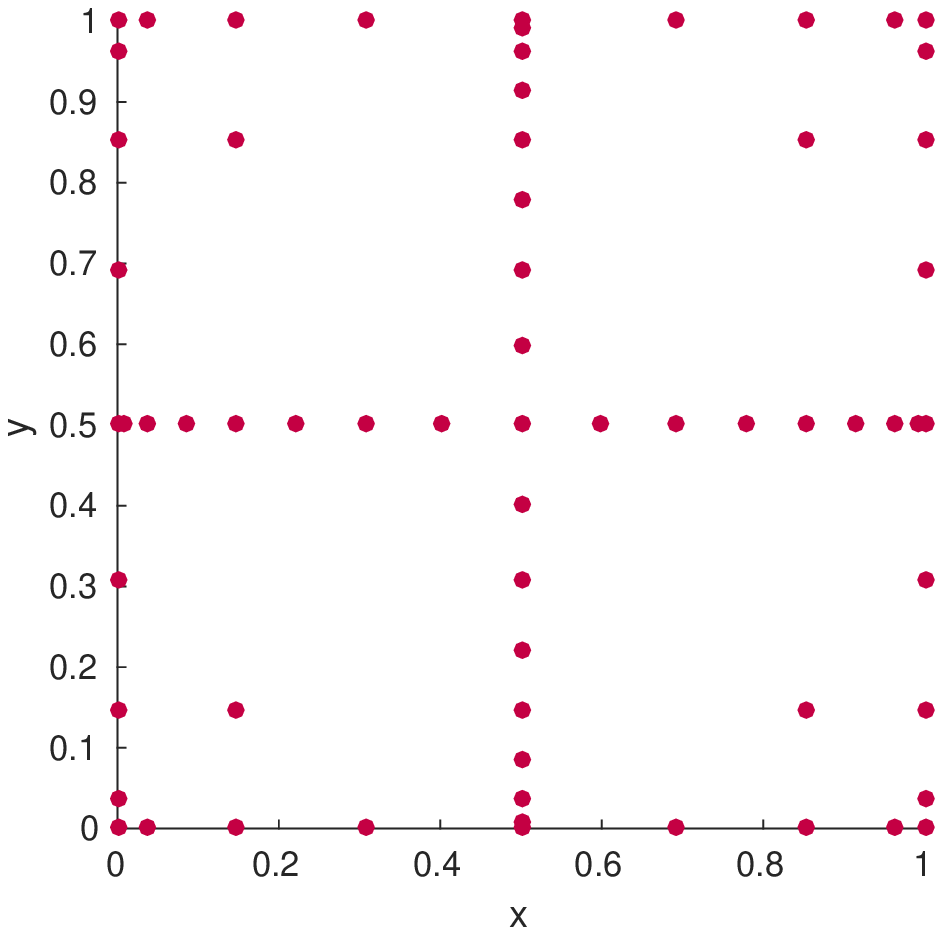}
		\subcaption{Clenshaw--Curtis \\ (65~nodes)}
		\label{fig:smolyak2dnested}
	\end{minipage}
	\caption{Two-dimensional Smolyak cubature rule nodes of several quadrature rules.}
	\label{fig:smolyak2d}
\end{figure}

Smolyak rules do not have positive weights in general, but the condition number $\kappa$ of the cubature rule is bounded if the original quadrature rule has positive weights \cite{Novak1999}:
\begin{equation}
	\label{eq:smolnorm}
	\kappa \coloneqq \frac{\sum_{k=1}^N |w_k|}{\sum_{k=1}^N w_k} = \sum_{k=1}^N |w_k| = \mathcal{O}\left( (\log N)^{d-1} \right).
\end{equation}

In Figure~\ref{fig:smolyak2d} two-dimensional sparse grids resulting from Smolyak applied to Gauss--Legendre, Gauss--Jacobi (with $\alpha = \beta = 4$), and Clenshaw--Curtis quadrature rules are plotted. Thanks to nesting, Clenshaw--Curtis rules result in less than half of the number of nodes of the other rules, and this benefit will improve in higher dimensions. If Gaussian rules are used, the weights of the original rule are certainly positive, and therefore \eqref{eq:smolnorm} holds.

\section{Carath\'eodory reduction of quadrature rules}
\label{sec:redquadrule}
In this section new quadrature rules are introduced based on the removal of nodes. Given an initial quadrature rule with positive weights, a set of nested quadrature rules with positive weights is determined by removing nodes while retaining symmetry. We call these new rules \emph{reduced quadrature rules} and the procedure to remove nodes the \emph{reduction step}. These rules are by construction nested and do have positive weights. The reduction step can be applied in such a way that the rules are also symmetric.

\subsection{Reduction step}
The principle of the reduction step is as follows. First, recall the linear system \eqref{eq:vandermondesys}:
\begin{equation}
	\underbrace{\begin{pmatrix}
		\xi_1^0 & \xi_2^0 & \dots & \xi_N^0 \\
		\xi_1^1 & \xi_2^1 & \dots & \xi_N^1 \\
		\vdots & \vdots & \ddots & \vdots \\
		\xi_1^{N-2} & \xi_2^{N-2} & \dots & \xi_N^{N-2} \\
		\xi_1^{N-1} & \xi_2^{N-1} & \dots & \xi_N^{N-1}
	\end{pmatrix}}_{V}
	\begin{pmatrix}
		w_1 \\
		w_2 \\
		\vdots \\
		w_{N-1} \\
		w_N
	\end{pmatrix}
	=
	\begin{pmatrix}
		\int_\Xi \xi^0 \dd P(\xi) \\
		\int_\Xi \xi^1 \dd P(\xi) \\
		\vdots \\
		\int_\Xi \xi^{N-2} \dd P(\xi) \\
		\int_\Xi \xi^{N-1} \dd P(\xi)
	\end{pmatrix},
\end{equation}
which describes an $N$-node quadrature rule of degree $N-1$. The goal is to find a subset of $N-1$ nodes that form a quadrature rule of degree $N-2$. Such a rule can easily be determined by considering the following system:
\begin{equation}
	\label{eq:redvandermondesys}
	\underbrace{\begin{pmatrix}
		\xi_1^0 & \xi_2^0 & \dots & \xi_N^0 \\
		\xi_1^1 & \xi_2^1 & \dots & \xi_N^1 \\
		\vdots & \vdots & \ddots & \vdots \\
		\xi_1^{N-2} & \xi_2^{N-2} & \dots & \xi_N^{N-2}
	\end{pmatrix}}_{V_{-1}}
	\begin{pmatrix}
		w_1 \\
		w_2 \\
		\vdots \\
		w_{N-1} \\
		w_N
	\end{pmatrix}
	=
	\begin{pmatrix}
		\int_\Xi \xi^0 \dd P(\xi) \\
		\int_\Xi \xi^1 \dd P(\xi) \\
		\vdots \\
		\int_\Xi \xi^{N-2} \dd P(\xi)
	\end{pmatrix},
\end{equation}
where $V_{-1}$ is the matrix $V$ after the removal of the last row. This notation is used hereafter in a more general way: $A_{-k}$ denotes matrix $A$ after the removal of the last $k$ rows.

Each column of $V_{-1}$ is related to a node of the quadrature rule, so removing a column from the matrix above and solving the resulting system yields a nested quadrature rule of degree $N-2$. The question remains which column can be removed such that the system that remains has a solution with positive elements. The answer follows from (a variant of) the well-known Carath\'eodory theorem. The constructive proof will be useful later.

\begin{theorem}[Carath\'eodory's theorem]
	Let $\mathbf{v}_1, \mathbf{v}_2, \dots, \mathbf{v}_N, \mathbf{v}_{N+1}$ be $N+1$ vectors spanning an $N$-dimensional space. Let $\mathbf{v} = \sum_{k=1}^{N+1} \lambda_k \mathbf{v}_k$ with $\lambda_k \geq 0$. Then there exist $\beta_k \geq 0$ such that $\mathbf{v} = \sum_{k \in I} \beta_k \mathbf{v}_k$ and $I \subset \{1, \dots, N+1\}$ with $|I| \leq N$.
\end{theorem}
\begin{proof}
	Because $\mathbf{v}_1, \dots, \mathbf{v}_{N+1}$ are $N+1$ vectors in an $N$-dimensional space, they must be linearly dependent. So there are $c_k$, not all equal to zero, such that
	\begin{equation}
		\sum_{k=1}^{N+1} c_k \mathbf{v}_k = 0.
	\end{equation}
	So for any $\alpha \in \mathbb{R}$, it is true that
	\begin{align}
		\mathbf{v} &= \sum_{k=1}^{N+1} \lambda_k \mathbf{v}_k - \alpha \sum_{k=1}^{N+1} c_k \mathbf{v}_k \\
		&= \sum_{k=1}^{N+1} (\lambda_k - \alpha c_k) \mathbf{v}_k.
	\end{align}
	Without loss of generality, we assume that at least one $c_k > 0$. Then the following choice is well-defined:
	\begin{equation}
		\alpha = \min_{k=1,\dots,N+1} \left\{\frac{\lambda_k}{c_k} : c_k > 0\right\} \eqqcolon \frac{\lambda_{k_0}}{c_{k_0}}.
	\end{equation}

	Choosing $\beta_k = \lambda_k - \alpha c_k$, it is true that $\beta_{k_0} = 0$ so with $I = \{1, 2, \dots, k_0-1, k_0+1, \dots, N\}$ the following holds:
	\begin{align}
		\mathbf{v} &= \sum_{k \in I} \beta_k \mathbf{v}_k. \qedhere
	\end{align}
\end{proof}

Carath\'eodory's theorem can be interpreted as a column-removal step. First, let $\mathbf{v}$ be the columns of $V_{-1}$. The elements $c_k$ from the proof form a null vector of the matrix. Determining $\alpha$ and $k_0$ from the proof yields that $w_k - \alpha c_k \geq 0$ and $w_{k_0} - \alpha c_{k_0} = 0$, such that the node $x_{k_0}$ can be removed from the quadrature rule. This yields a quadrature rule of $N-1$ nodes of degree $N-2$ with positive weights, which was the goal.

Repeatedly applying the reduction step to an existing quadrature rule yields a set of nested quadrature rules with positive weights. The reduction step is however not unique in general. The null vector $\mathbf{c}$ contains both positive and negative elements (guaranteed by the fact that the first row of the matrix contains only positive values), so $-\mathbf{c}$ is also a null vector with both positive and negative elements and each null vector can be used to eliminate a different node.

This non-uniqueness imposes a choice. We suggest a heuristic greedy strategy of eliminating at each reduction step that node with the lowest probability based on the underlying probability density function (of the two nodes that can be eliminated). Nodes with high probability are retained. We call this the \emph{prior criterion}. We shall see that for symmetric distributions and rules, this criterion does not apply (nodes have equal probability), which will be discussed in the next section.

The Smolyak cubature rule has been determined for several sets of reduced Gauss quadrature rules (see Figure~\ref{fig:smolyaknested}, here the standard normal distribution is used) using the prior criterion. If two nodes have equal probability, the node which is most far from the center is removed. Comparing this to the Smolyak cubature rules which were determined previously (see Figure~\ref{fig:smolyak2d}) yields that the number of nodes is the same as for the Clenshaw--Curtis quadrature rule, but the weights are positive (hence \eqref{eq:smolnorm} can be used) and the location of the nodes is dependent on the distribution.

In the plots of the quadrature rules (below the sparse grids in Figure~\ref{fig:smolyaknested}) it is clearly visible that the quadrature rules are not symmetric but do have positive weights. We extend the reduction step such that the rules are symmetric.

\begin{figure}
	\centering
	\begin{minipage}{.3\textwidth}
		\centering
		\includegraphics[width=\textwidth]{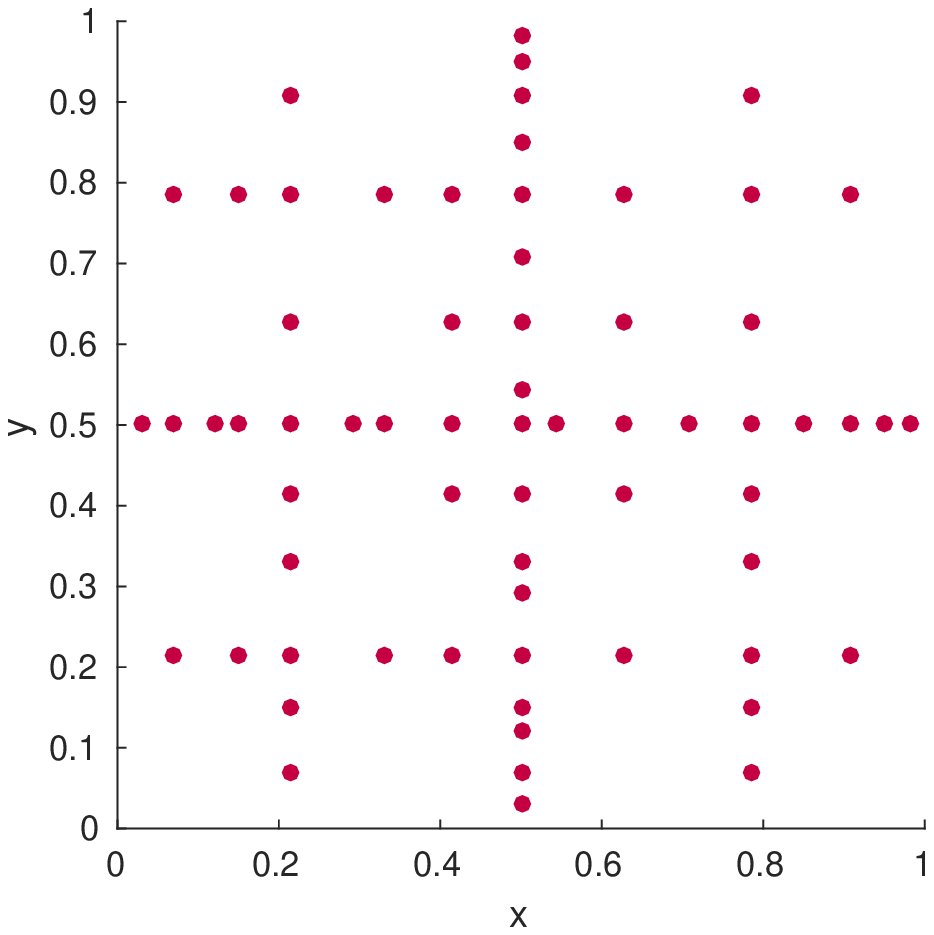}
		\includegraphics[width=\textwidth]{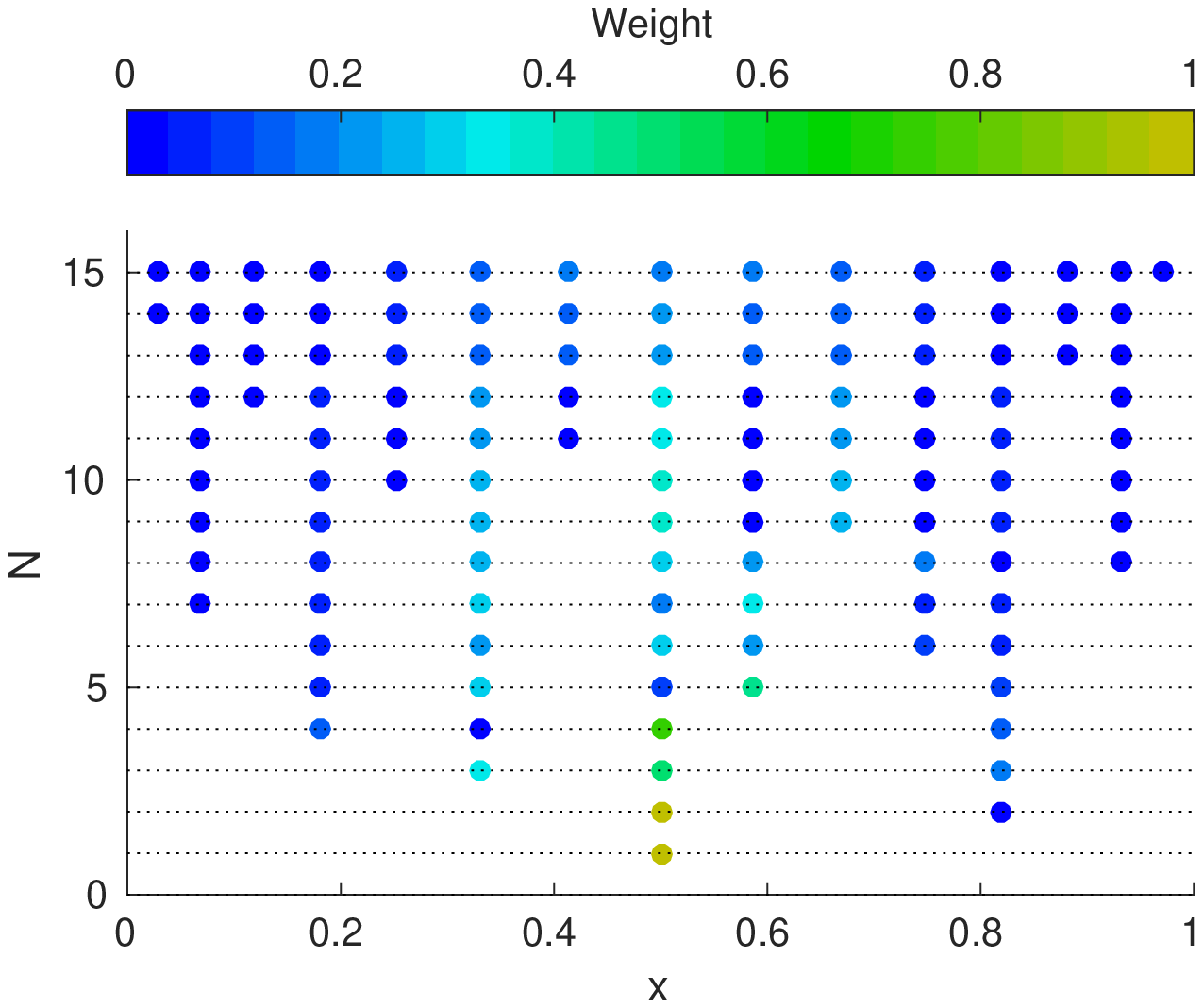}
		\subcaption{Reduced~Gauss--Jacobi}
		\label{fig:smolyaknestedbeta23}
	\end{minipage}
	\begin{minipage}{.3\textwidth}
		\centering
		\includegraphics[width=\textwidth]{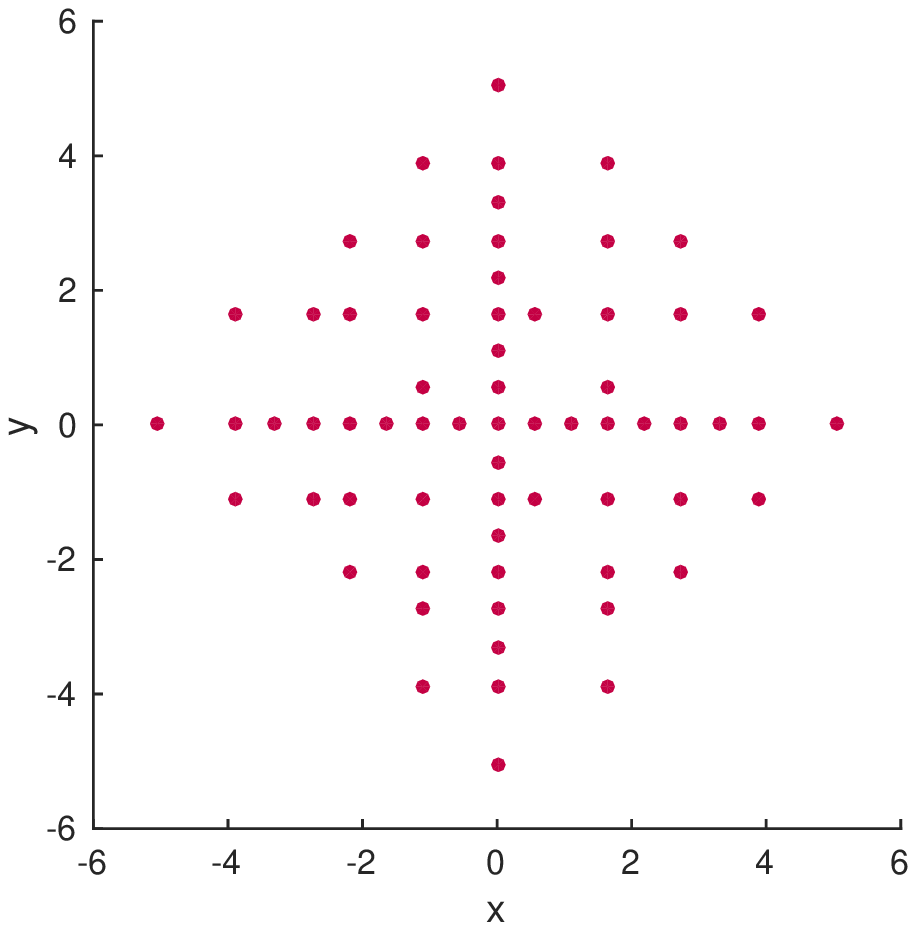}
		\includegraphics[width=\textwidth]{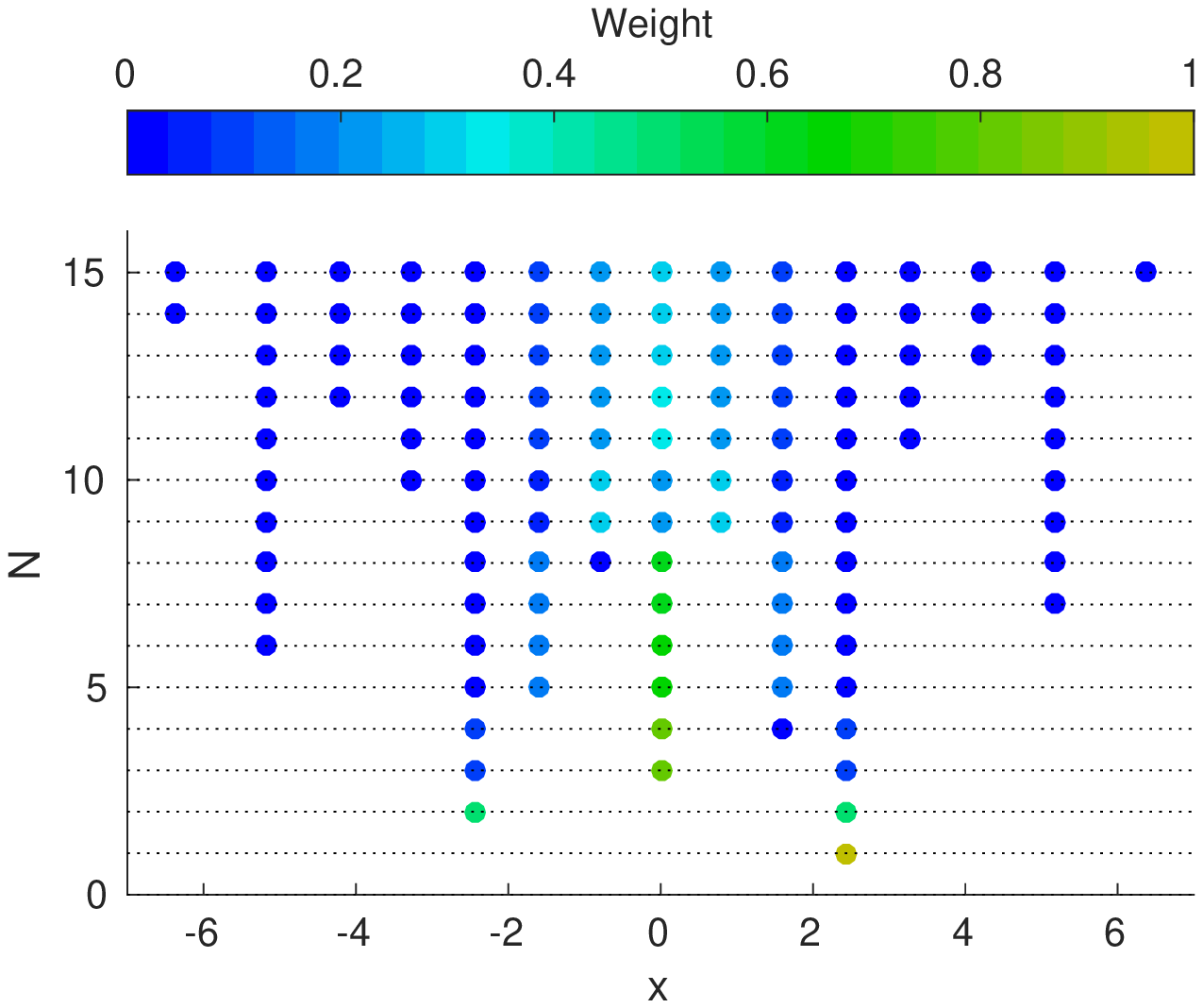}
		\subcaption{Reduced~Gauss--Hermite}
		\label{fig:smolyaknestednorm01}
	\end{minipage}
	\begin{minipage}{.3\textwidth}
		\centering
		\includegraphics[width=\textwidth]{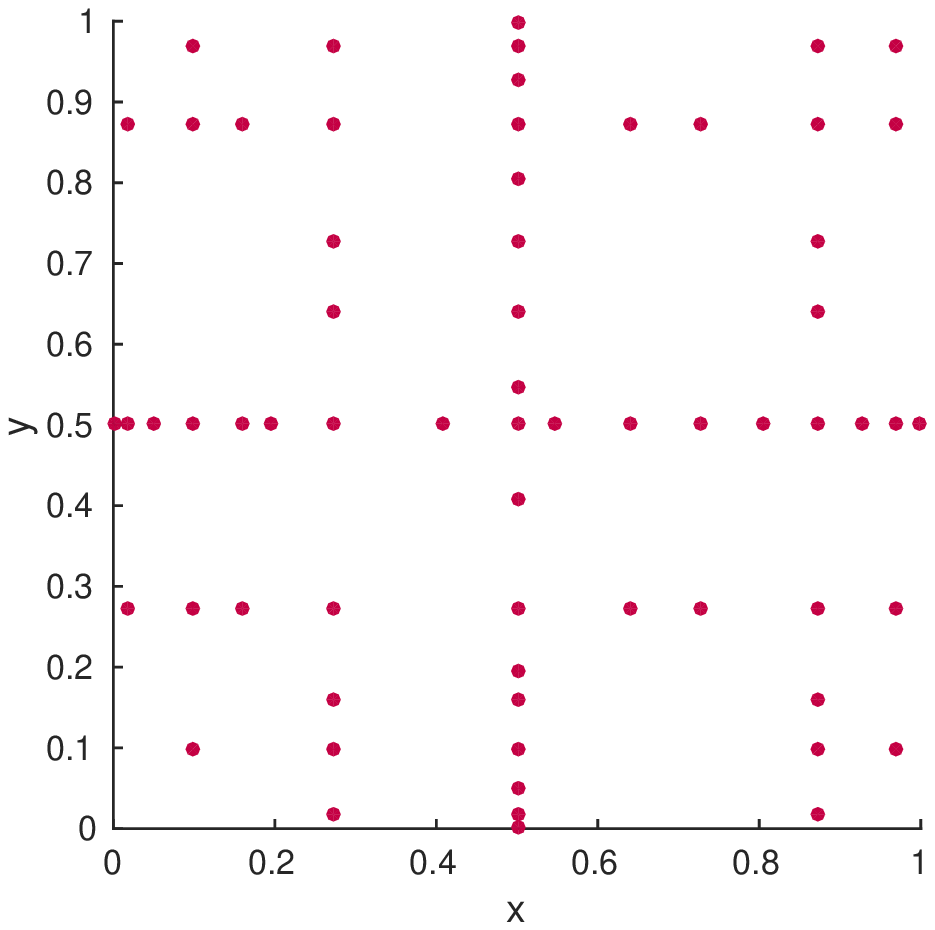}
		\includegraphics[width=\textwidth]{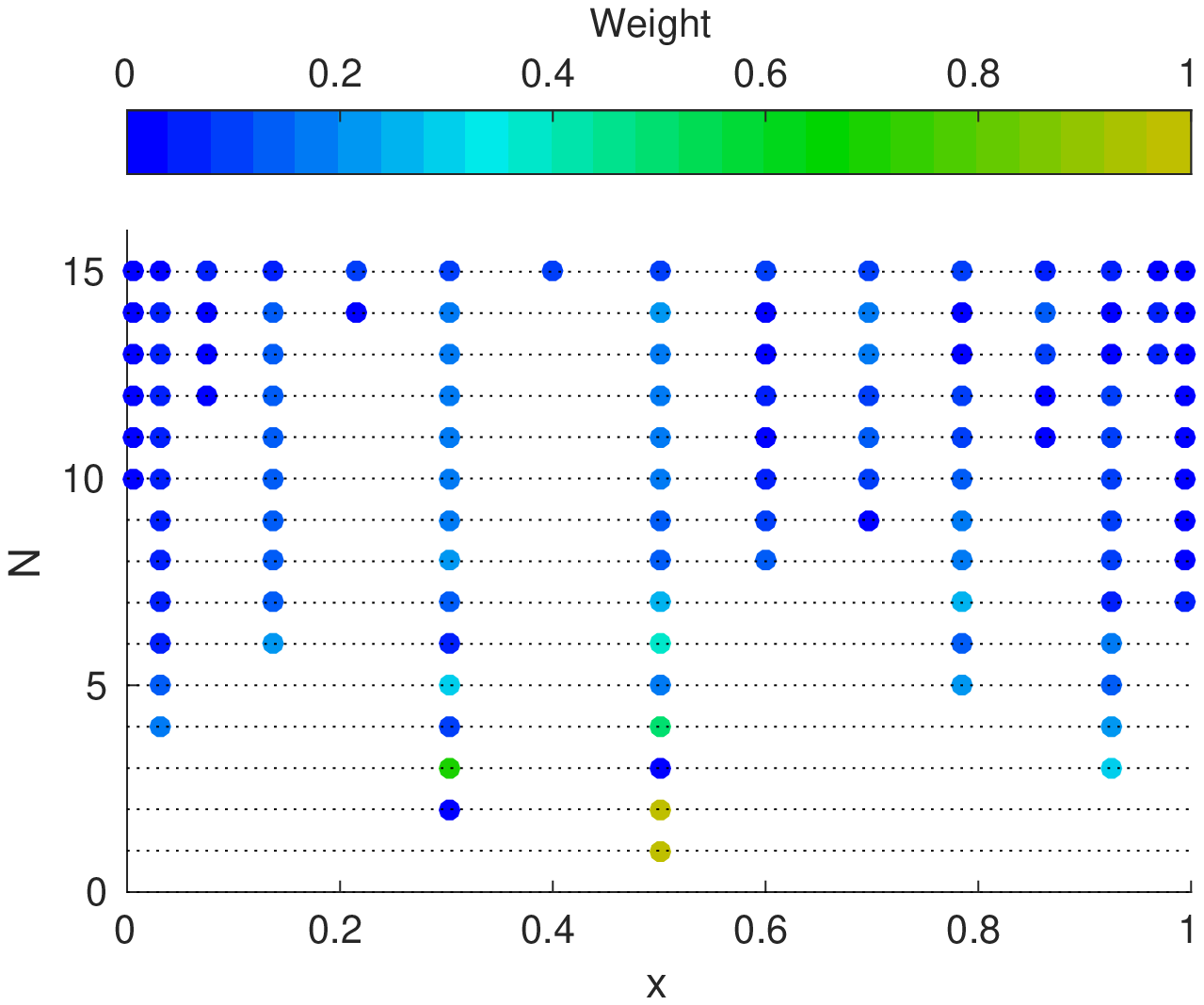}
		\subcaption{Reduced~Gauss--Legendre}
		\label{fig:smolyaknesteduni}
	\end{minipage}
	\caption{Two-dimensional Smolyak cubature rule nodes of the reduced quadrature rule (shown below the sparse grids for various numbers of nodes $N$). All grids consist of 65 nodes.}
	\label{fig:smolyaknested}
\end{figure}

\subsection{Symmetry}
The reduction step does not necessarily keep a symmetric quadrature rule symmetric, because nodes are generally removed one-by-one. This is undesirable and does not happen if the null vector used in the procedure has the same symmetry as the weights, because eliminating one weight then automatically eliminates the symmetric weight.

Such a symmetric null vector always exists, as shown in the following lemma. The key notion is that to keep a symmetric quadrature rule symmetric two nodes have to be removed, which can be implemented by removing two rows from the Vandermonde-matrix (i.e.\ constructing $V_{-2}$ instead of $V_{-1}$). The proof is constructive, providing an algorithm for the reduction.

\begin{lemma}
	\label{lmm:symquad}
	There exists a symmetric null vector of $V_{-2}$.
\end{lemma}
\begin{proof}
	The principle of the proof is to (i) construct a matrix $V'_{-2}$ that encodes the symmetry property, (ii) prove that this matrix is singular and (iii) state a procedure to derive the null vector of this matrix.

	Let $\{\xi_k\}_{k=1}^N$ be the nodes of an $N$-node symmetric quadrature rule.

	A case distinction is made. First, let $N$ be even. Without loss of generality, we assume that $\xi_1 < \xi_2 < \dots < \xi_N$ and that the quadrature rule is symmetric around $0$. Then the nodes can be written as follows:
	\begin{equation}
		\{\xi_1, \xi_2, \dots, \xi_N\} = \{\xi_1, \xi_2, \dots, \xi_{\frac{N}{2}}, -\xi_{\frac{N}{2}}, \dots, -\xi_2, -\xi_1\}.
	\end{equation}
	Consider the following matrix:
	\begin{equation}
		V'_{-2} = \begin{pmatrix}
			\xi_1^0 + \xi_N^0 & \xi_2^0 + \xi_{N-1}^0 & \dots & \xi_{\frac{N}{2}}^0 + \xi_{N - \frac{N}{2} + 1}^0 \\
			\xi_1^1 + \xi_N^1 & \xi_2^1 + \xi_{N-1}^1 & \dots & \xi_{\frac{N}{2}}^1 + \xi_{N - \frac{N}{2} + 1}^1 \\
			\xi_1^2 + \xi_N^2 & \xi_2^2 + \xi_{N-1}^2 & \dots & \xi_{\frac{N}{2}}^2 + \xi_{N - \frac{N}{2} + 1}^2 \\
			\vdots & \vdots & \ddots & \vdots \\
			\xi_1^{N-3} + \xi_N^{N-3} & \xi_2^{N-3} + \xi_{N-1}^{N-3} & \dots & \xi_{\frac{N}{2}}^{N-3} + \xi_{N - \frac{N}{2} + 1}^{N-3}
		\end{pmatrix}.
	\end{equation}
	The matrix $V'_{-2}$ is constructed by combining columns of $V_{-2}$. Each column of $V_{-2}$ is used exactly once. $V'_{-2}$ is not a square matrix, so it is not trivial to see that a non-trivial null vector exists. However, if a null vector $\mathbf{c}'$ exists, it can easily be transformed into a symmetric null vector of $V_{-2}$ using $\mathbf{c} = \trans{(c'_1, c'_2, \dots, c'_{\frac{N}{2}}, c'_{\frac{N}{2}}, \dots, c'_2, c'_1)}$.

	There always exists such a null vector $\mathbf{c}'$ because for $p$ odd, it is true that
	\begin{equation}
		\xi_k^p + \xi_{N-k+1}^p = \xi_k^p + (-1)^p \xi_k^p = 0.
	\end{equation}
	Therefore, determining a null vector of $V'_{-2}$ is equivalent to determining a null vector of the following matrix:
	\begin{equation}
		A_{-2} = \begin{pmatrix}
			2 \xi_1^0 & 2 \xi_2^0 & \dots & 2 \xi_{\frac{N}{2}}^0 \\
			2 \xi_1^2 & 2 \xi_2^2 & \dots & 2 \xi_{\frac{N}{2}}^2 \\
			\vdots & \vdots & \ddots & \vdots \\
			2 \xi_1^{N-4} & 2 \xi_2^{N-4} & \dots & 2 \xi_{\frac{N}{2}}^{N-4}
		\end{pmatrix}.
	\end{equation}
	Here, the rows of $V_{-2}$ consisting of zeros are removed and it is used that $\xi_1^2 = \xi_N^2$, $\xi_2^2 = \xi_{N-1}^2$, etc. $A_{-2}$ is an $\left(\frac{N}{2}-1\right) \times \frac{N}{2}$-matrix, which is singular, hence always has a non-trivial null vector.

	If $N$ is odd, the same principle can be applied with the nodes:
	\begin{equation}
		\{\xi_1, \xi_2, \dots, \xi_N\} = \{ \xi_1, \xi_2, \dots, \xi_{\left\lfloor \frac{N}{2} \right\rfloor},0, -\xi_{\left\lfloor \frac{N}{2} \right\rfloor}, \dots, -\xi_1 \}.
	\end{equation}
	Therefore, after constructing $V'_{-2}$ a null vector needs to be determined of the following matrix:
	\begin{equation}
		A_{-2} = \begin{pmatrix}
			2 \xi_1^0 & 2 \xi_2^0 & \dots & 2 \xi_{\left\lfloor \frac{N}{2} \right\rfloor}^0 & 1 \\
			2 \xi_1^2 & 2 \xi_2^2 & \dots & 2 \xi_{\left\lfloor \frac{N}{2} \right\rfloor}^2 & 0 \\
			\vdots & \vdots & \ddots & \vdots & \vdots \\
			2 \xi_1^{N-3} & 2 \xi_2^{N-3} & \dots & 2 \xi_{\left\lfloor \frac{N}{2} \right\rfloor}^{N-3} & 0
		\end{pmatrix}.
	\end{equation}
	In this case, $A_{-2}$ is an $\left\lfloor \frac{N}{2} \right\rfloor \times (\left\lfloor \frac{N}{2} \right\rfloor + 1)$-matrix, which is again singular.

	Concluding, in both cases ($N$ odd or even) there exists a symmetric null vector of $V_{-2}$, that can be constructed using a null vector of matrix $A_{-2}$.
\end{proof}

The previous lemma leads to the following main theorem about reduced \emph{symmetric} quadrature rules with positive weights. Note that it is important to remove two nodes each time in both cases ($N$ odd or even), because only then it is possible to start with a quadrature rule of odd length and iteratively remove nodes until a quadrature rule of only the middle node is obtained.

\begin{theorem}
	Let $\{\xi_1, \dots, \xi_N\}$ form an $N$-node symmetric quadrature rule with positive weights of degree $N-1$. Then there exist $\xi_i$ and $\xi_j$ with $i \neq j$ such that $\{\xi_1, \dots, \xi_N\} \setminus \{\xi_i, \xi_j\}$ forms an $(N-2)$-node symmetric quadrature rule with positive weights of degree $N-3$.
\end{theorem}
\begin{proof}
	The proof follows directly from Lemma~\ref{lmm:symquad}. Let $A_{-2}$ as in Lemma~\ref{lmm:symquad} be given and let $\mathbf{c}$ be a null vector of $A_{-2}$. There are two cases:
	\begin{itemize}
		\item If $N$ is even: both $\mathbf{c}$ and $-\mathbf{c}$ yield the removal of two nodes in the reduction step.
		\item If $N$ is odd: it is possible that either $\mathbf{c}$ or $-\mathbf{c}$ yields the removal of the middle node, which would result in the removal of just one node. However, either $\mathbf{c}$ or $-\mathbf{c}$ yields the removal of two nodes.
	\end{itemize}
	In both cases, pick $i = k_0$ and $j = N-k_0+1$, where $k_0$ is from the proof of Carath\'eodory's theorem.
\end{proof}

From this theorem and the case distinction between quadrature rules of even and odd length, an algorithm can be formulated that generates the nested quadrature rule keeping weights positive and a symmetric quadrature rule symmetric (see Algorithm~\ref{alg:redquadrule}). Although the matrices $V_{-1}$ and $A_{-2}$ can become ill-conditioned for large $N$, we did not observe any numerical issues in determining null vectors of these matrices for $N$ up to $2^{10}$.

\begin{algorithm}[h]
\caption{Determining the reduced quadrature rule}
\label{alg:redquadrule}
\begin{algorithmic}[1]
\Require {Quadrature rule nodes $\{\xi_1, \xi_2, \dots, \xi_N\}$ and weights $\{w_1, w_2, \dots, w_N\}$ of degree $N-1$}
\Ensure {Non-negative weights $\{w_1^*, w_2^*, \dots, w_N^*\}$ having \emph{either} two weights equal to 0 if the original quadrature rule is symmetric \emph{or} one weight equal to 0 otherwise. Using these weights, the quadrature rule has either degree $N-3$ or $N-2$ respectively.}
\Statex ~
\If {quadrature rule $\{\xi_1, \xi_2, \dots, \xi_N\}$ is symmetric}
\State Construct $A_{-2}$ from Lemma~\ref{lmm:symquad}
\State Determine a null vector $\mathbf{c}^*$ of $A_{-2}$
\State Using $\mathbf{c}^*$, construct a symmetric null vector $\mathbf{c}$ of the matrix $V_{-2}$, where $V$ is the Vandermonde-matrix.
\Else
\State Construct the matrix $V_{-1}$, where $V$ is the Vandermonde-matrix.
\State Determine a null vector $\mathbf{c}$ of $V_{-1}$
\EndIf
\State $\alpha^{(1)} \gets \min_{k=1, \dots, N} \left\{\frac{w_k}{c_k} : c_k > 0\right\}$
\State $\alpha^{(2)} \gets \max_{k=1, \dots, N} \left\{-\frac{w_k}{c_k} : c_k < 0\right\}$
\State $w^{(1)}_k \gets w_k - \alpha^{(1)} c_k$ and $w^{(2)}_k \gets w_k + \alpha^{(2)} c_k$ for $k = 1, \dots, N$.
\Statex ~
\State $N_Z^{(1)} \gets \#\left\{w^{(1)}_k = 0 \mid k = 1, \dots, N\right\}$ 
\State $N_Z^{(2)} \gets \#\left\{w^{(2)}_k = 0 \mid k = 1, \dots, N\right\}$
\If {$N_Z^{(1)} = 1$ and $N_Z^{(2)} = 2$}
\State \textbf{return} $\{w^{(2)}_k\}$
\ElsIf {$N_Z^{(2)} = 1$ and $N_Z^{(1)} = 2$}
\State \textbf{return} $\{w^{(1)}_k\}$
\Else
\State \emph{Here, a selection criterion can be applied:}
\State \textbf{return} either $\{w^{(1)}_k\}$ or $\{w^{(2)}_k\}$
\EndIf
\end{algorithmic}
\end{algorithm}

Using this algorithm, symmetric reduced quadrature rules can be generated using the prior criterion. The resulting Smolyak cubature rules are therefore also symmetric (see Figure~\ref{fig:smolyaknestedsym} for examples). The symmetry is also clearly visible in the plots of the quadrature rules.

\begin{figure}
	\centering
	\begin{minipage}{.3\textwidth}
		\includegraphics[width=\textwidth]{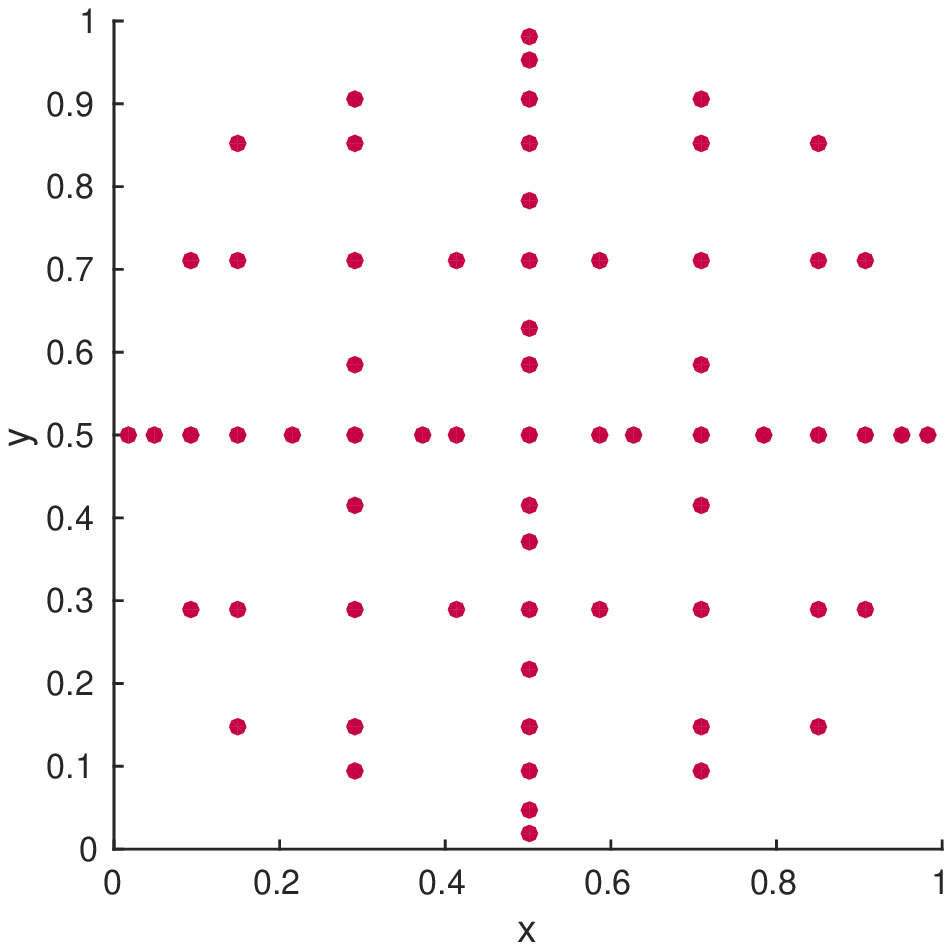}
		\includegraphics[width=\textwidth]{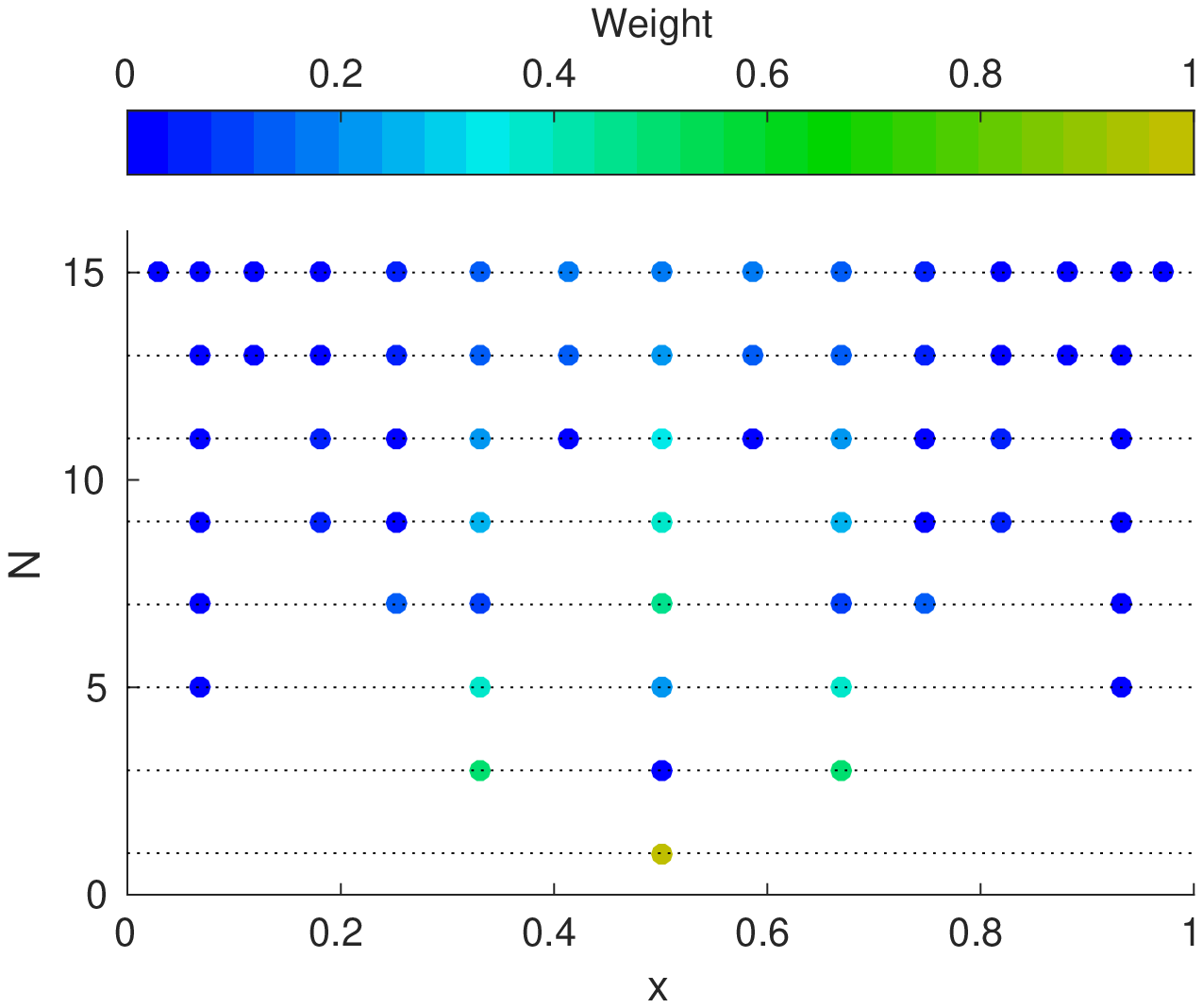}
		\subcaption{Reduced~Gauss--Jacobi}
		\label{fig:smolyaknestedsymbeta23}
	\end{minipage}
	\begin{minipage}{.3\textwidth}
		\includegraphics[width=\textwidth]{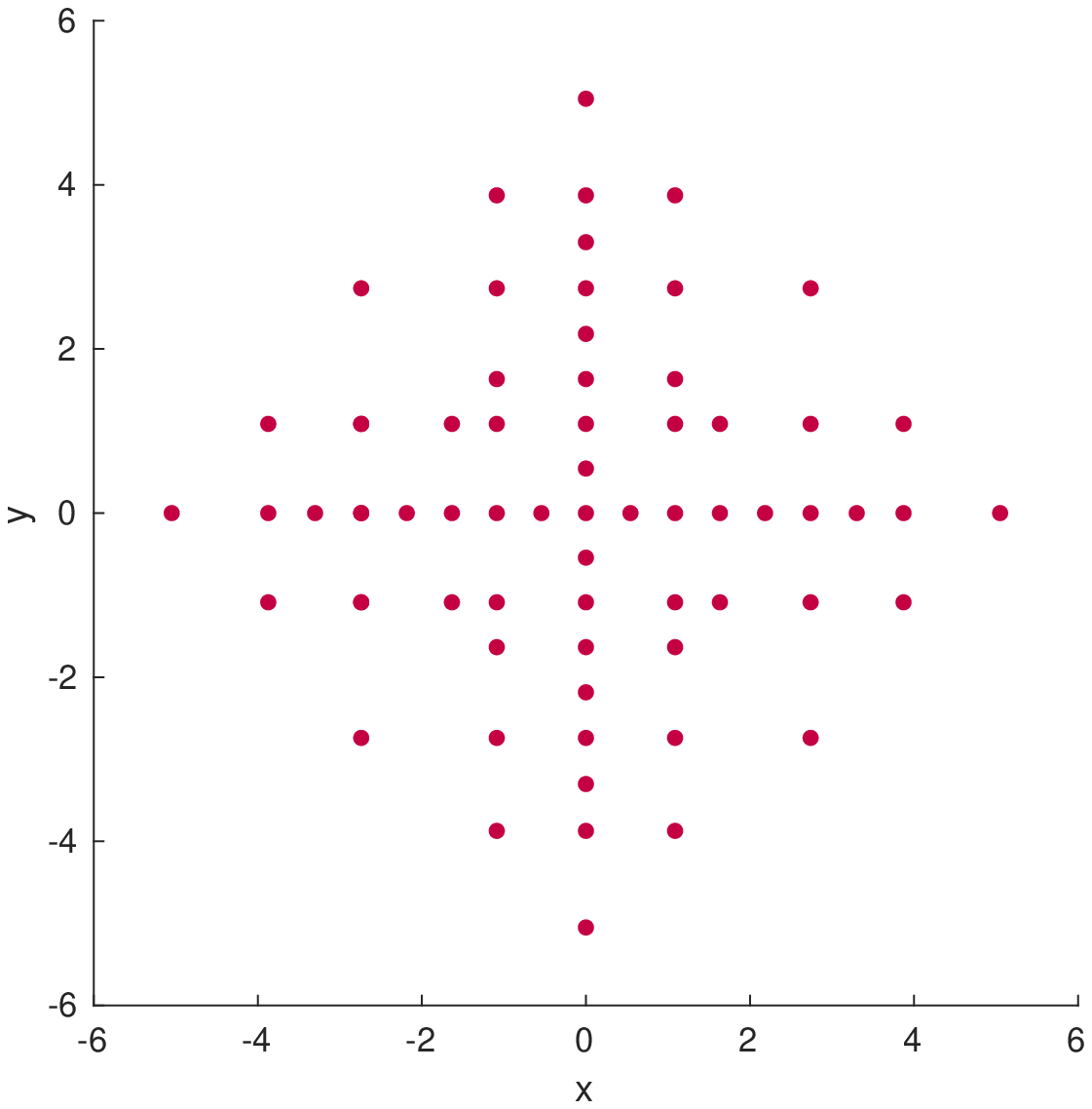}
		\includegraphics[width=\textwidth]{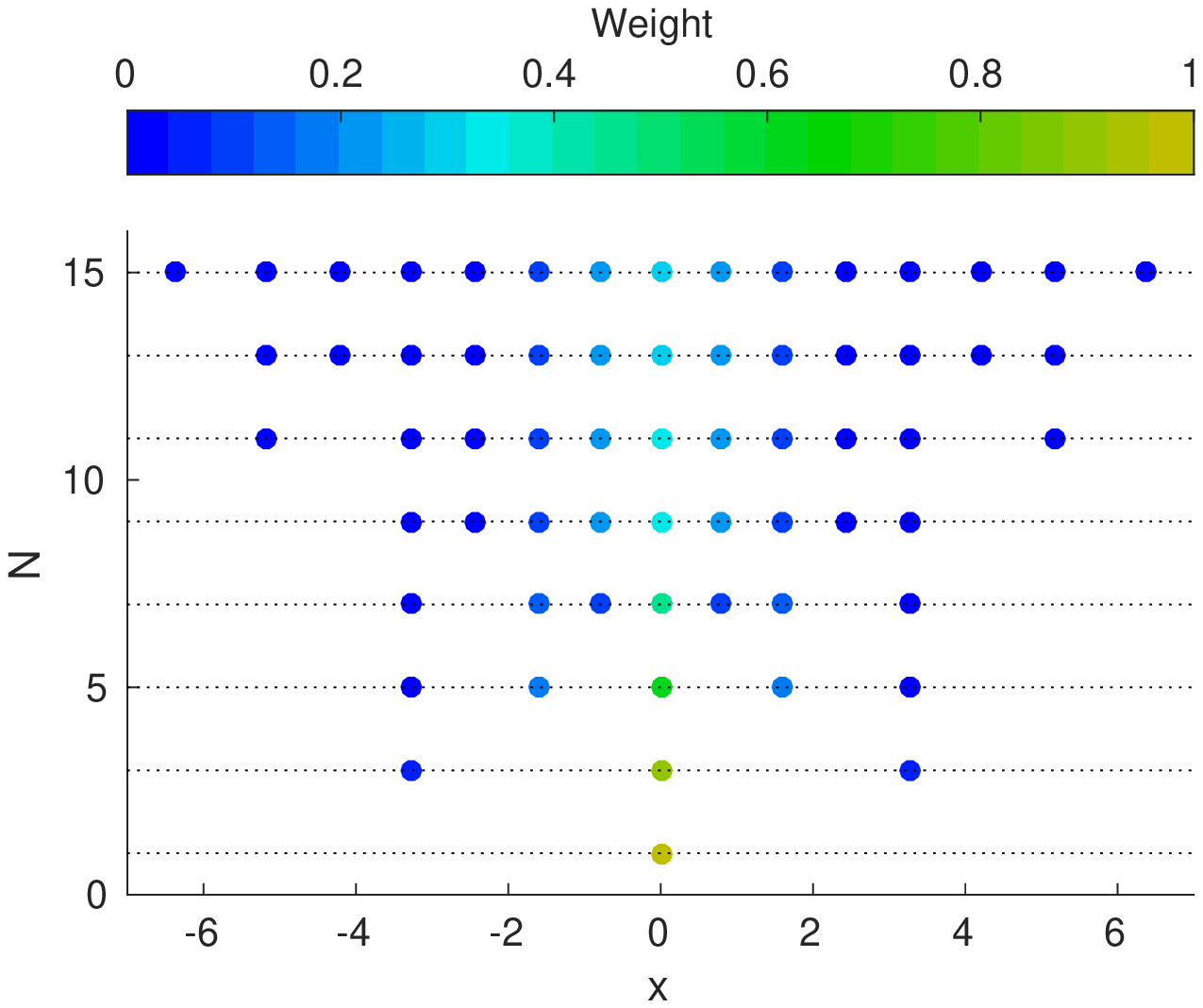}
		\subcaption{Reduced~Gauss--Hermite}
		\label{fig:smolyaknestedsymnorm01}
	\end{minipage}
	\begin{minipage}{.3\textwidth}
		\includegraphics[width=\textwidth]{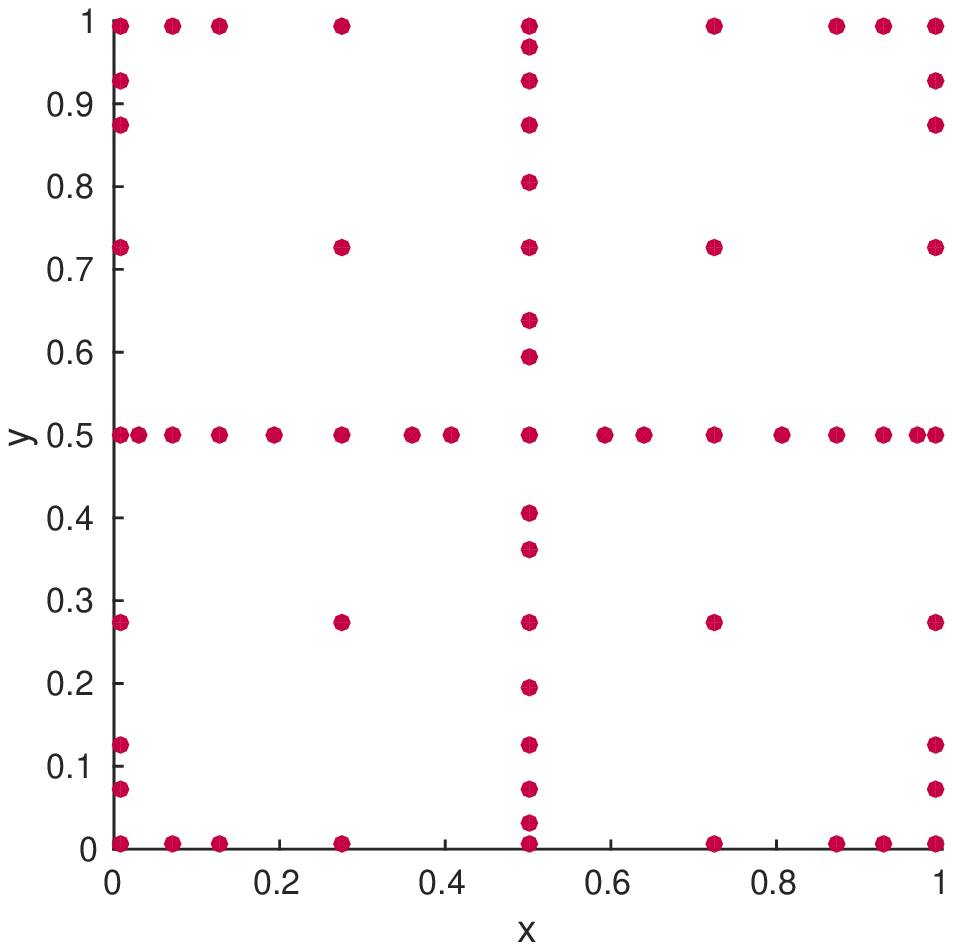}
		\includegraphics[width=\textwidth]{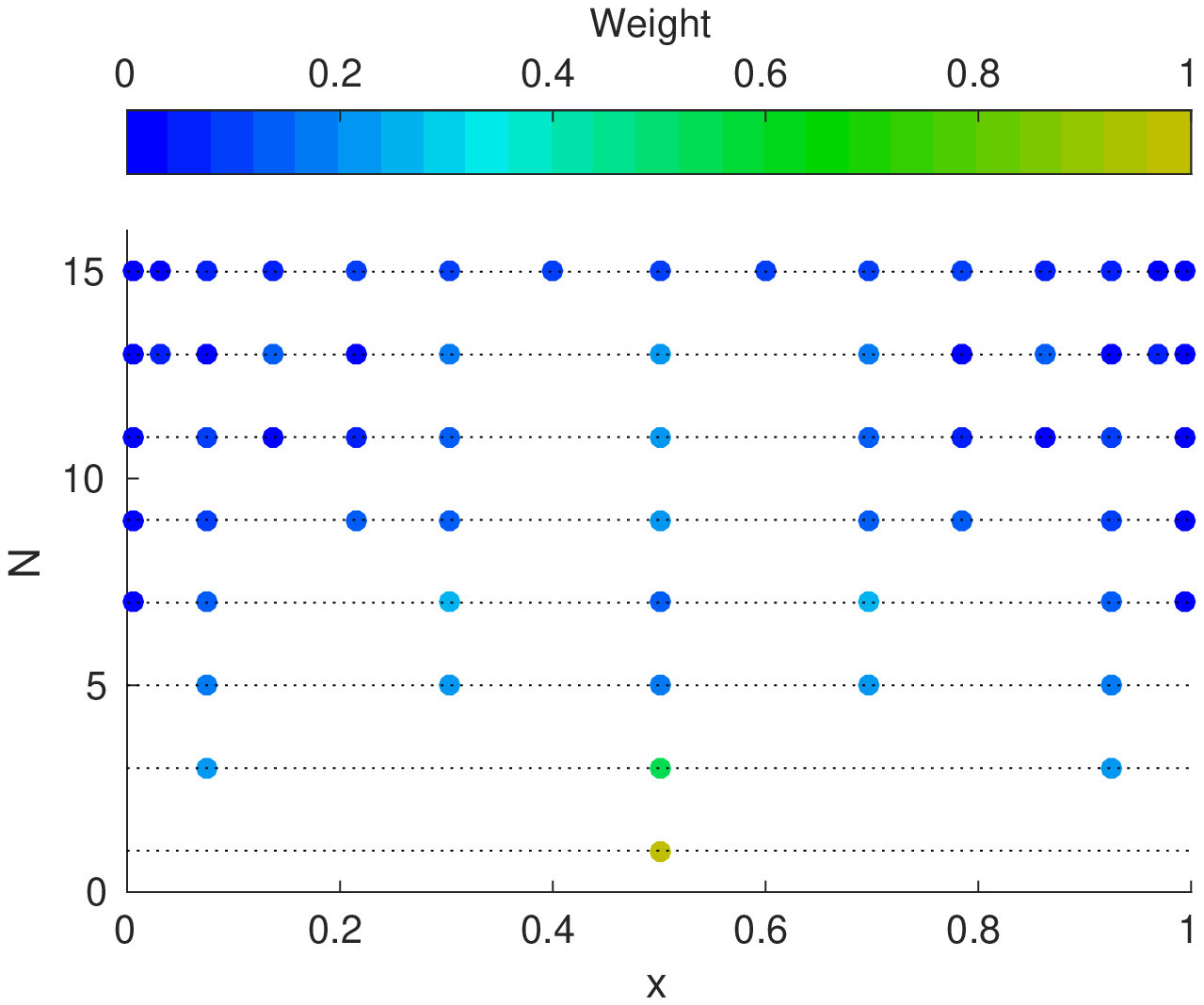}
		\subcaption{Reduced~Gauss--Legendre}
		\label{fig:smolyaknestedsymuni}
	\end{minipage}
	\caption{Two-dimensional Smolyak cubature rule nodes of the symmetric reduced quadrature rule (shown below the sparse grids for various numbers of nodes $N$). All sparse grids consist of 65 nodes.}
	\label{fig:smolyaknestedsym}
\end{figure}

\section{Reduced cubature rules}
\label{sec:redcubrule}
In the previous section, a procedure has been outlined to generate a set of symmetric and nested quadrature rules with positive weights. Exactly the same principles can be applied to cubature rules, i.e.\ in a multi-dimensional setting.

The set-up is the same as in the previous section. First, the reduction step is introduced ignoring symmetry. This extension is straightforward. Then symmetries in multi-dimensional spaces are studied and a similar theory as in the one-dimensional case is developed regarding the symmetry of nested cubature rules.

\subsection{Multi-dimensional reduction step}
\label{subsec:redstep}
Let $\{\boldsymbol \xi_1, \dots, \boldsymbol \xi_N\}$ and $\{w_1, \dots, w_N\}$ be $N$ cubature nodes and positive weights in a $d$-dimensional space forming a cubature rule of degree $K$, with $N = \dim \mathbb{P}(K, d)$. Let $G$ be the $N \times N$ generalized Vandermonde-matrix, as introduced previously (see~\eqref{eq:genvandermondesys}). The goal is to determine a subset of nodes that forms a cubature rule of degree $K-1$, with positive weights. Such a cubature rule has a generalized Vandermonde-matrix of size $\dim \mathbb{P}(K-1, d) \times \dim \mathbb{P}(K-1, d)$.

The one-dimensional reduction step can be easily generalized to determine this nested cubature rule as follows. First, let $G$ be the generalized Vandermonde matrix again. Then $G_{-C}$ with $C = \dim \mathbb{P}(K, d) - \dim \mathbb{P}(K-1, d)$ has a $C$-dimensional null space\footnotemark.
\footnotetext{Recall that $A_{-k}$ is matrix $A$ without its last $k$ rows.}%
Applying Carath\'eodory's theorem iteratively to this matrix allows for the removal of $C$ columns, which yields the pursued $(\dim \mathbb{P}(K-1, d))$-node cubature rule with positive weights. 

Just as in the one-dimensional case, the choice of basis for the null space is not unique, and since the number of null vectors is larger in multiple dimensions, more freedom to select the node to be removed is available. Iteratively applying the prior criterion is again an option, and will again result in loss of symmetry.

In multi-dimensional spaces, a symmetric $K$-degree cubature rule with positive weights of $\dim \mathbb{P}(K, d)$ nodes is not trivial to derive. However, a good initial cubature rule can be determined using the introduced reduction step. Starting with a $K$-degree tensor product rule, nodes can be removed from this rule until $\dim \mathbb{P}(K, d)$ nodes are left and using this cubature rule, a set of nested cubature rules can be generated.

\subsection{Symmetries}
As in the one-dimensional case, the reduction step does not keep a symmetric cubature rule symmetric. In a multi-dimensional space, many different types of symmetries can be considered. We consider two types of reflectional symmetry:
\begin{enumerate}
	\item Symmetry along an axis, i.e.\ the plane of symmetry has the property $x^{(k)} = 0$, where $x^{(k)}$ is a coordinate. If a cubature rule is symmetric in this way in all dimensions, the planes of symmetry divide the space into $2^d$ orthants (multi-dimensional quadrants). We call this a \emph{type-1} symmetry (see Figure~\ref{fig:lmm:main1} for a sketch).
	\item Symmetry along a plane having $x^{(k)} = x^{(j)}$, where $x^{(k)}$ and $x^{(j)}$ are two coordinates. If a cubature rule is symmetric in this way in all dimensions, the planes of symmetry divide the space into $2^d$ orthants after a rotation over $\frac{1}{4} \pi$ of the complete basis. We call this a \emph{type-2} symmetry (see Figure~\ref{fig:lmm:main2} for a sketch).
\end{enumerate}
In a tensor product cubature rule, the first symmetry occurs if the rule is generated using a symmetric quadrature rule. The second symmetry occurs if one quadrature rule is used multiple times in several dimensions.

To preserve symmetry after the removal of nodes, the null vector used to remove the nodes must have the same symmetry. As in the one-dimensional case, it is not guaranteed that such a null vector exists. Under certain conditions such a null vector does exist. The theory is more cumbersome than in the one-dimensional case, but has the same general structure: to determine a symmetric null vector of $G_{-C}$, a matrix $G'$ is constructed and a proof is given that a null vector of $G'$ can be transformed into a null vector of $G_{-C}$. Dependencies in the row space of $G'$ finish the proof.

\begin{figure}
	\centering
	\begin{minipage}{.4\textwidth}
		\includegraphics[width=\textwidth]{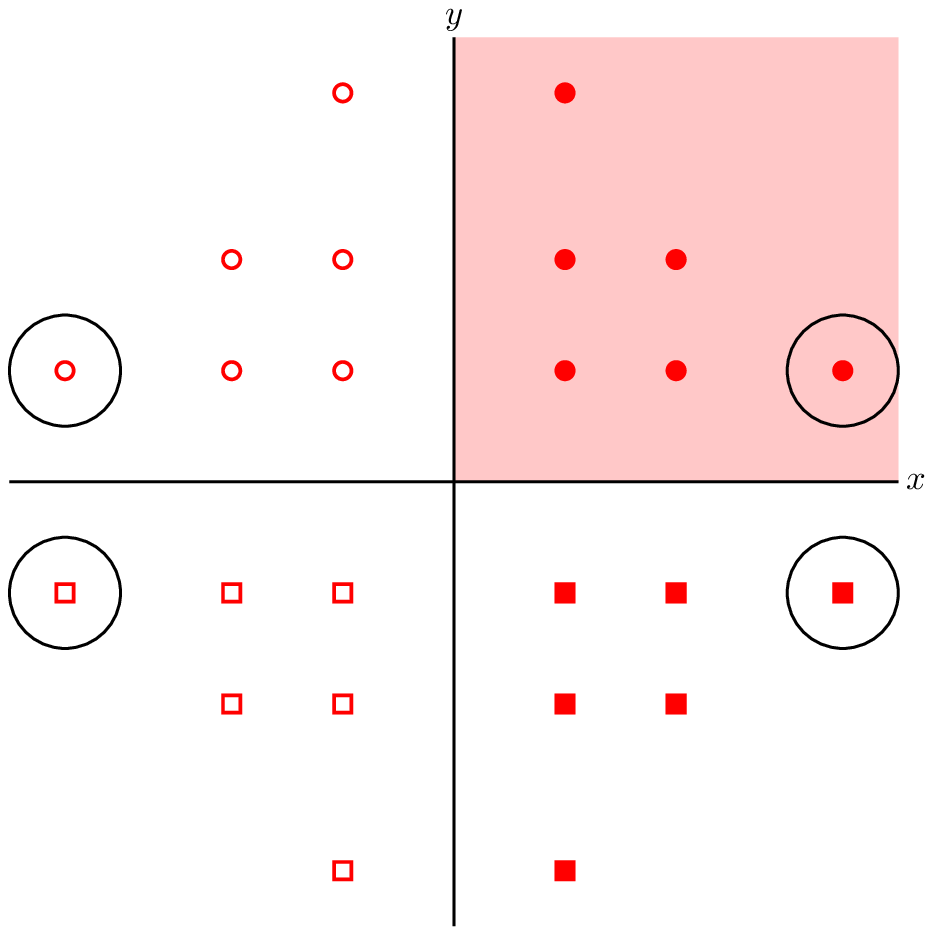}
		\subcaption{Type-1 symmetry}
		\label{fig:lmm:main1}
	\end{minipage}
	~
	\begin{minipage}{.4\textwidth}
		\includegraphics[width=\textwidth]{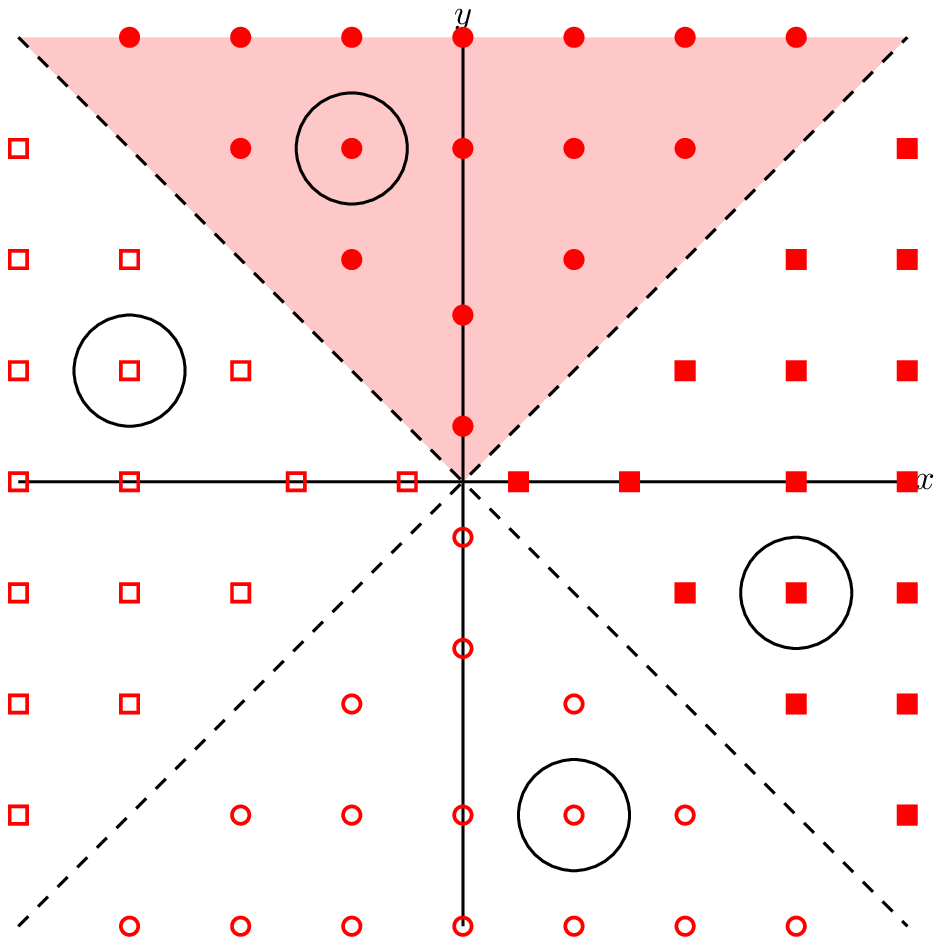}
		\subcaption{Type-2 symmetry}
		\label{fig:lmm:main2}
	\end{minipage}
	\caption{Visual proofs of Theorems~\ref{thm:main1} and \ref{thm:main2}. In both cases, removing a node in a pink region, which is an orthant or an orthant after $\frac{1}{4} \pi$ rotation, results into the removal of 4 nodes in total to keep the cubature rule symmetric. The number of nodes in the pink region is denoted by $N_{\bar{Q}}$ and the number of independent rows in $G_{-C}$ determines the number of nodes that can be removed.}
\end{figure}

\subsubsection{Type-1 symmetry}
First we demonstrate the existence of a suitable null vector. The proof has the same structure as the proof of Lemma~\ref{lmm:symquad}.

\begin{lemma}
	\label{lmm:main1}
	Let $\{\boldsymbol \xi_1, \dots, \boldsymbol \xi_N\}$ be a type-1 symmetric cubature rule of degree $K$ with positive weights $\{w_1, \dots, w_N\}$. Let $Q$ be an orthant and let $N_{\bar{Q}}$ be the number of cubature nodes in $\bar{Q}$. Then there exists a symmetric null vector of $G_{-C}$ if
	\begin{equation}
		\binom{\lfloor \frac{K}{2} \rfloor + d}{d} < N_{\bar{Q}}.
	\end{equation}
\end{lemma}
\begin{proof}
	See~\ref{app:proofs}.
\end{proof}

Using this lemma, a theorem can be stated about nested type-1 symmetric cubature rules with positive weights.

\begin{theorem}
	\label{thm:main1}
	Let $\{\boldsymbol \xi_1, \dots, \boldsymbol \xi_N\}$ be a type-1 symmetric cubature rule of degree $K$ with positive weights. Assume there are no cubature nodes shared between orthants (i.e., on the plane of symmetry). Then there exist $I = 2^d \binom{\lfloor\frac{K-1}{2}\rfloor + d}{d}$ indices $i_1, i_2, \dots, i_I$ such that $\{\boldsymbol \xi_{i_1}, \boldsymbol \xi_{i_2}, \dots, \boldsymbol \xi_{i_I}\}$ forms a type-1 symmetric cubature rule of degree $K-1$ with positive weights.
\end{theorem}
\begin{proof}
	For a visual proof in two dimensions, see Figure~\ref{fig:lmm:main1}. Let $N_Q$ be the number of nodes in an orthant $Q$. Because no nodes are shared between orthants, it is true that
	\begin{equation}
		N_Q = N_{\bar{Q}},
	\end{equation}
	where $N_{\bar{Q}}$ is the number of nodes in $\bar{Q}$. Therefore the total number of nodes of the cubature rule equals $N_Q 2^d$. From Lemma~\ref{lmm:main1} it is known that the number of nodes that can be removed from one orthant equals
	\begin{equation}
		N_Q - \binom{\lfloor\frac{K-1}{2}\rfloor + d}{d}.
	\end{equation}
	Hence, the number of nodes \emph{remaining} in the orthant equals
	\begin{equation}
		\binom{\lfloor\frac{K-1}{2}\rfloor + d}{d}.
	\end{equation}
	So, the total number of nodes after all removal steps is $\binom{\lfloor\frac{K-1}{2}\rfloor + d}{d} 2^d$.
\end{proof}

If there are cubature rule nodes shared between orthants (which is almost always the case), then $I$ is an upper bound of the number of nodes after a removal procedure.

\subsubsection{Type-2 symmetry}
Again, we state a lemma about the existence of a null vector. And again, the proof has the same structure as the proof of Lemma~\ref{lmm:symquad}. However, the number of independent rows cannot be deduced explicitly anymore, such that the following lemma is necessary.

\begin{lemma}
	\label{lmm:bound1inc}
	Let $\mathbf{s} = (s_1, \dots, s_d) \in \mathbb{N}_+^d$ be a sequence. If
	\begin{itemize}
		\item $\|\mathbf{s}\|_1 \leq B$, where $B > 0$ and $B \in \mathbb{N}$,
		\item $\mathbf{s}$ is weakly increasing, i.e.\ $s_1 \leq s_2 \leq \dots$
	\end{itemize}
	then there exist $1 + \sum_{l=1}^B p_d(l)$ such sequences, where $p$ is the restricted partition function\footnotemark.
	\footnotetext{There are several definitions of the restricted partition number. Here, it is the number of compositions of the number $l$ with at most $d$ summands.}
\end{lemma}

\begin{lemma}
	\label{lmm:main2}
	Let $\{\boldsymbol \xi_1, \dots, \boldsymbol \xi_N\}$ be a type-2 symmetric cubature rule of degree $K$ with positive weights $\{w_1, \dots, w_N\}$. Let $Q$ be an orthant after a rotation over $\frac{1}{4} \pi$ of all axes. Let $N_{\bar{Q}}$ be the number of cubature nodes in $\bar{Q}$. Then there exists a symmetric null vector of $G_{-C}$ if
	\begin{equation}
		1 + \sum_{l=1}^K p_d(l) < N_{\bar{Q}},
	\end{equation}
	where $p_d(l)$ is the restricted partition function.
\end{lemma}
\begin{proof}
	See \ref{app:proofs}.
\end{proof}

A similar theorem can be developed about the nested cubature rule in this case.

\begin{theorem}
	\label{thm:main2}
	Let $\{\boldsymbol \xi_1, \dots, \boldsymbol \xi_N\}$ be a type-2 symmetric cubature rule of degree $K$ with positive weights. Assume there are no cubature nodes shared between orthants after a rotation over $\frac{1}{4} \pi$ (i.e., on the plane of symmetry). Then there exist $I = 2^d \left(1 + \sum_{l=1}^{K-1} p_d(l)\right)$ indices $i_1, i_2, \dots, i_I$ such that $\{\boldsymbol \xi_{i_1}, \boldsymbol \xi_{i_2}, \dots, \boldsymbol \xi_{i_I}\}$ forms a type-2 symmetric cubature rule of degree $K-1$ with positive weights.
\end{theorem}
\begin{proof}
	Combine the proof of Theorem~\ref{thm:main1} with Lemma~\ref{lmm:main2}. For a visual proof in two dimensions, see Figure~\ref{fig:lmm:main2}.
\end{proof}

If there are nodes shared between orthants, the theorem provides an upper bound of the number of nodes. The two lemmas can be combined into the following corollary. The resulting theorem has the same structure as the two theorems above and is therefore omitted.

\begin{corollary}
	\label{cor:main12}
	Let $\{\boldsymbol \xi_1, \dots, \boldsymbol \xi_N\}$ be a type-1 and type-2 symmetric cubature rule of degree $K$ with positive weights $\{w_1, \dots, w_N\}$. Let $Q_1$ be an orthant and let $Q_2$ be an orthant after a rotation over $\frac{1}{4}\pi$ of all axes. Let $N_{\bar{Q_1} \cap \bar{Q_2}} \eqqcolon N_{\bar{Q}}$ be the number of nodes in both $\bar{Q_1}$ and $\bar{Q_2}$. Then there exists a symmetric null vector of $G_{-C}$ if
	\begin{equation}
		1 + \sum_{l=1}^{\lfloor K / 2 \rfloor} p_d(l) < N_{\bar{Q}},
	\end{equation}
	where $p_d(l)$ is the restricted partition function.
\end{corollary}
\begin{proof}
	Combine Lemma~\ref{lmm:main1}~and~\ref{lmm:main2}.
\end{proof}

\subsection{Reduced cubature rule}
\label{subsec:redcubrule}
With the construction above, three different reduced cubature rules can be considered:
\begin{enumerate}
	\item The \emph{reduced cubature rule}, which is a set of nested cubature rules with positive weights (but no symmetry). This set can be generated by repeatedly applying the reduction step and is the multi-dimensional extension of the \emph{reduced quadrature rule}. See Figure~\ref{fig:nestedex1} for an example.
	\item The \emph{symmetric reduced cubature rule}, which is a set of nested cubature rules with positive weights, incorporating symmetry. This set can be generated by applying reduction steps with the theories deduced above. This is the multi-dimensional extension of the \emph{symmetric reduced quadrature rule}. See Figure~\ref{fig:nestedex2} for an example.
	\item The \emph{negative symmetric reduced cubature rule}, which is a set of nested cubature rules incorporating symmetry, but having possibly multiple negative weights. Although the weights are not absolutely bounded, this cubature rule yields a very small number of nodes. The set can be generated using the theories above trying to remove as many nodes as possible in each step. In the one-dimensional case, this rule was not studied because the number of nodes was not relevant. See Figure~\ref{fig:nestedex3} for an example. In the figure it is clearly visible that nodes on the boundaries of the orthants of the theorems above are maintained.
\end{enumerate}
The cubature rules are constructed using the multi-dimensional Vandermonde-matrix and incorporating Lemma~\ref{lmm:main1} and \ref{lmm:main2}. Pseudo-code for this is provided in \ref{app:algs}. The algorithm stated there considers a fixed fraction of a tensor grid, whose number of nodes increases rapidly. Therefore determining a null vector of the corresponding generalized Vandermonde-matrix becomes rapidly computationally expensive. The condition number of the matrix depends on the symmetries of the original cubature rule, i.e., the more symmetries there are, the better the condition number is. Creating a general efficient implementation is ongoing research.

If the number of dimensions is not too large ($d \lesssim 5$), then the cubature rule with positive weights has approximately the same number of nodes as the Smolyak sparse grid (see Figure~\ref{fig:nestedexsmol} for an example). For higher dimensions, all cubature rules suffer from the curse of dimensionality. The cubature rule with positive weights has the largest growth in number of nodes. The cubature rule with (some) negative weights has the smallest growth compared with both the Smolyak cubature rule and the cubature rule with (only) positive weights.

Both the cubature rules with positive weights and negative weights are based on the removal of groups of nodes. If all groups would be of equal size, both approaches yield an equal number of nodes. However, nodes on the plane of symmetry belong to smaller groups. The cubature rule with positive weights removes nodes such that the weights remain positive and does not take this into account. The growth of the nodes with respect to the dimension is therefore large. In Figure~\ref{fig:nestedex2} it is clearly visible that nodes on the plane of symmetry are being removed. As opposed to this, the cubature rule with negative weights does not remove these nodes (see Figure~\ref{fig:nestedex3}). Choosing a quadrature rule of odd length results therefore in a smaller cubature rule as there are more groups of nodes with equal weights.

Just as for the Smolyak cubature rule, the number of nodes cannot be deduced analytically but must be tabulated (see Table~\ref{tbl:results}). For higher dimensional cases, the results for the positive reduced cubature rule are omitted due to computational constraints. In the table the differences in growth are clearly visible. In the 5-dimensional case it can be observed that choosing a quadrature rule of odd length yields less nodes. Selecting a larger initial rule can therefore result in a smaller reduced rule. For the same reason there are less results for the 15-, 20-, and 25-dimensional cases: the cases where the original quadrature rule has even length are computationally unfeasible.

\begin{table}
	\centering
	\caption{Number of nodes of several cubature rules for several dimensions ($d$) and several degrees ($K$). $N_\textrm{Positive}$ denotes the number of nodes of the symmetric reduced cubature rule with positive weights, $N_\textrm{Negative}$ denotes the number of nodes of the negative symmetric reduced cubature rule, and $N_\textrm{Smolyak}$ denotes the number of nodes of the smallest Smolyak sparse grid of at least degree $K$.}
	\label{tbl:results}
	\begin{tabular}{r r | r r r r}
		$d$ & $K$ & $\dim \mathbb{P}(K, d)$ & $N_\textrm{Smolyak}$ & $N_\textrm{Positive}$ & $N_\textrm{Negative}$ \\
		\hline
		\hline
		5 & 5 & 252 & 61 & 113 & 43 \\
		5 & 7 & 792 & 241 & 544 & 384 \\
		5 & 9 & 2\,002 & 805 & 1\,313 & 325 \\
		5 & 11 & 4\,368 & 2\,473 & 4\,096 & 2\,016 \\
		5 & 13 & 8\,568 & 7\,245 & 6\,005 & 1\,607 \\
		\hline
		7 & 5 & 792 & 113 & 689 & 99 \\
		7 & 7 & 3\,432 & 589 & 1\,797 & 325 \\
		7 & 9 & 11\,440 & 2\,471 & 19\,717 & 901 \\
		7 & 11 & 31\,824 & 9\,101 & 28\,479 & 2\,863 \\
		7 & 13 & 77\,520 & 30\,907 & 158\,709 & 28\,479 \\
		\hline
		10 & 5 & 3\,003 & 221 & 13\,461 & 201 \\
		10 & 7 & 19\,448 & 1\,581 & 20\,533 & 1\,361 \\
		10 & 9 & 92\,378 & 8\,810 & 1\,368\,449 & 3\,705 \\
		10 & 11 & 352\,716 & 41\,445 & 8\,284\,617 & 12\,489 \\
		10 & 13 & 1\,144\,066 & 172\,055 & 26\,598\,325 & 38\,353 \\
		\hline
		15 & 5 & 15\,504 & 481 & & 451 \\
		15 & 9 & 1\,307\,504 & 40\,001 & & 30\,861 \\
		15 & 13 & 37\,442\,160 & 1\,472\,697 & & 362\,063 \\
		\hline
		20 & 5 & 53\,130 & 841 & & 801 \\
		20 & 9 & 10\,015\,005 & 120\,401 & & 98\,881 \\
		\hline
		25 & 5 & 142\,506 & 1\,301 & & 1\,251 \\
		25 & 9 & 52\,451\,256 & 286\,001 & & 244\,101
	\end{tabular}
\end{table}

\begin{figure}
	\centering
	\begin{minipage}{.45\textwidth}
		\includegraphics[width=\textwidth]{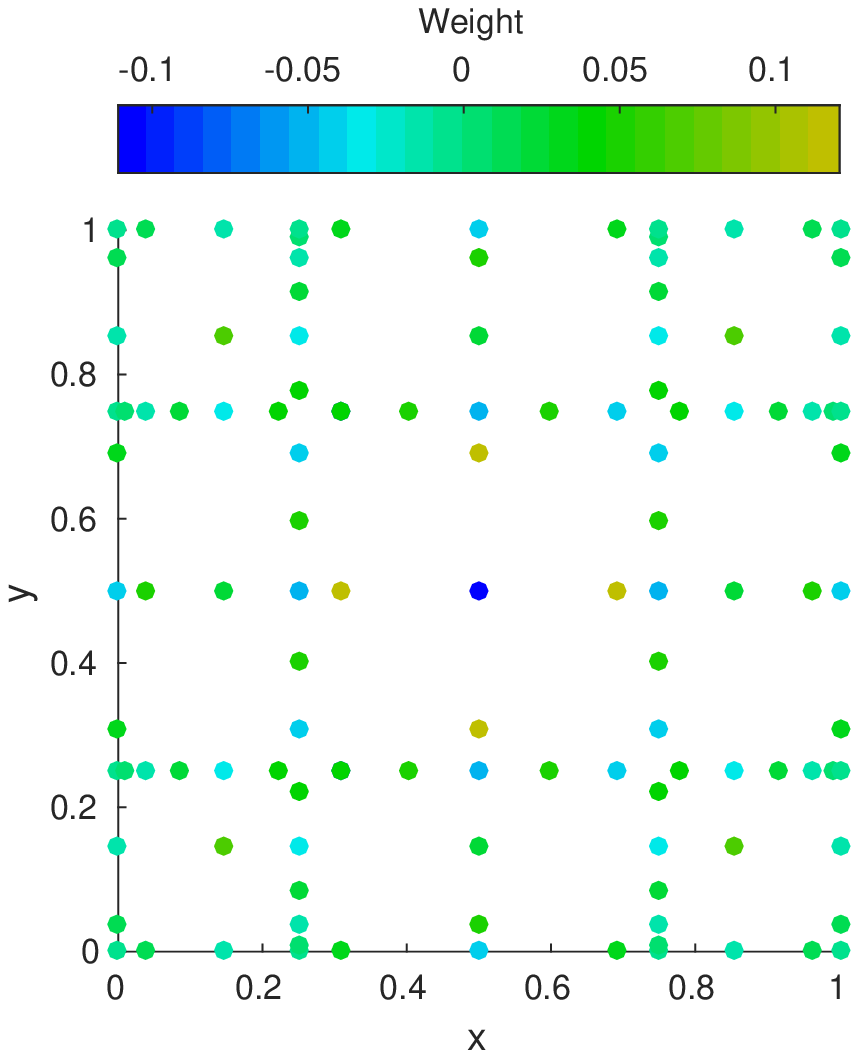}
		\subcaption{Smolyak sparse grid (117~nodes)}
		\label{fig:nestedexsmol}
	\end{minipage}
	\begin{minipage}{.45\textwidth}
		\includegraphics[width=\textwidth]{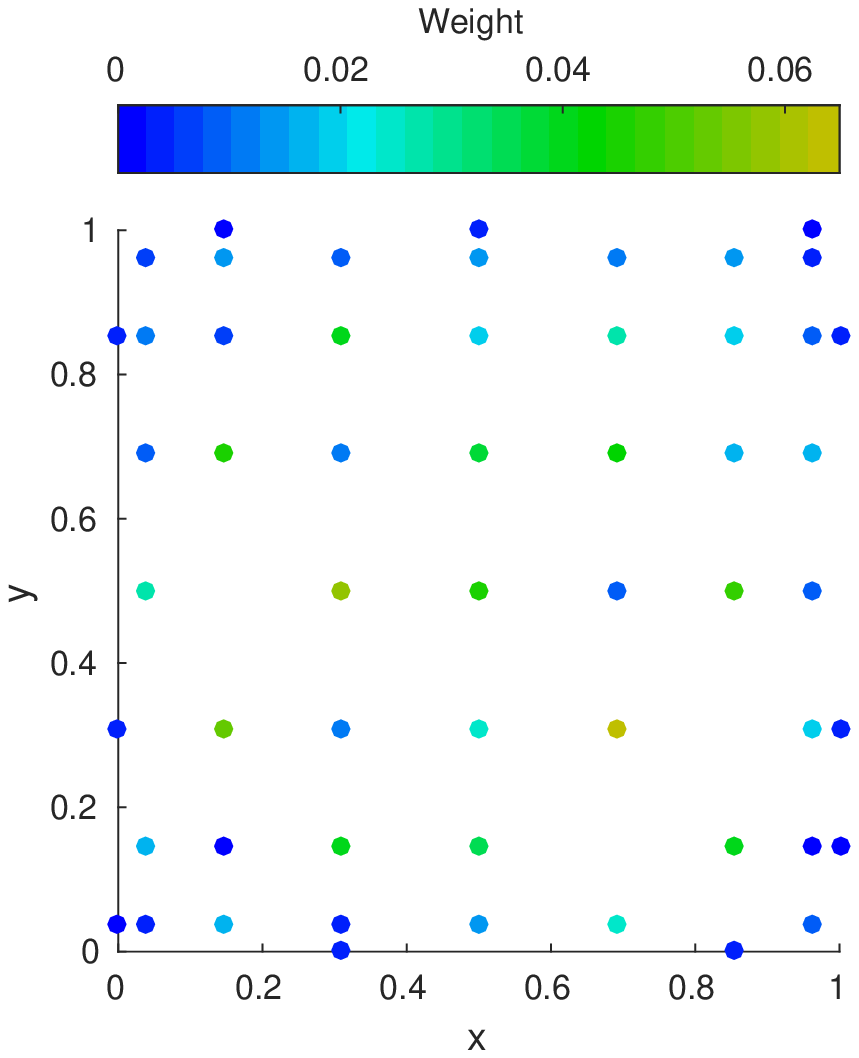}
		\subcaption{Red.\ cubature rule (55~nodes)}
		\label{fig:nestedex1}
	\end{minipage}
	\begin{minipage}{.45\textwidth}
		\includegraphics[width=\textwidth]{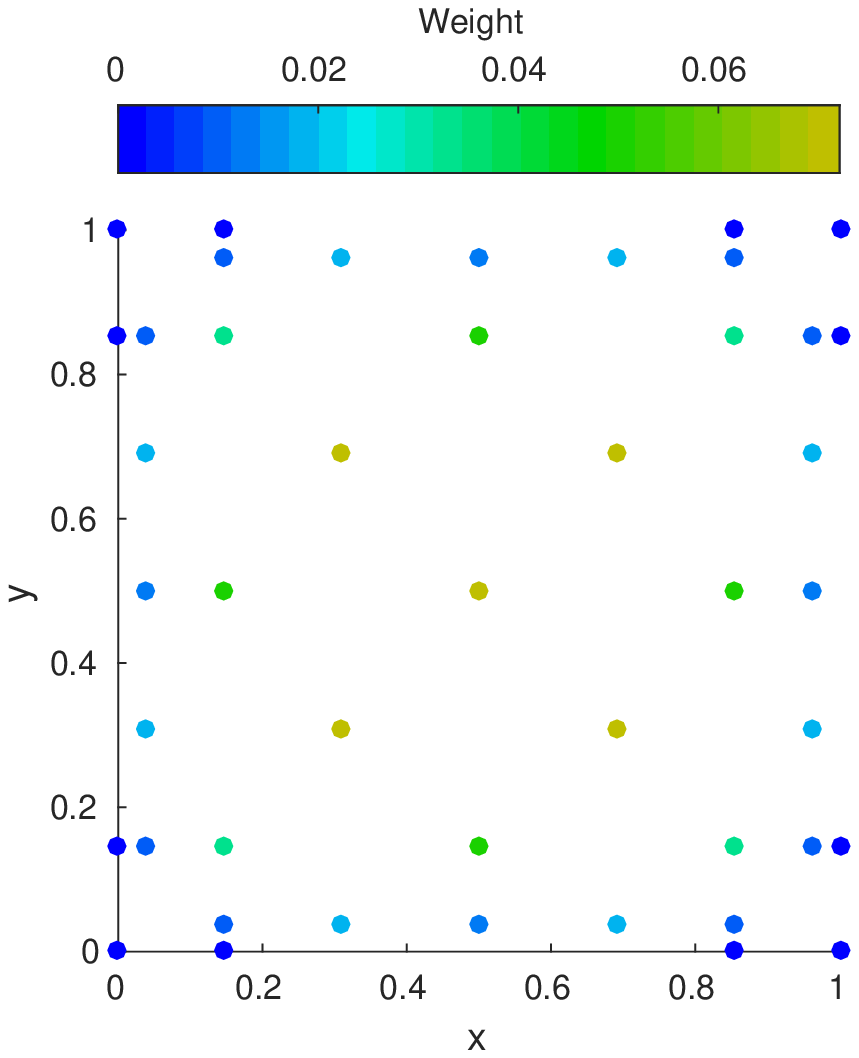}
		\subcaption{Symmetric red.\ cubature rule (45~nodes)}
		\label{fig:nestedex2}
	\end{minipage}
	\begin{minipage}{.45\textwidth}
		\includegraphics[width=\textwidth]{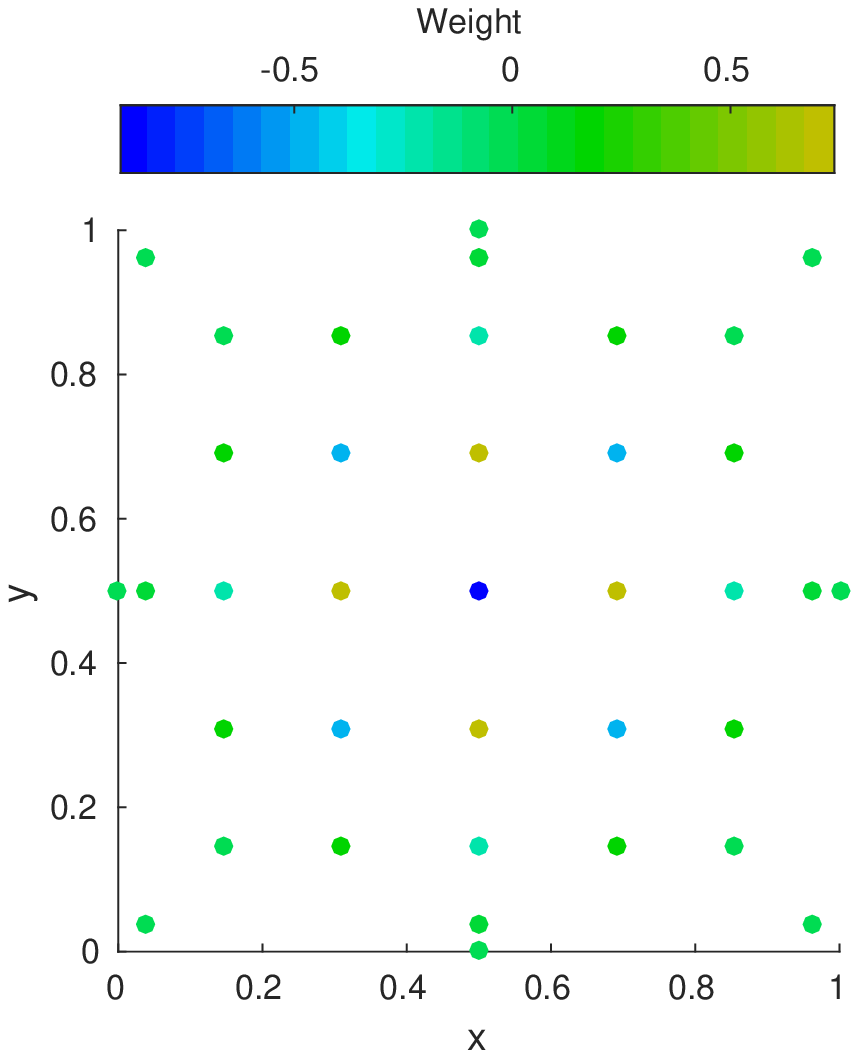}
		\subcaption{Neg.\ symmetric red.\ cubature rule (37~nodes)}
		\label{fig:nestedex3}
	\end{minipage}
	\caption{The discussed multi-dimensional cubature rules. All rules are of degree 9, generated using Clenshaw--Curtis quadrature rules. Initial tensor grids of the reduced rules are $9 \times 9$. Negative (neg.) and reduced (red.) is abbreviated.}
\end{figure}

\subsection{Condition number}
\label{subsec:condnum}
For the Smolyak cubature rule the growth of the condition number with respect to the number of nodes is bounded (recall \eqref{eq:smolnorm}). The reduced quadrature rule has positive weights, so the condition number equals $1$ in this case. For the negative reduced cubature rule no such bounds exist, as far as the authors know. The condition number can be deduced numerically to assess its growth (see Figure~\ref{fig:cond}). The maximum degree (which is 15 here) is chosen such that (numerically) the sum of the weights equals 1 with a maximum error of $10^{-12}$. We are primarily interested in $\kappa$, not in the numerical accuracy of the procedure.

The growth of the condition number of the Smolyak rules is equal, which is evident. The Smolyak rule generated with a reduced quadrature rule of the prior criterion has larger condition number than the Smolyak rule generated using Clenshaw--Curtis rules. If this is unwanted, we suggest a \emph{weight criterion} where the reduced rule is selected with the smallest mutual difference, i.e., with the smallest $\max_k w_k - \min_k w_k$.
The condition number of the Smolyak rule generated with this reduced quadrature rule is significantly smaller and close to the condition number of a Clenshaw--Curtis. However, this criterion does not use the underlying distribution, so we do not study it further.

The condition number of the reduced negative cubature rule is smaller than that of the Smolyak rule. However, for larger degrees severe numerical errors occur in the algorithm to generate these rules, which is not the case for the Smolyak rule.

\begin{figure}
	\centering
	\includegraphics[width=.5\textwidth]{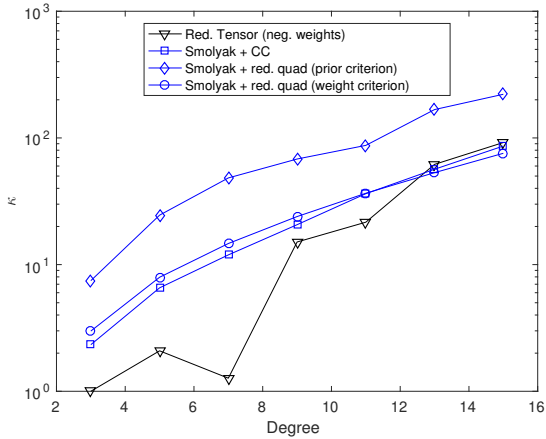}
	\caption{The condition number $\kappa$ of the four cubature rules under consideration that have negative weights. All rules are 5-dimensional. Here, ``red.\ quad'' and ``CC'' stand for ``reduced quadrature rule'' and ``Clenshaw--Curtis'' respectively.}
	\label{fig:cond}
\end{figure}

\section{Numerical results}
\label{sec:numerics}
In this section the proposed cubature rules are applied to multiple problems and compared with tensor product and Smolyak cubature rule.

This section is built as follows: in the first sub-section the cubature rules will be used to integrate the Genz test functions. These functions are designed for testing cubature rules. The second and third sub-section contain applications of the cubature rules to two UQ cases. In the second sub-section the standard lid-driven cavity flow problem with uncertain boundary conditions and material properties will be studied, using a Lattice Boltzmann method to compute the flow. In the final sub-section the main advantage of allowing negative weights is shown, i.e.\ high accuracy for a moderately high-dimensional problem. An aircraft aerodynamics test case is considered with seven uncertain parameters, using the Euler equations of gas dynamics and a finite-volume discretization of these to compute the corresponding aircraft aerodynamics. Because the conventional methods require a large number of simulations, only the results of the reduced cubature rule with negative weights are discussed in this case.

\subsection{Genz test functions}
\subsubsection{Uniform distribution}
\label{subsubsec:unifdist}
To test the quality of cubature rules, several functions have been developed by Genz \cite{Genz1984}. Each function has a different specific property or attribute, of which the effect can be enlarged by a parameter $\mathbf{a}$. A shape parameter $\mathbf{u}$ can be used to transform the function without changing the property (see Table~\ref{tbl:genz} for all functions and their relevant attributes). For all functions the exact value of the integral can be determined \cite{Patterson1987}.

\begin{table}
	\centering
	\caption{The test functions from Genz \citep{Genz1984}. All functions are from a certain integrand family and depend on the parameters $\mathbf{a} = \trans{(a_1, \dots, a_N)}$ and $\mathbf{u} = \trans{(u_1, \dots, u_N)}$. The parameter $\mathbf{u}$ is a parameter that does not affect the difficulty of the integral. The parameter $\mathbf{a}$ determines the degree to which the family attribute is present.}
	\label{tbl:genz}
	\begin{tabular}{l l}
		\textbf{Integrand Family} & \textbf{Attribute} \\
		\hline
		\hline
		$f_1(\mathbf{x}) = \cos\left(2\pi u_1 + \sum_{i=1}^n a_i x_i\right)$ & Oscillatory \\
		$f_2(\mathbf{x}) = \prod_{i=1}^n \left(a_i^{-2} + (x_i - u_i)^2\right)^{-1}$ & Product Peak \\
		$f_3(\mathbf{x}) = \left(1 + \sum_{i=1}^n a_i x_i\right)^{-(n+1)}$ & Corner Peak \\
		$f_4(\mathbf{x}) = \exp\left(- \sum_{i=1}^n a_i^2 (x_i - u_i)^2 \right)$ & Gaussian \\
		$f_5(\mathbf{x}) = \exp\left(- \sum_{i=1}^n a_i |x_i - u_i|\right)$ & $C_0$ function \\
		$f_6(\mathbf{x}) = \begin{cases}
			0 &\text{if $x_1 > u_1$ or $x_2 > u_2$} \\
			\exp\left(\sum_{i=1}^n a_i x_i\right) &\text{otherwise}
		\end{cases}$ & Discontinuous
	\end{tabular}
\end{table}

Reducing a cubature rule only maintains the polynomial accuracy, which requires sufficient smoothness of the integrand. The first four Genz functions are in $C^\infty({[0,1]}^d)$, while the fifth is in $C^0({[0,1]}^d)$, and the sixth is only piecewise continuous. Hence \emph{a priori} we expect spectral convergence for the first four functions and poor convergence for the fifth and sixth, independent of the particular rule.

To obtain meaningful, instructive results the coefficients $\mathbf{a}$ and $\mathbf{u}$ are chosen randomly, with each component from similar uniform distributions, subject to the constraints $\|\mathbf{a}\|_2 = 2.5$ and $\|\mathbf{u}\|_2 = 1$. Moreover, each component of both $\mathbf{a}$ and $\mathbf{u}$ is positive. The integration error is determined with respect to the exact solution, and averaged over 100 runs. 

Convergence plots for all methods and all Genz functions are depicted in Figure~\ref{fig:genzllf}. The multi-dimensional reduced rules are initiated using tensor products of Gaussian quadrature rules and the reduction procedure is only applied once to keep numerical artifacts small. The Smolyak procedure is applied twice using Clenshaw--Curtis quadrature rules or reduced quadrature rules using a fine Clenshaw--Curtis rule as initial rule. The results can be divided into three classes that exhibit different behaviors: (i) $f_1$, $f_2$, $f_4$, (ii) $f_3$, and (iii) $f_5$, $f_6$.

Class (i) is formed by smooth results, which show almost spectral convergence for all methods. Both symmetric reduced rules consistently outperform the tensor product, and negative symmetric reduced rules also consistently outperform both Smolyak rules. No significant difference exists between the two Smolyak rules.

Class (ii) is an exception, most likely caused by the concentration of mass at one corner of the integration domain. Both Smolyak and reduced (negative weights) rules remove nodes at corners, and thereby poorly approximate the most important region of the integrand. The reduced quadrature rule keeps some nodes at the corner up to small levels, so therefore the Smolyak rule with reduced rules performs slightly better. The tensor and reduced (positive weights) rules do have nodes there, which makes the error much smaller.

For class (iii) the integrands lack sufficient smoothness for polynomial approximations to be stable. As expected, spectral convergence is not evident, but some limited linear convergence is visible. In both cases Smolyak rules acquit themselves well compared to all other methods.

In summary the proposed reduced rules are empirically converging at the level of Smolyak or better, given sufficient smoothness in the integrand.

\begin{figure}
	\centering
	\begin{minipage}{.45\textwidth}
		\centering
		\includegraphics[width=\textwidth]{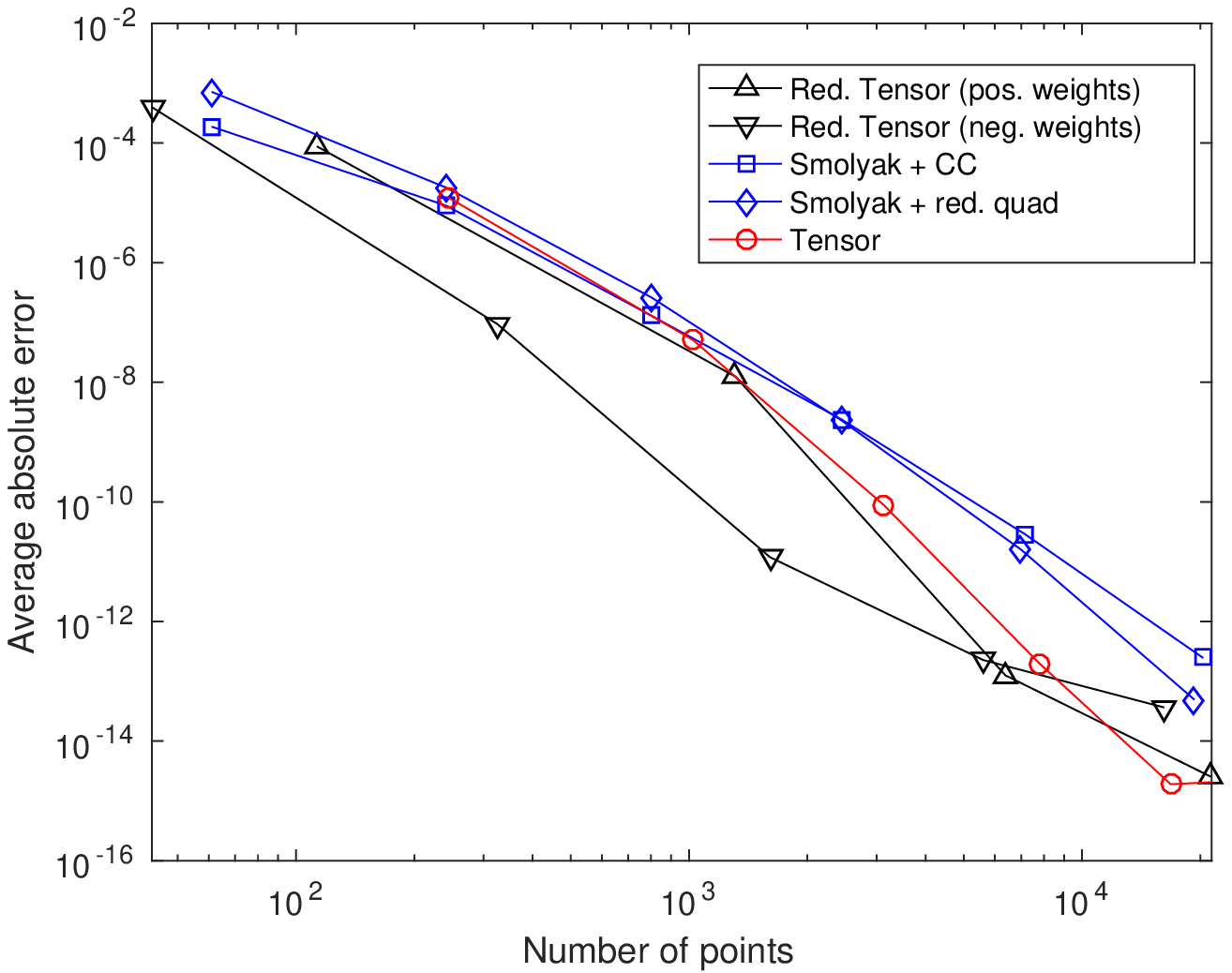}
		\subcaption{$f_1$}
	\end{minipage}
	\begin{minipage}{.45\textwidth}
		\centering
		\includegraphics[width=\textwidth]{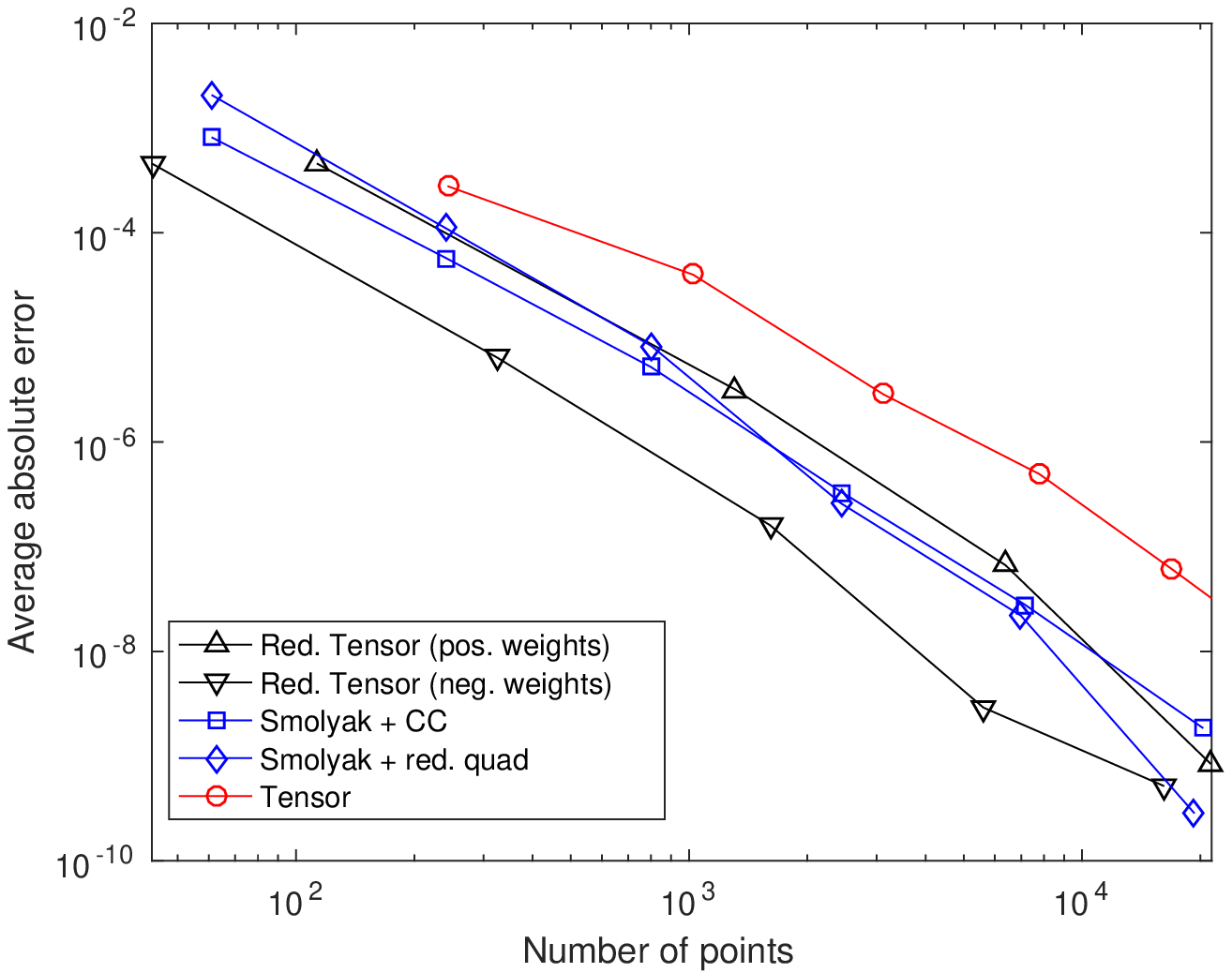}
		\subcaption{$f_2$}
	\end{minipage}
	\begin{minipage}{.45\textwidth}
		\centering
		\includegraphics[width=\textwidth]{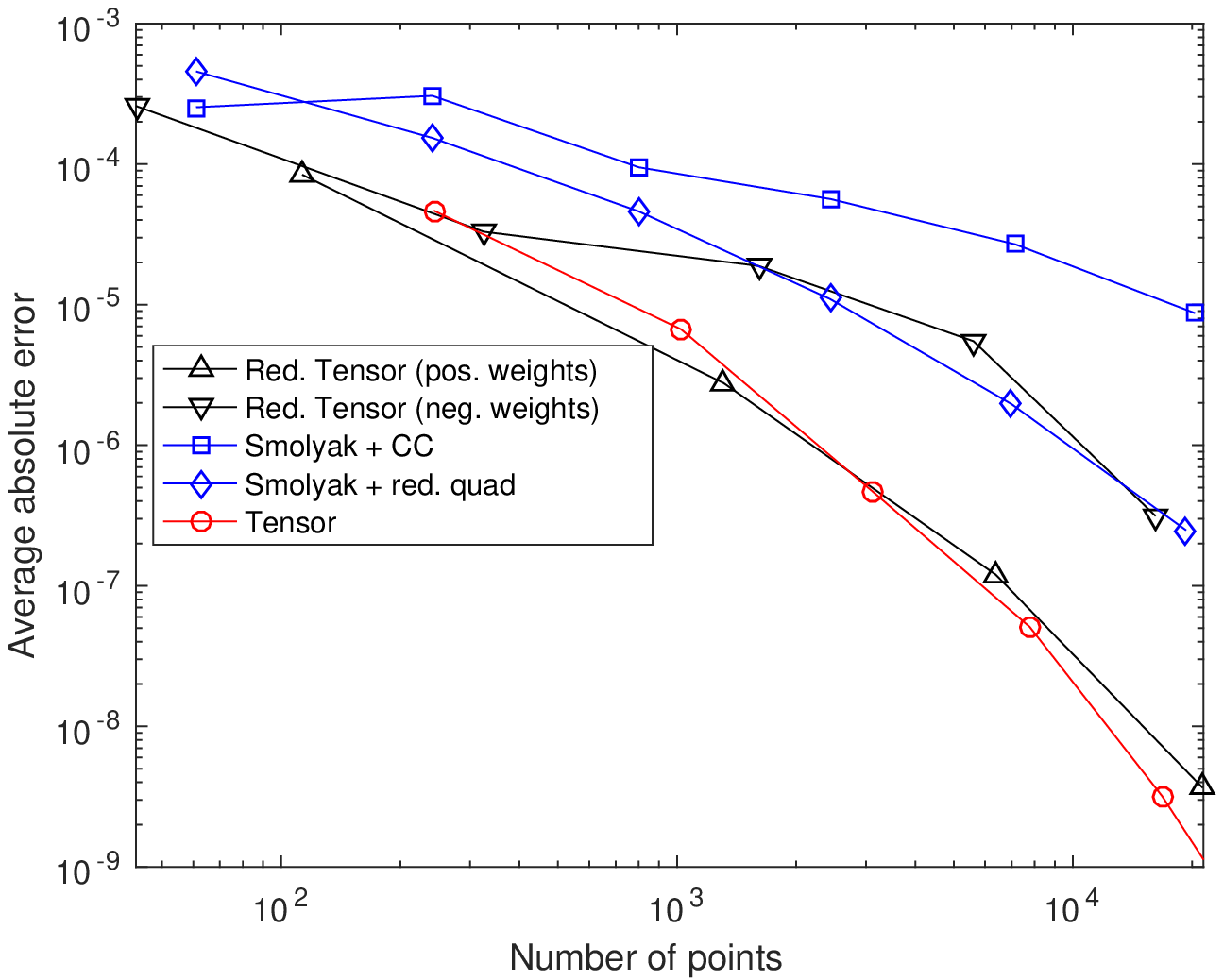}
		\subcaption{$f_3$}
	\end{minipage}
	\begin{minipage}{.45\textwidth}
		\centering
		\includegraphics[width=\textwidth]{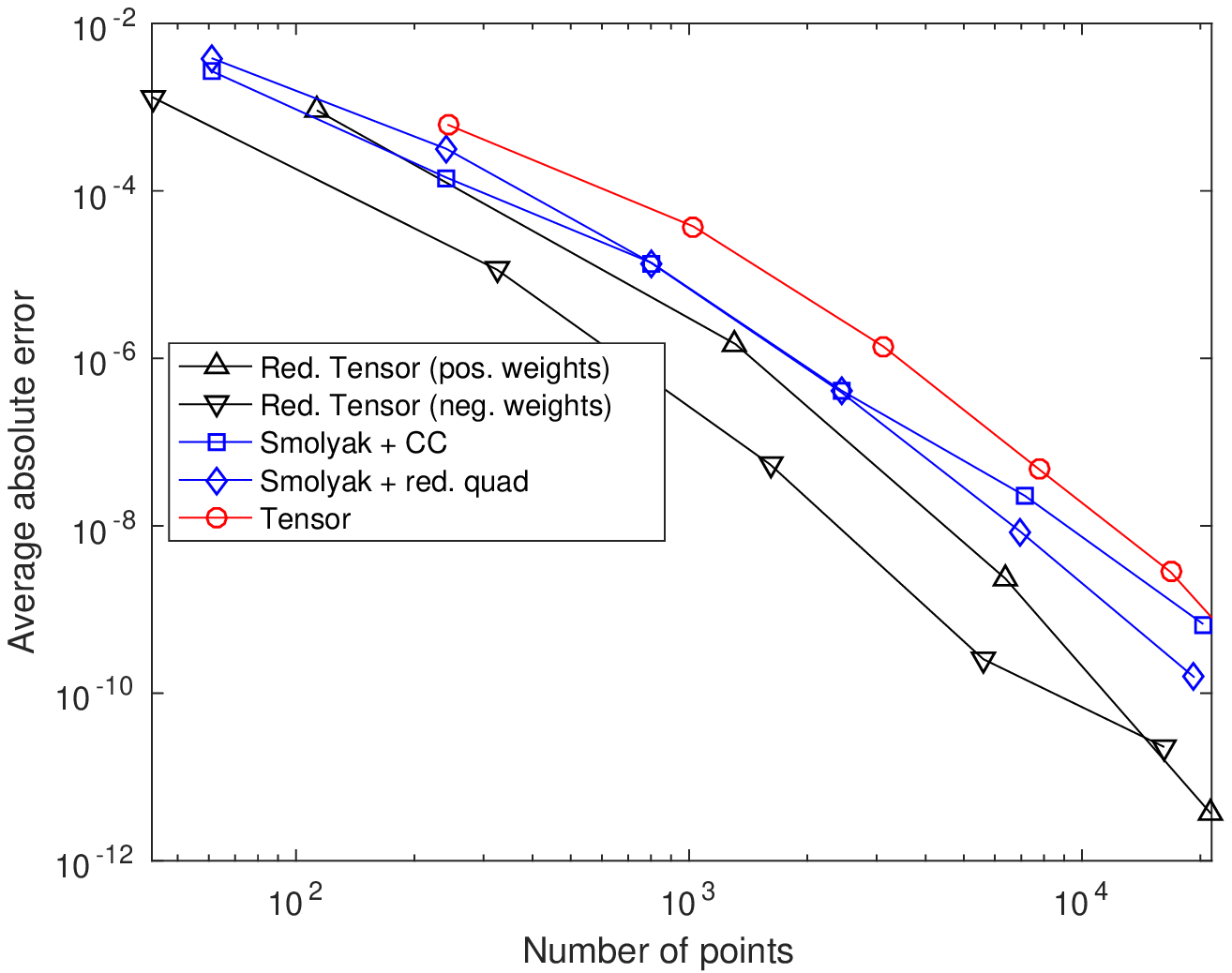}
		\subcaption{$f_4$}
	\end{minipage}
	\begin{minipage}{.45\textwidth}
		\centering
		\includegraphics[width=\textwidth]{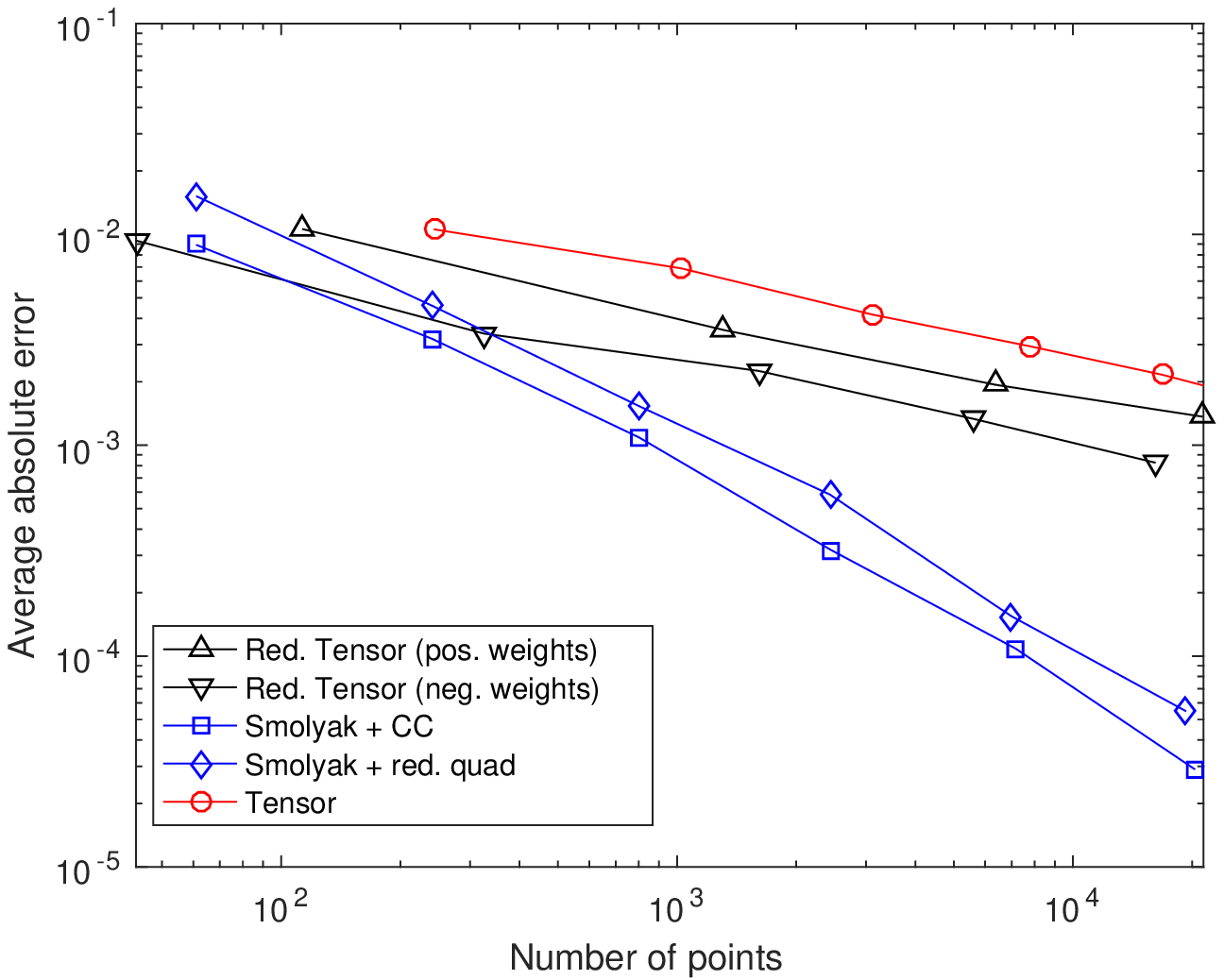}
		\subcaption{$f_5$}
	\end{minipage}
	\begin{minipage}{.45\textwidth}
		\centering
		\includegraphics[width=\textwidth]{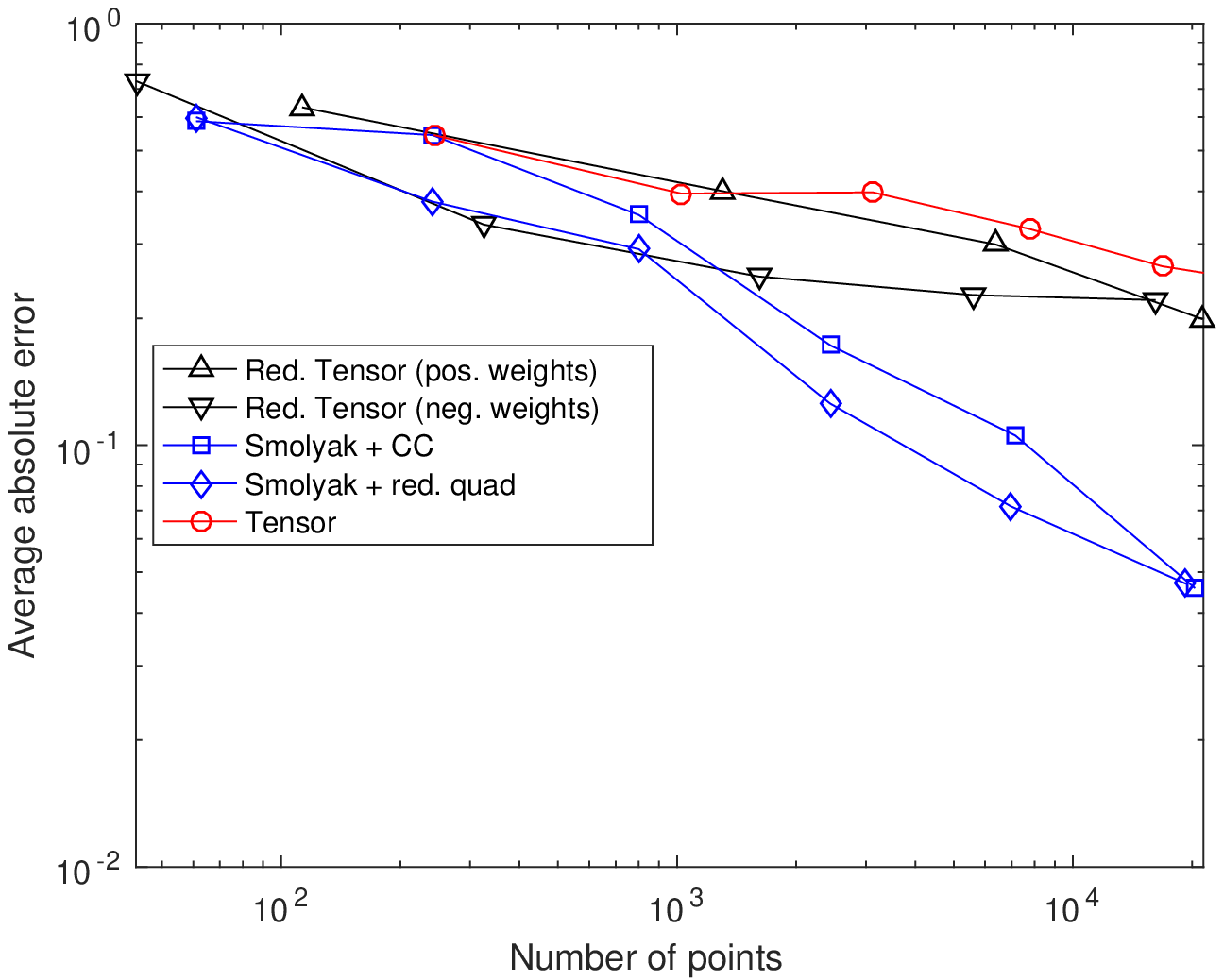}
		\subcaption{$f_6$}
	\end{minipage}
	\caption{The accuracy of several cubature rules versus the number of nodes that are in the cubature rule. All integrals are 5-dimensional. Here, ``red.\ quad'' and ``CC'' stand for ``reduced quadrature rule'' and ``Clenshaw--Curtis'' respectively.}
	\label{fig:genzllf}
\end{figure}

\subsubsection{Non-uniform distribution}
The reduced cubature rule can be determined for any distribution whose moments can be evaluated. Therefore we assess the convergence of the rules using a $\beta(10, 10)$-distribution, a highly non-uniform distribution. The Smolyak cubature rule with Clenshaw--Curtis quadrature rules is not considered anymore. We only examine the results of $f_1$, $f_2$, and $f_4$, i.e., the cases where convergence was observed.

The results are created in a similar way as in the uniform case, i.e., the coefficients $\mathbf{a}$ and $\mathbf{u}$ are chosen randomly subject to the constraints $\|\mathbf{a}\|_2 = 2.5$ and $\|\mathbf{u}\|_2 = 1$. The integration error is determined with respect to a reference value, calculated using a $30^5$ tensor grid created with Gaussian quadrature rules. Convergence plots are depicted in Figure~\ref{fig:genzllfbeta}.

In comparison with the previous results we see better convergence of $f_1$ for all cubature rules. This is due to the $\beta(10, 10)$-distribution, which damps the oscillations of the function under consideration. Moreover it is clearly visible that an approximate value of the integral is used here.

$f_2$ and $f_4$ show similar results: in both cases the reduced rule with negative weights shows the best results. The differences are larger in this case, which is due to the $\beta(10, 10)$-distribution that introduces many small weights in the initial quadrature and cubature rule. These weights are prone to removal in both the negative and positive reduction algorithm.

\begin{figure}[t]
	\centering
	\begin{minipage}{.45\textwidth}
		\centering
		\includegraphics[width=\textwidth]{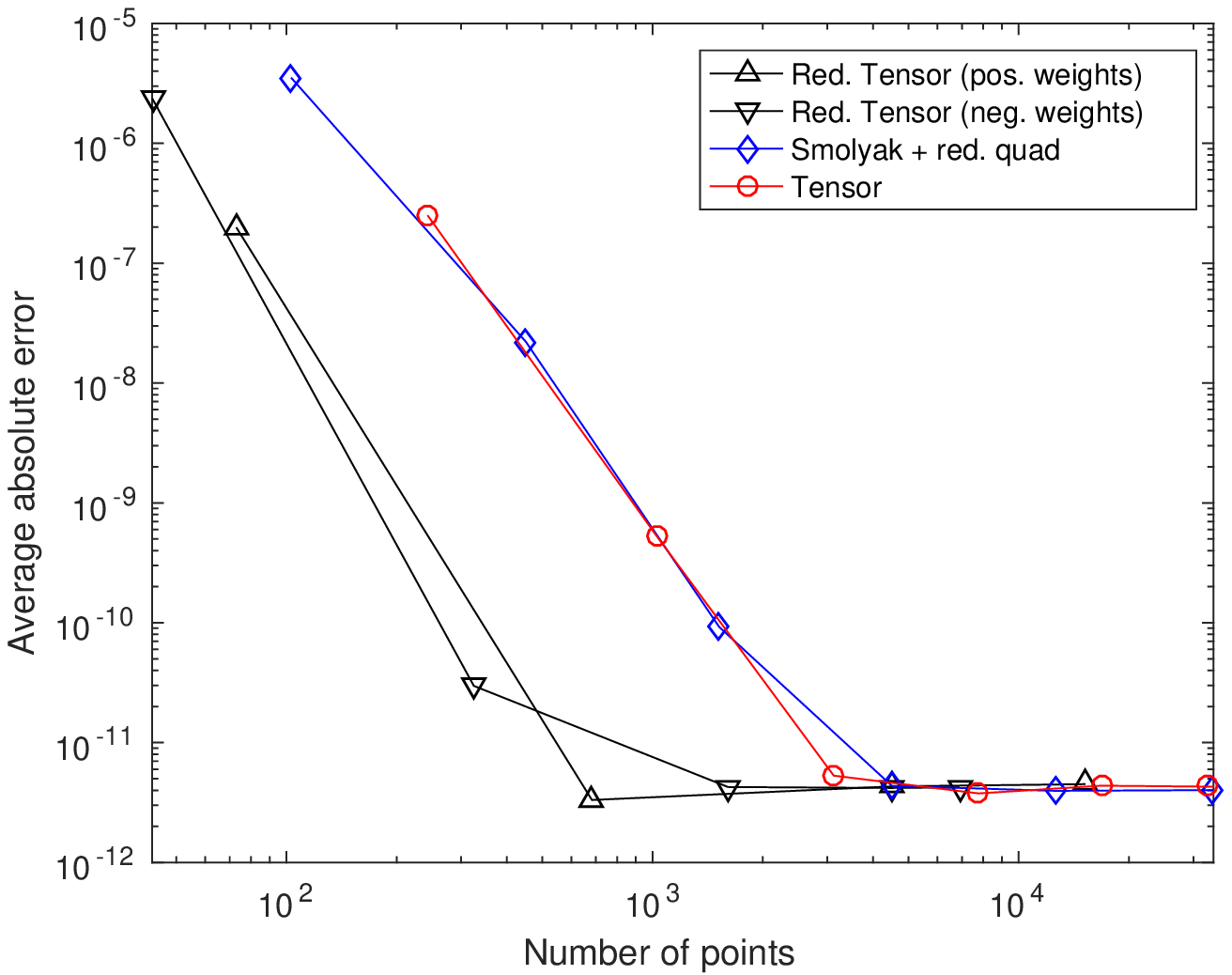}
		\subcaption{$f_1$}
	\end{minipage}
	\begin{minipage}{.45\textwidth}
		\centering
		\includegraphics[width=\textwidth]{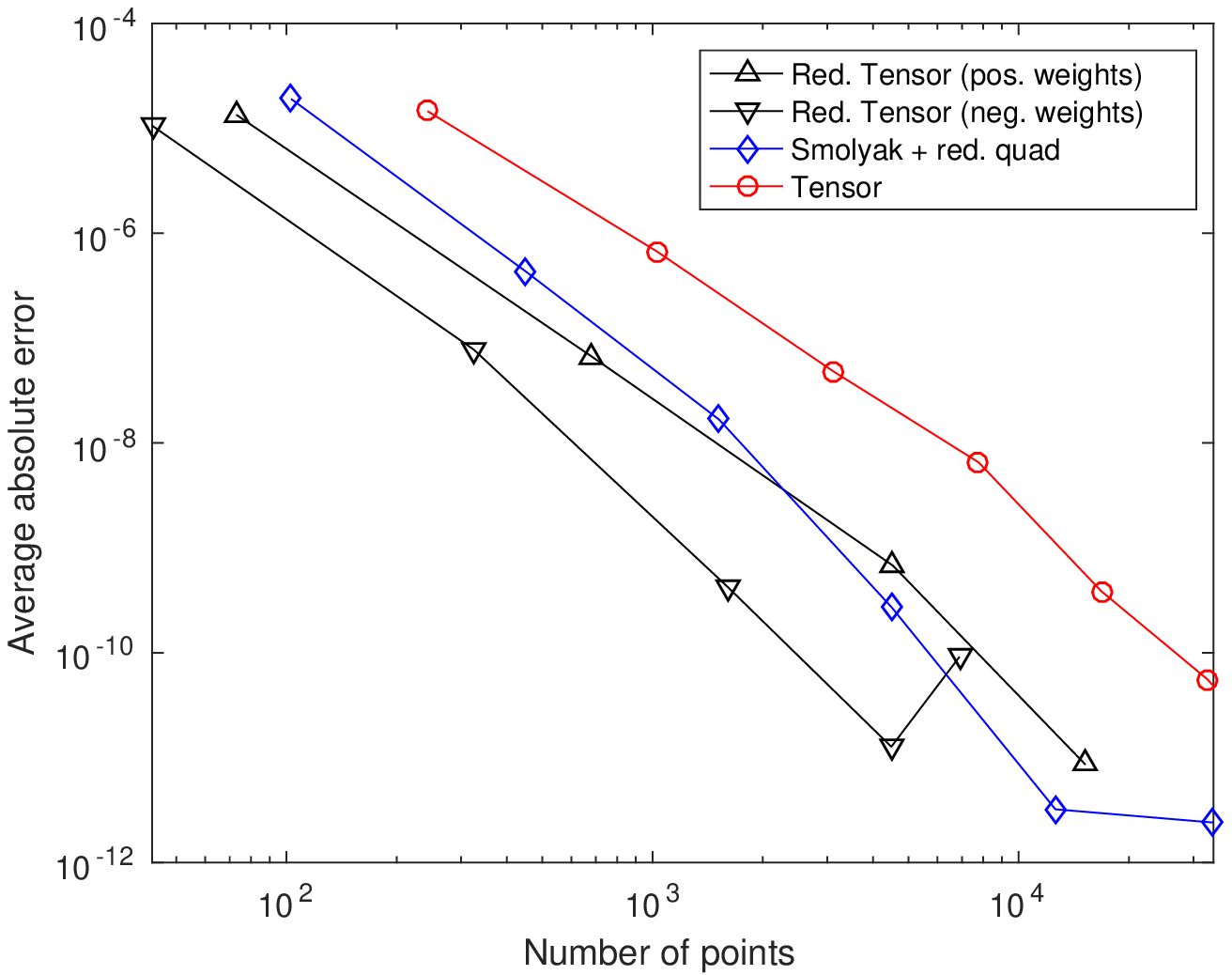}
		\subcaption{$f_2$}
	\end{minipage}
	\begin{minipage}{.45\textwidth}
		\centering
		\includegraphics[width=\textwidth]{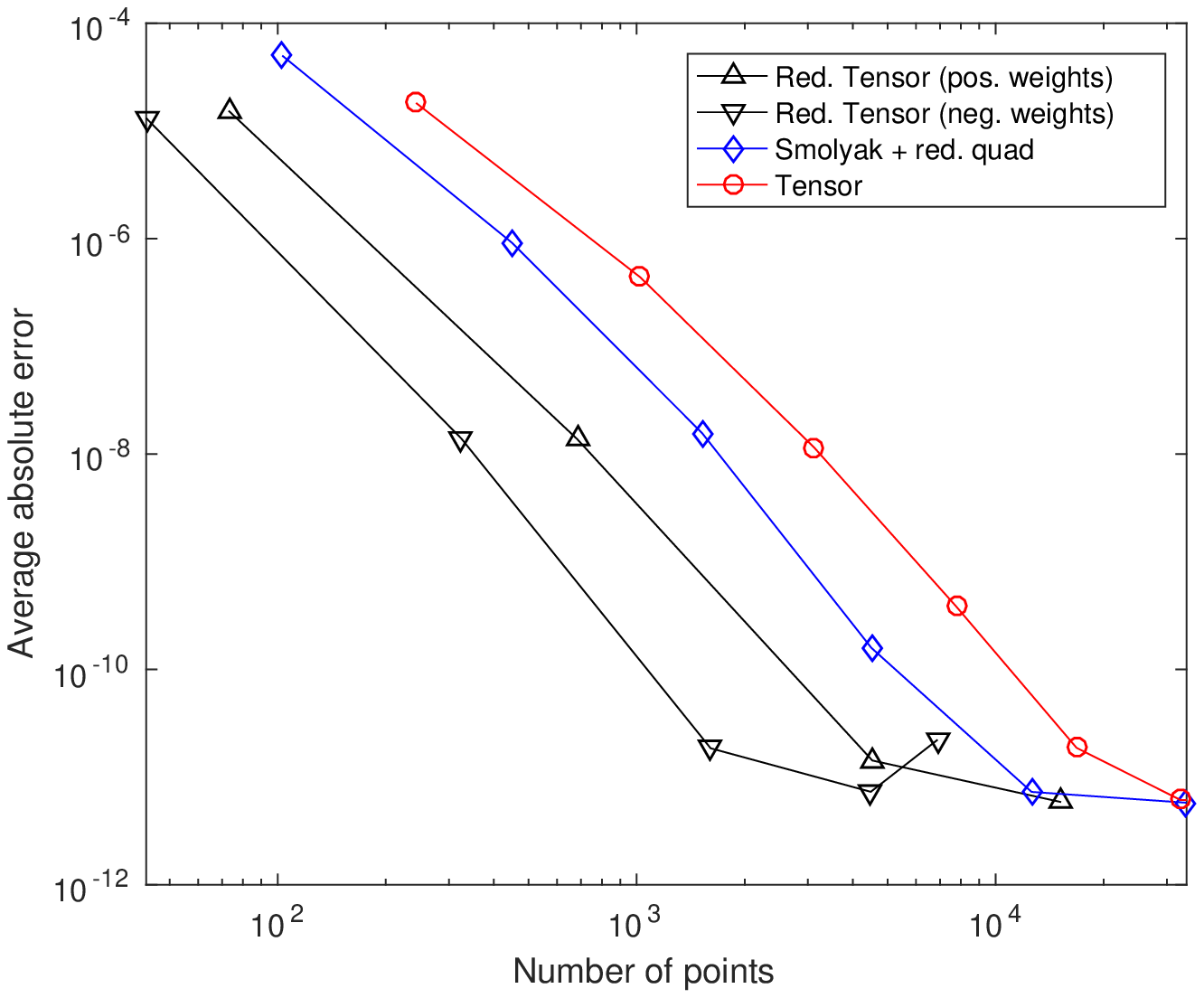}
		\subcaption{$f_4$}
	\end{minipage}
	\caption{The accuracy of several cubature rules versus the number of nodes that are in the cubature rule. All integrals are 5-dimensional, the distribution under consideration is $\beta(10, 10)$.}
	\label{fig:genzllfbeta}
\end{figure}

\subsubsection{Dimension dependence}
\label{subsubsec:dimdep}
All integrals so far have been determined using 5-dimensional rules. To quantify the performance of rules depending on the dimension, the integration error is studied for varying dimension.

The results are again created in a similar way as in the previous case, i.e.\ using the random coefficients and averaging the result. The uniform distribution is reconsidered, such that an exact value of the integral is known. We again limit ourselves to $f_1$, $f_2$, and $f_4$. All cubature rules are generated such that they are of degree 9. The number of nodes is not taken into account here (but can be found in Table~\ref{tbl:results}). The integration errors up to 10 dimensions are plotted in Figure~\ref{fig:genzllgrowth}.

The oscillatory function $f_1$ and the Gaussian function do not become more difficult to integrate in higher dimensions, as the Taylor expansions are comparable to the one-dimensional case. However, the mass of the product peak of $f_2$ becomes smaller as the dimension increases, which makes the integral easier to evaluate numerically, as integrating the peak accurately becomes of less importance. This is also reflected in the results. The tensor product rule shows excellent results for both functions, which is due to the exact integration of more polynomials in comparison to the other rules.

The positive reduced cubature rule shows the smallest growth compared to the other rules (excluding the tensor rule). This is due to the positive weights, that yield a condition number equal to 1.

The negative reduced cubature rule and the two Smolyak rules show similar growth. Although the negative reduced rule does not reduce or increase the integration error, it does need much less nodes to obtain this error compared with both Smolyak rules (see Table~\ref{tbl:results}).

In summary, the tensor grid yields the largest cubature rule and shows the smallest error. The reduced positive rule needs less nodes, but yields a larger error. The two Smolyak rules and the reduced negative rule have approximately equivalent error, but the reduced negative rule yields a much smaller grid.

\begin{figure}[h!]
	\centering
	\begin{minipage}{.45\textwidth}
		\centering
		\includegraphics[width=\textwidth]{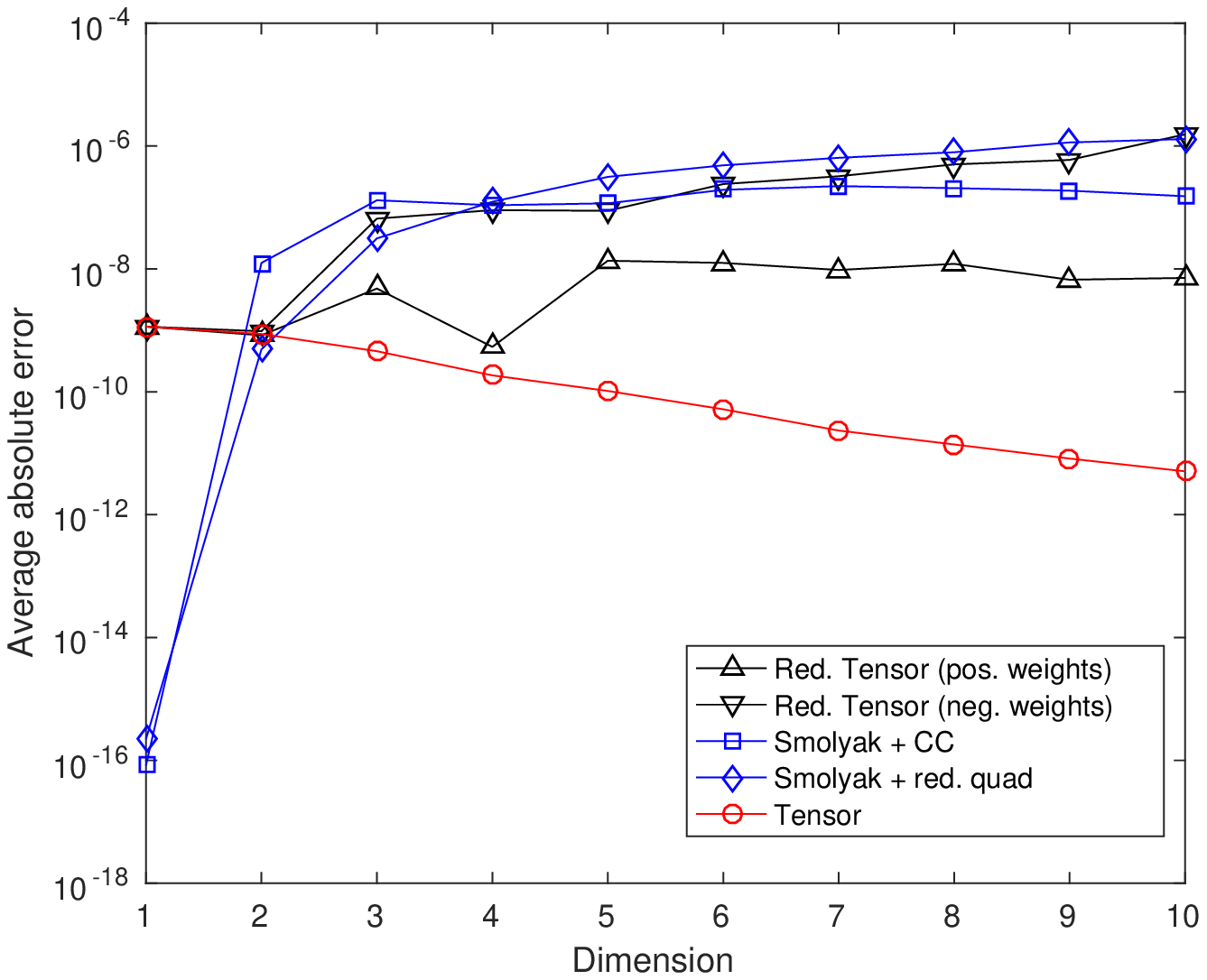}
		\subcaption{$f_1$}
	\end{minipage}
	\begin{minipage}{.45\textwidth}
		\centering
		\includegraphics[width=\textwidth]{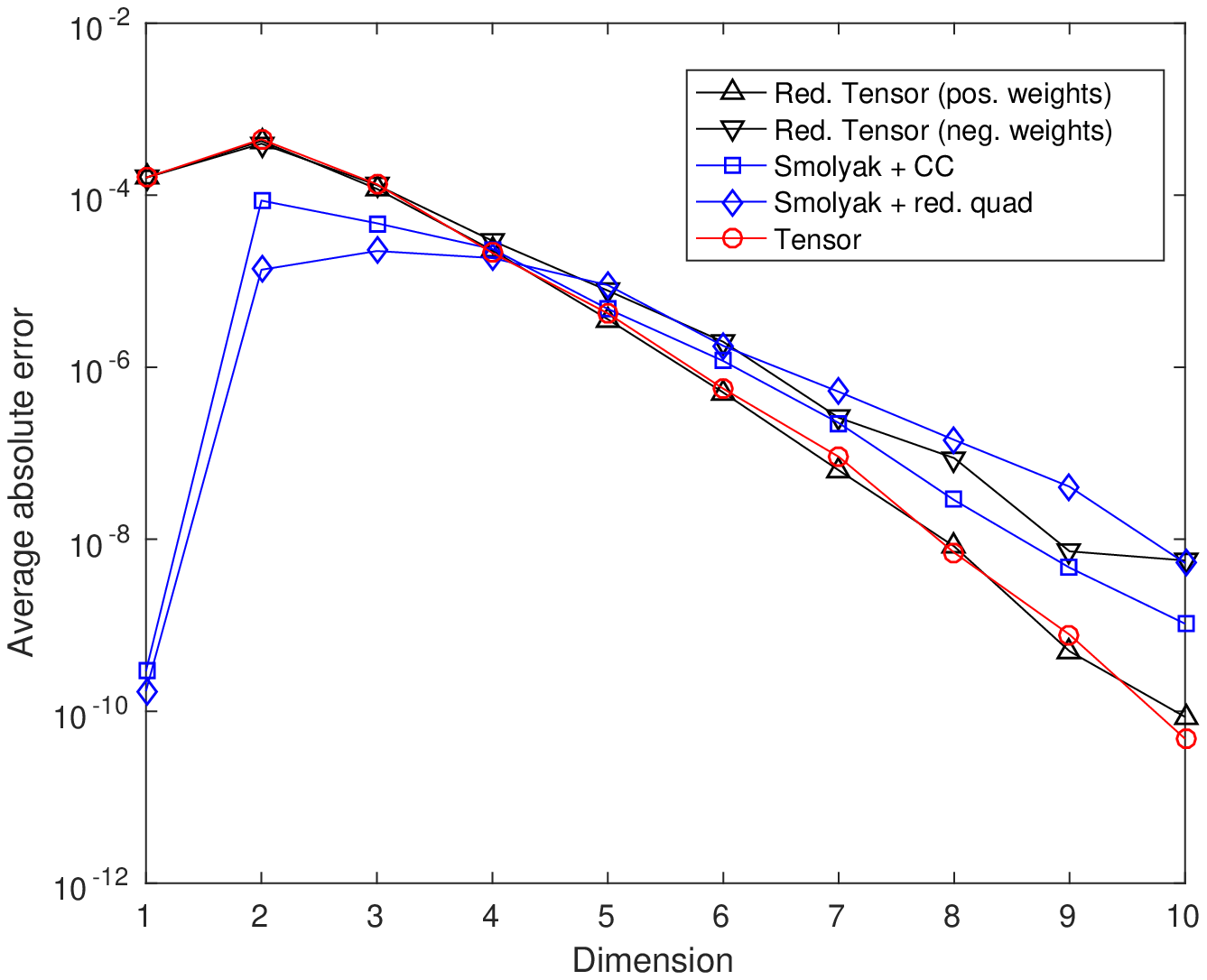}
		\subcaption{$f_2$}
	\end{minipage}
	\begin{minipage}{.45\textwidth}
		\centering
		\includegraphics[width=\textwidth]{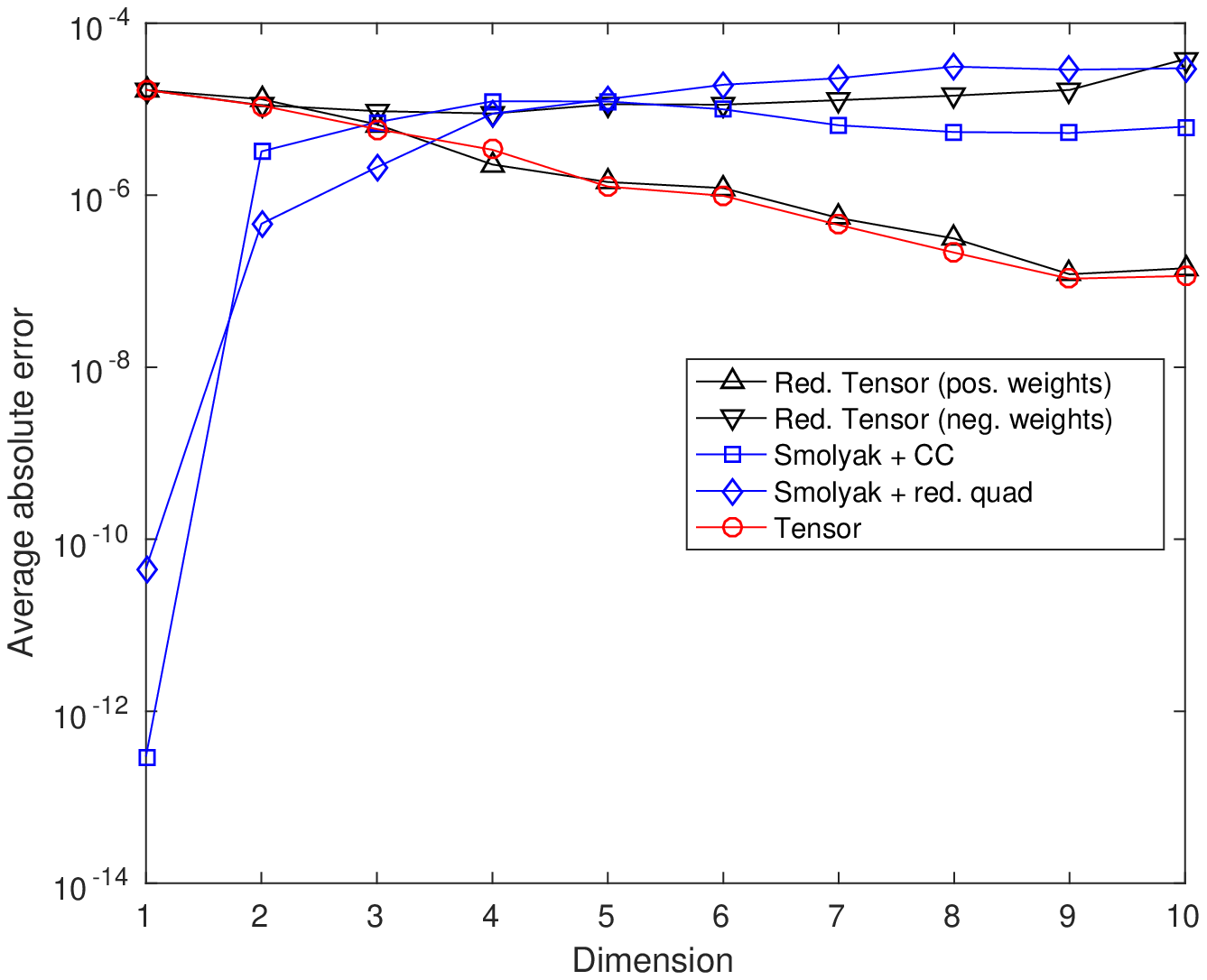}
		\subcaption{$f_4$}
	\end{minipage}
	\caption{The accuracy of the cubature rules under consideration versus the dimension. All cubature rules are of degree 9 and have minimal number of nodes.}
	\label{fig:genzllgrowth}
\end{figure}

\subsection{Lid-driven cavity flow test case, using a Lattice Boltzmann method}
\subsubsection{Problem description}
The standard lid-driven cavity flow (e.g.~\cite{Ghia1982}) is considered with two uncertain flow parameters, and four UQ methods are compared, namely MC, SC with a Smolyak sparse grid, and SC with the two new reduced cubature rules. A sketch of the geometry and the imposed boundary conditions is given in Figure~\ref{fig:lcf}. The boundary condition imposed at both singular corners is $\mathbf{u} = 0$, where $\mathbf{u}$ is the fluid velocity vector.

\begin{figure}[t]
	\centering
	\includegraphics[width=.3\textwidth]{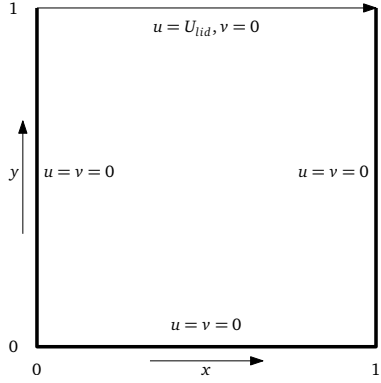}
	\caption{The geometry and boundary conditions of the lid-driven cavity flow.}
	\label{fig:lcf}
\end{figure}

The deterministic problem is solved using a Lattice Boltzmann method. The implementation used for the current case is a straightforward D2Q9 BGK-model using Zou--He boundary conditions \cite{Zou1997}. Reference data for several values of the Reynolds number can be found in Ghia et al.~\cite{Ghia1982}. The results from the Lattice Boltzmann implementation compare well with the data provided (see Figure~\ref{fig:ghia}).

\begin{figure}
	\centering
	\begin{minipage}{.45\textwidth}
		\includegraphics[width=\textwidth]{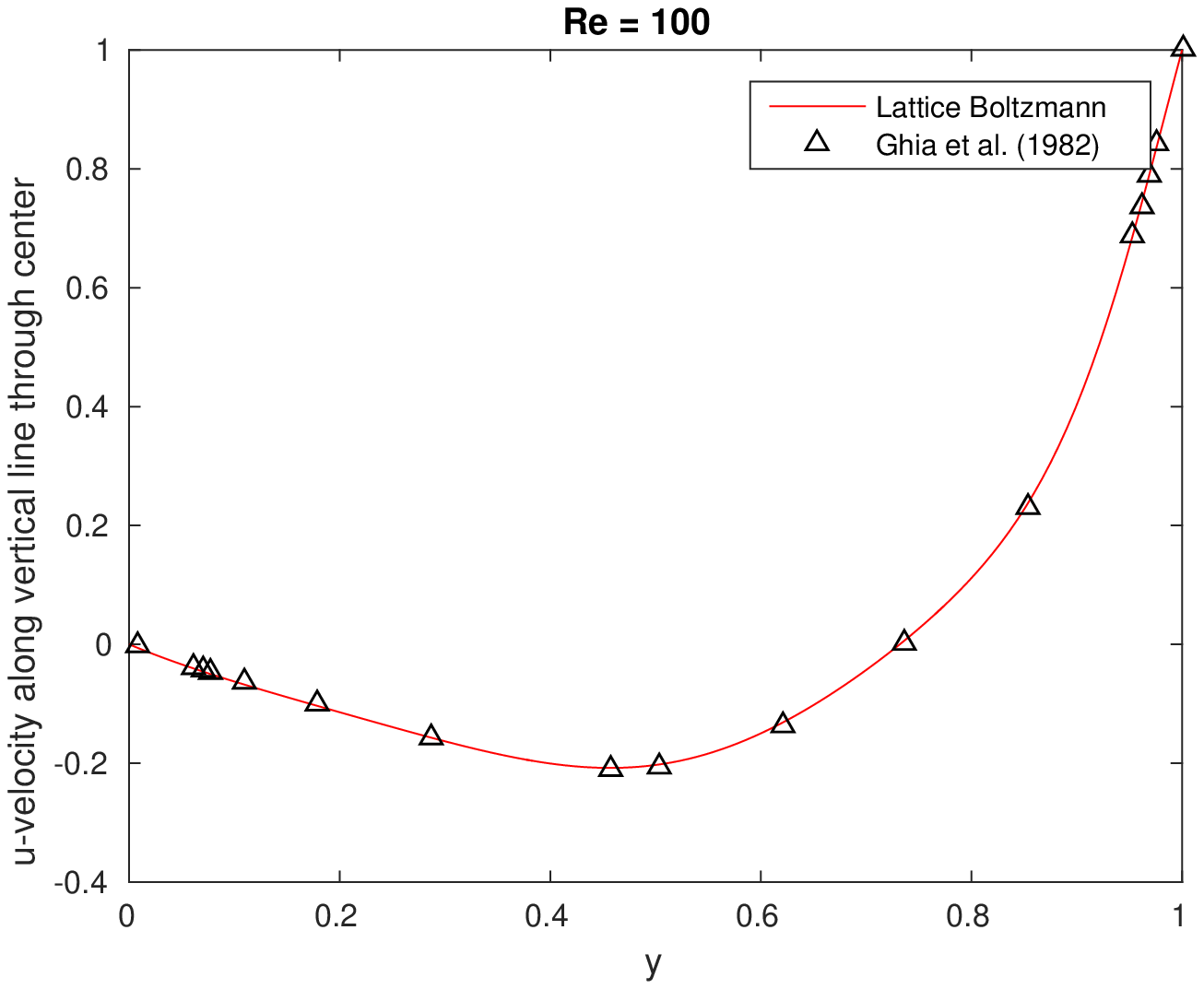}
	\end{minipage}
	~
	\begin{minipage}{.45\textwidth}
		\includegraphics[width=\textwidth]{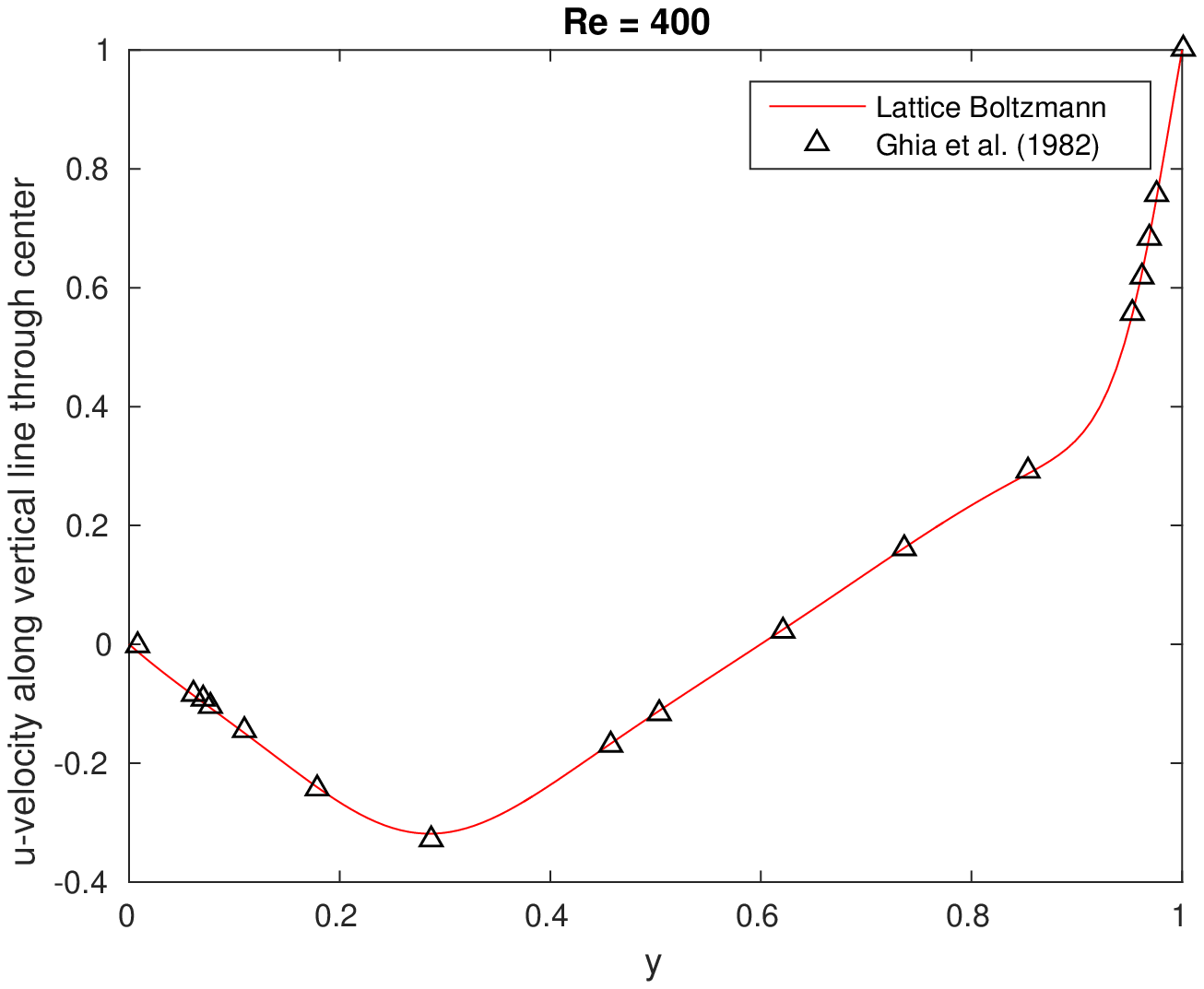}
	\end{minipage}
	\caption{The $u$-component of the flow velocity along the vertical line through the geometrical center of the cavity.}
	\label{fig:ghia}
\end{figure}

Two uncertain parameters are specified (see Table~\ref{tbl:lbmuparams}). Both parameters have a $\beta(a, b)$-distribution, with probability density function
\begin{equation}
	p(x; a, b) \propto x^{a-1} (1 - x)^{b-1} \text{ for } 0 \leq x \leq 1.
\end{equation}
The ranges are chosen such that the Reynolds number based on the lid velocity is between 10 and 400 (see Figure~\ref{fig:lbmuqcornercases} for the solutions of the two extreme cases).

\begin{table}[b]
	\centering
	\caption{Uncertain parameters and their distribution as considered for the lid-driven cavity flow problem.}
	\label{tbl:lbmuparams}
	\begin{tabular}{ll}
		\textbf{Parameter} & \textbf{Distribution} \\
		\hline
		\hline
		$u_\textrm{lid}$ (speed of the lid) & $\beta(3, 3)$ with range $(0.5, 1.5)$ \\
		$\nu$ (viscosity) & $\beta(4, 4)$ with range $(0.0038, 0.05)$
	\end{tabular}
\end{table}

\begin{figure}[t]
	\centering
	\begin{minipage}{.4\textwidth}
		\includegraphics[width=\textwidth]{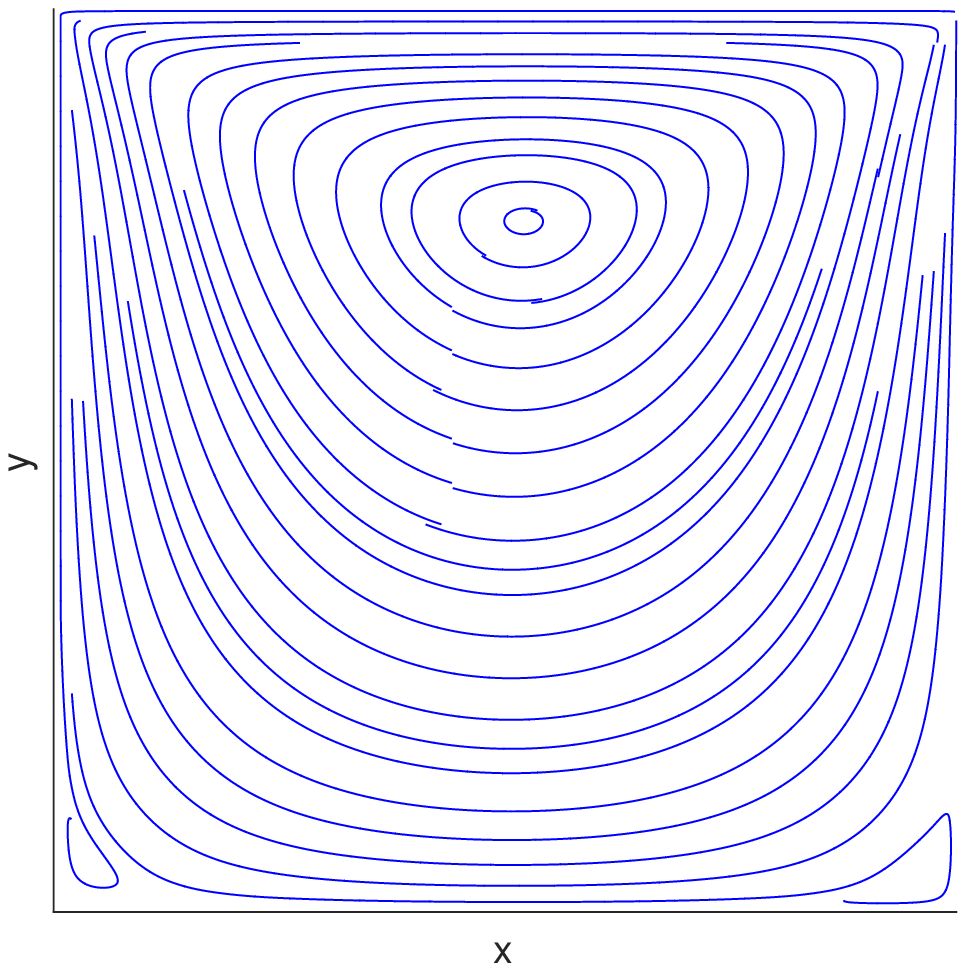}
		\subcaption{$Re = 10$}
	\end{minipage}
	~
	\begin{minipage}{.4\textwidth}
		\includegraphics[width=\textwidth]{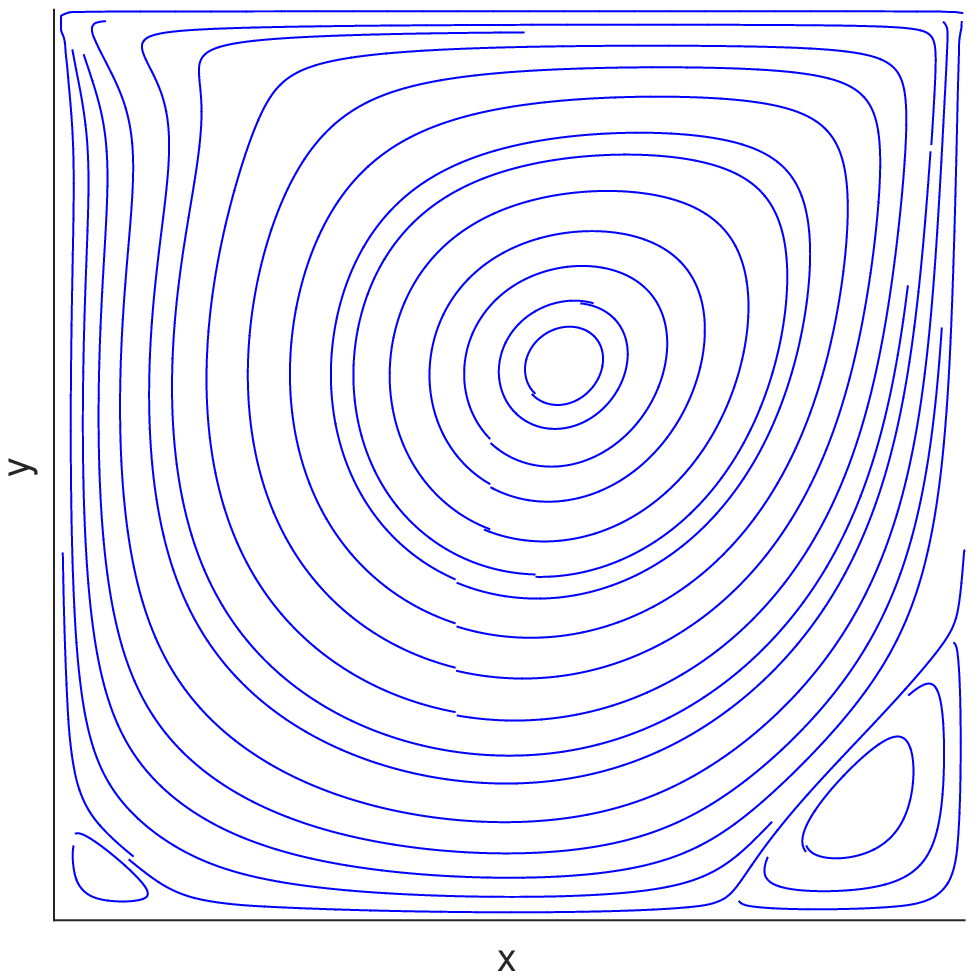}
		\subcaption{$Re = 400$}
	\end{minipage}
	\caption{Stream lines of the lid-driven cavity flow for the two extreme cases considered in the UQ problem.}
	\label{fig:lbmuqcornercases}
\end{figure}

\begin{figure}
	\centering
	\begin{minipage}{.4\textwidth}
		\includegraphics[width=\textwidth]{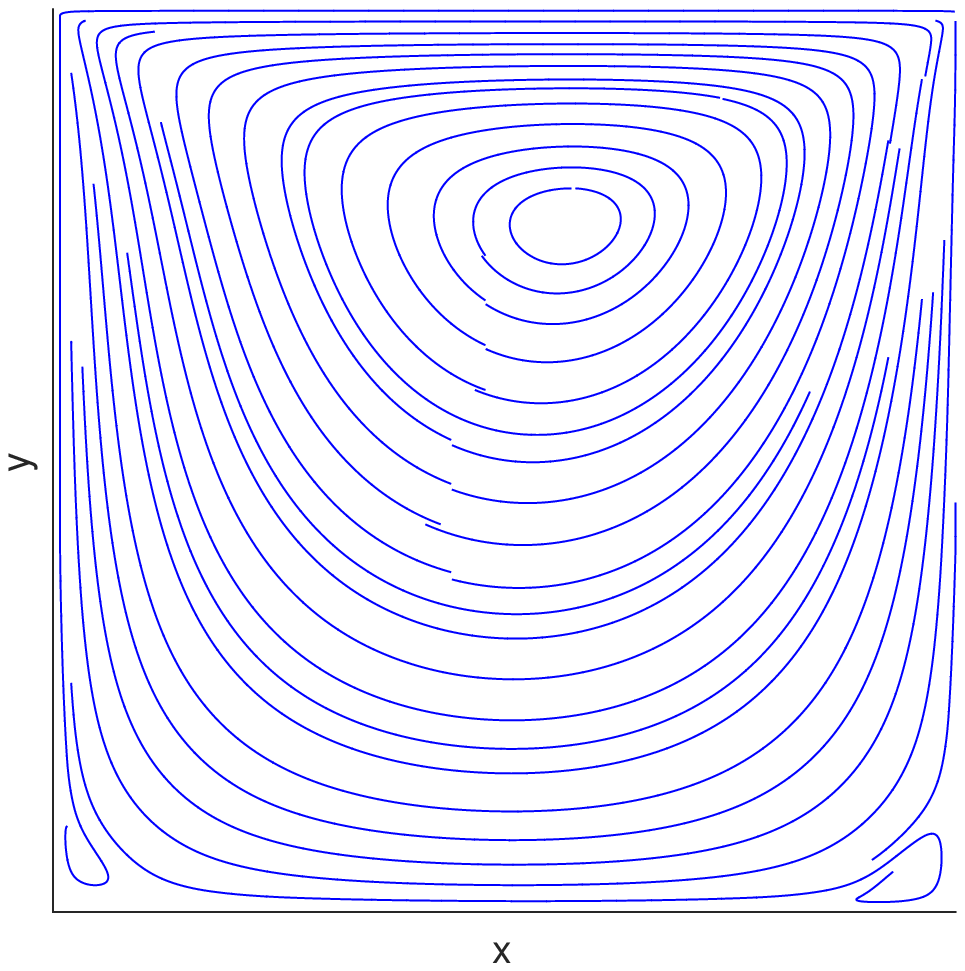}
	\end{minipage}
	~
	\begin{minipage}{.5\textwidth}
		\includegraphics[width=\textwidth]{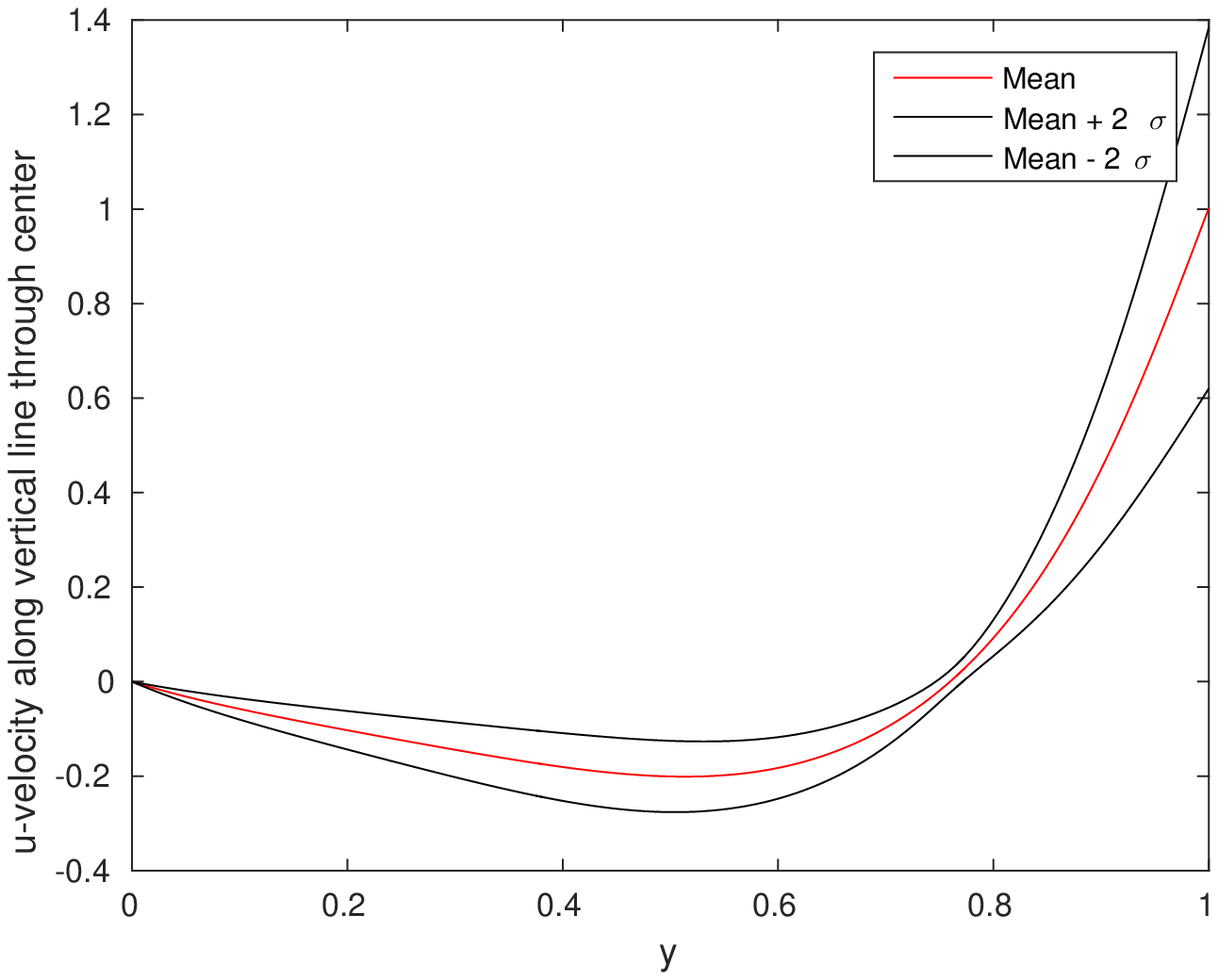}
	\end{minipage}
	\caption{Left: stream lines of the mean flow of the lid-driven cavity flow. Right: mean velocity component $u$ at $x = \frac{1}{2}$ with $2\sigma$ (``2 times standard deviation'') ranges.}
	\label{fig:lbmsolution}
\end{figure}

\subsubsection{Results}
To evaluate the accuracy of the methods we consider the $u$-component of the fluid velocity everywhere in the domain. A reference mean solution is obtained using a fine tensor product rule of $65 \times 65$ Gaussian nodes in the parameter space, resulting in a reference mean solution $\bar u^*$, depicted in Figure~\ref{fig:lbmsolution} (the tensor grid is plotted in Figure~\ref{fig:lbmgrids:tensor}).

\begin{figure}
	\centering
	\begin{minipage}{.45\textwidth}
		\centering
		\includegraphics[width=\textwidth]{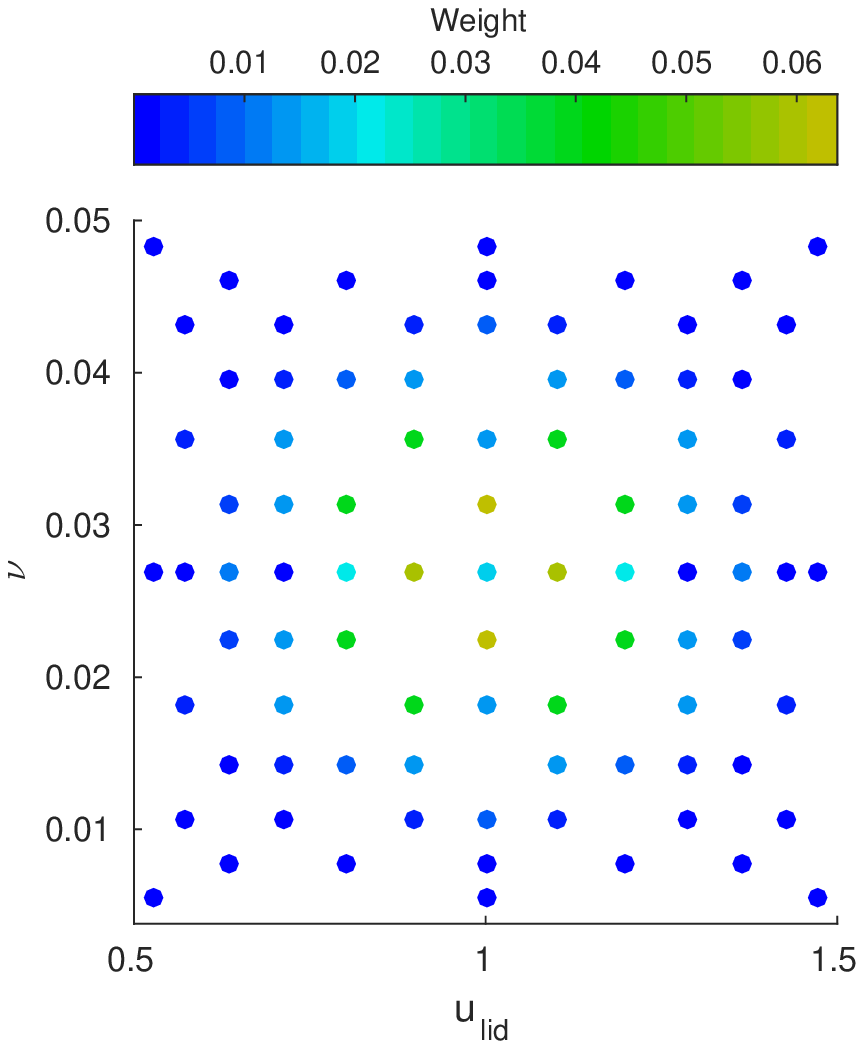}
		\subcaption{Symmetric red.\ cubature rule (87~nodes)}
		\label{fig:lbmgrids:nd1}
	\end{minipage}
	\begin{minipage}{.45\textwidth}
		\centering
		\includegraphics[width=\textwidth]{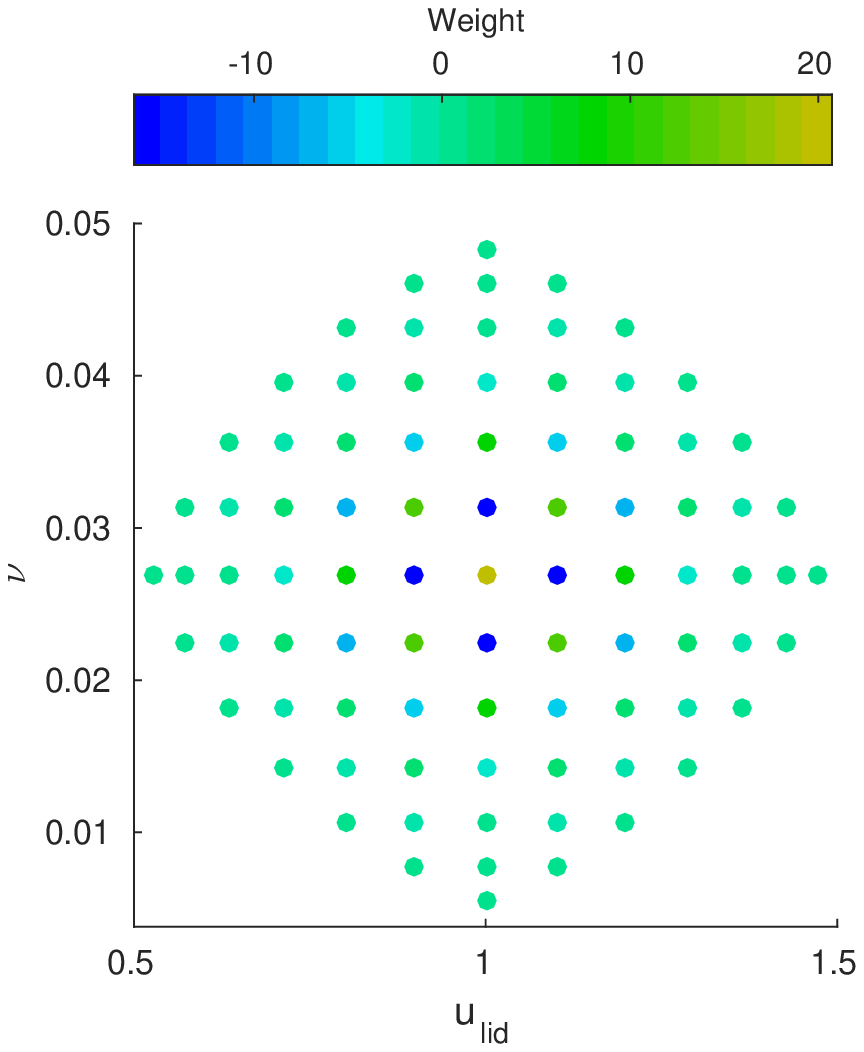}
		\subcaption{Neg.\ symmetric red.\ cubature rule (85~nodes)}
		\label{fig:lbmgrids:nd2}
	\end{minipage}
	\begin{minipage}{.45\textwidth}
		\centering
		\includegraphics[width=\textwidth]{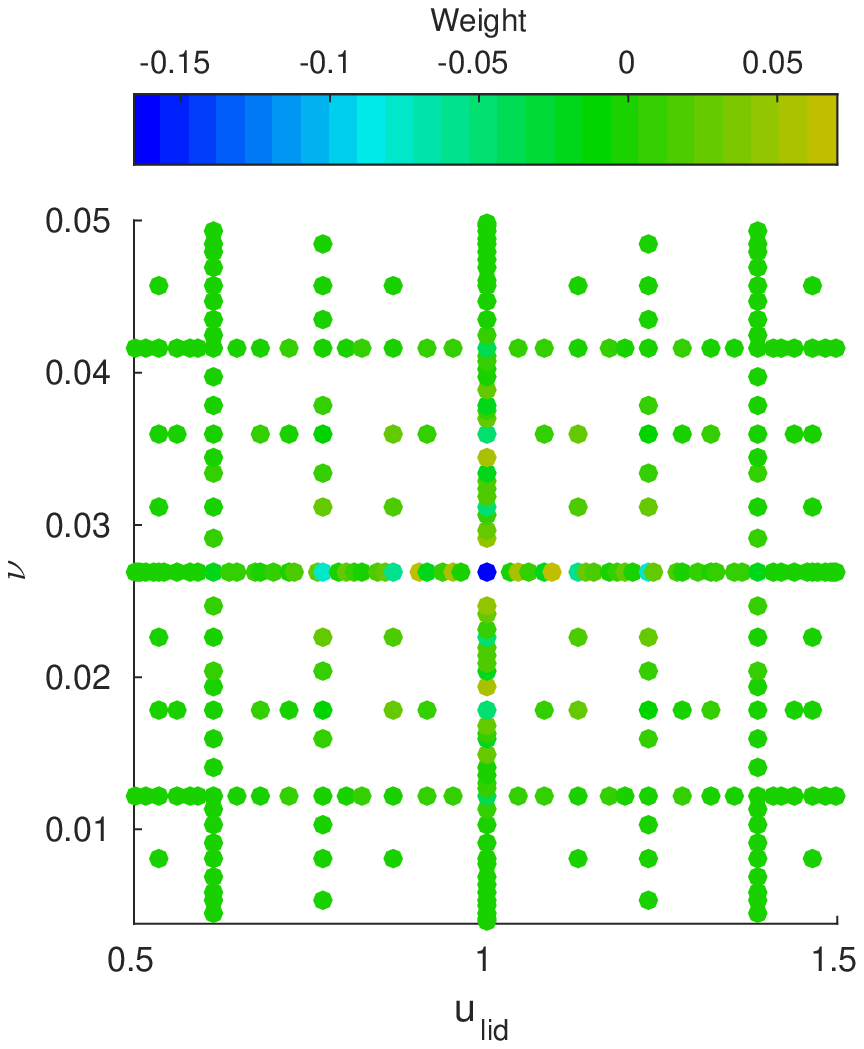}
		\subcaption{Smolyak sparse grid (321~nodes)}
		\label{fig:lbmgrids:sc1}
	\end{minipage}
	\begin{minipage}{.45\textwidth}
		\centering
		\includegraphics[width=\textwidth]{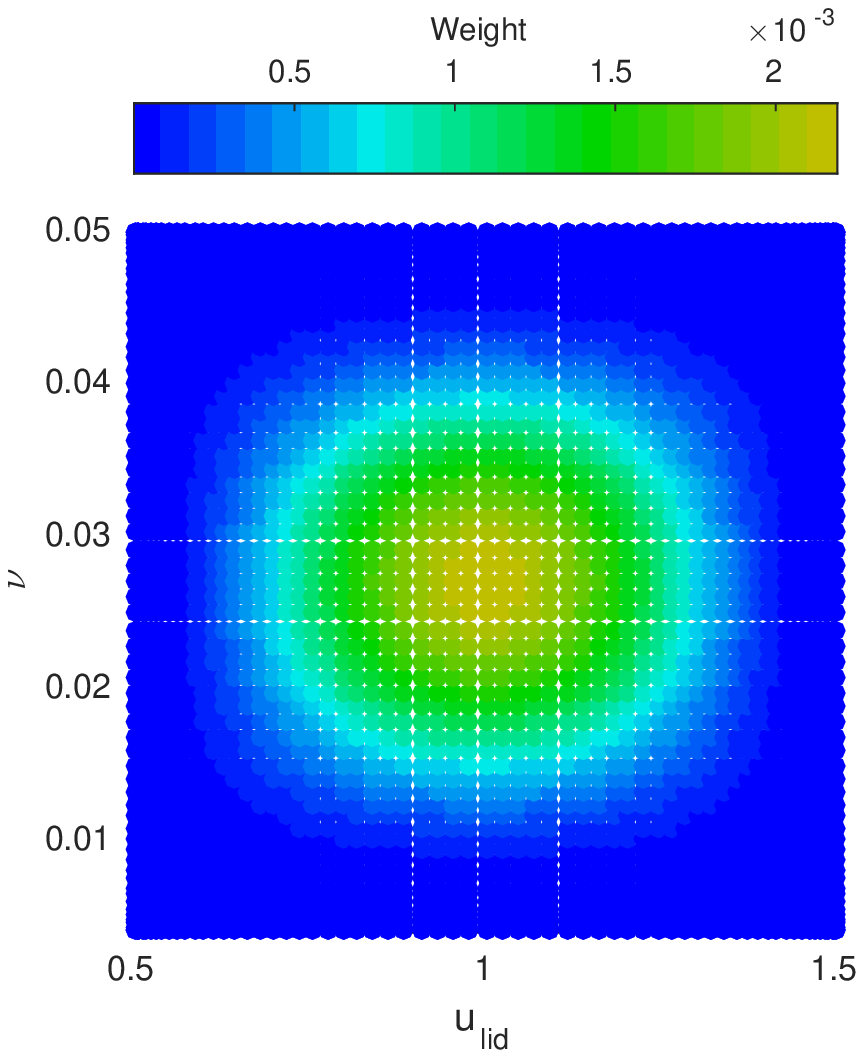}
		\subcaption{``Exact'' tensor product (4225~nodes)}
		\label{fig:lbmgrids:tensor}
	\end{minipage}
	\caption{The grids used for UQ in the Lattice Boltzmann test case. All grids (except the tensor product grid) are of degree 13.}
	\label{fig:lbmgrids}
\end{figure}

UQ is applied with four different methods:
\begin{enumerate}
	\item MC using random samples;
	\item SC with a Smolyak sparse grid created with positive symmetric reduced Gauss--Jacobi quadrature rules, see Figure~\ref{fig:lbmgrids:sc1};
	\item SC with the symmetric reduced rule initiated with a tensor grid with positive weights, see Figure~\ref{fig:lbmgrids:nd1};
	\item SC with the negative symmetric reduced rule, see Figure~\ref{fig:lbmgrids:nd2}.
\end{enumerate}
The number of nodes is chosen in such a way that the degree of the resulting rule equals 13. For the Smolyak cubature rule, this can be achieved by choosing $K = 8$, because then $2(K-d)+1 = 13$ (see Lemma~\ref{lmm:smoldegree}). The initial quadrature rule is chosen such that it is the finest quadrature rule used by the Smolyak procedure. For the reduced cubature rules, the initial cubature rule is a $13 \times 13$ tensor grid of Gaussian quadrature rules. All grids (except the MC nodes) are shown in Figure~\ref{fig:lbmgrids}.

For a measure of accuracy we take the $L^2$-norm of the difference between the predicted mean velocity field $\bar u(N)$ and the reference mean field $\bar u^*$, where $N$ is the number of nodes, i.e.\ 
\begin{equation}
	\epsilon(N) \coloneqq \|\bar u(N) - \bar u^* \|_2.
\end{equation}

The convergence is shown in Figure~\ref{fig:convergence}. The $\mathcal{O}(1/\sqrt{N})$ convergence of MC is already becoming apparent. The polynomial-based methods seem to show spectral convergence, suggesting that the response is smooth with respect to the parameters, as might be expected from physical considerations. Of the polynomial methods, the Smolyak rule and symmetric reduced rule perform approximately the same. The negative symmetric reduced rule performs poorly. The authors attribute this to the low dimension of the problem; we have seen the benefits of allowing negative weights primarily in five and more dimensions. In low-dimensional cases the difference in the number of nodes is too small to overcome the large absolute differences in weights (see the color bar in Figure~\ref{fig:lbmgrids:nd2}).

\begin{figure}[t]
	\centering
	\includegraphics[width=.5\textwidth]{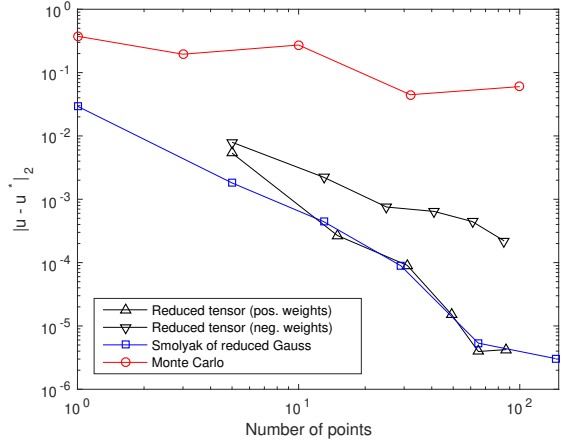}
	\caption{The 2-norm error of the Lattice Boltzmann test case using the discussed cubature rules for varying numbers of nodes.}
	\label{fig:convergence}
\end{figure}

\subsection{Aircraft aerodynamics test case, using the Euler equations and a finite-volume method}
\subsubsection{Problem description}
To show the performance the reduced cubature rule with negative weights (i.e.\ high accuracy with a relatively small number of nodes) we consider an aircraft aerodynamics test case with seven uncertain parameters. Only the reduced cubature rule with negative weights is used as the other cubature rules require too many nodes.

The geometry of the airplane is based on a sample airplane geometry of the program \texttt{sumo} \cite{Sumo2009}, the so-called ``twin-engine utility aircraft'' (see Figure~\ref{fig:plane}).

As flow model we consider the Euler equations of gas dynamics. The Euler-flow problem is solved using the second-order accurate finite-volume code $\text{SU}^2$ \cite{Palacios2014}.  The tools \texttt{sumo} and \texttt{TetGen} are used for mesh generation \cite{Si2015}. Besides modeling the surfaces, \texttt{sumo} creates surface meshes, which are used as input for \texttt{TetGen} which generates the volume meshes with a spherical far field boundary.  See Figure~\ref{fig:plane_example} for an example solution.

\begin{figure}
	\centering
	\includegraphics[width=.6\textwidth]{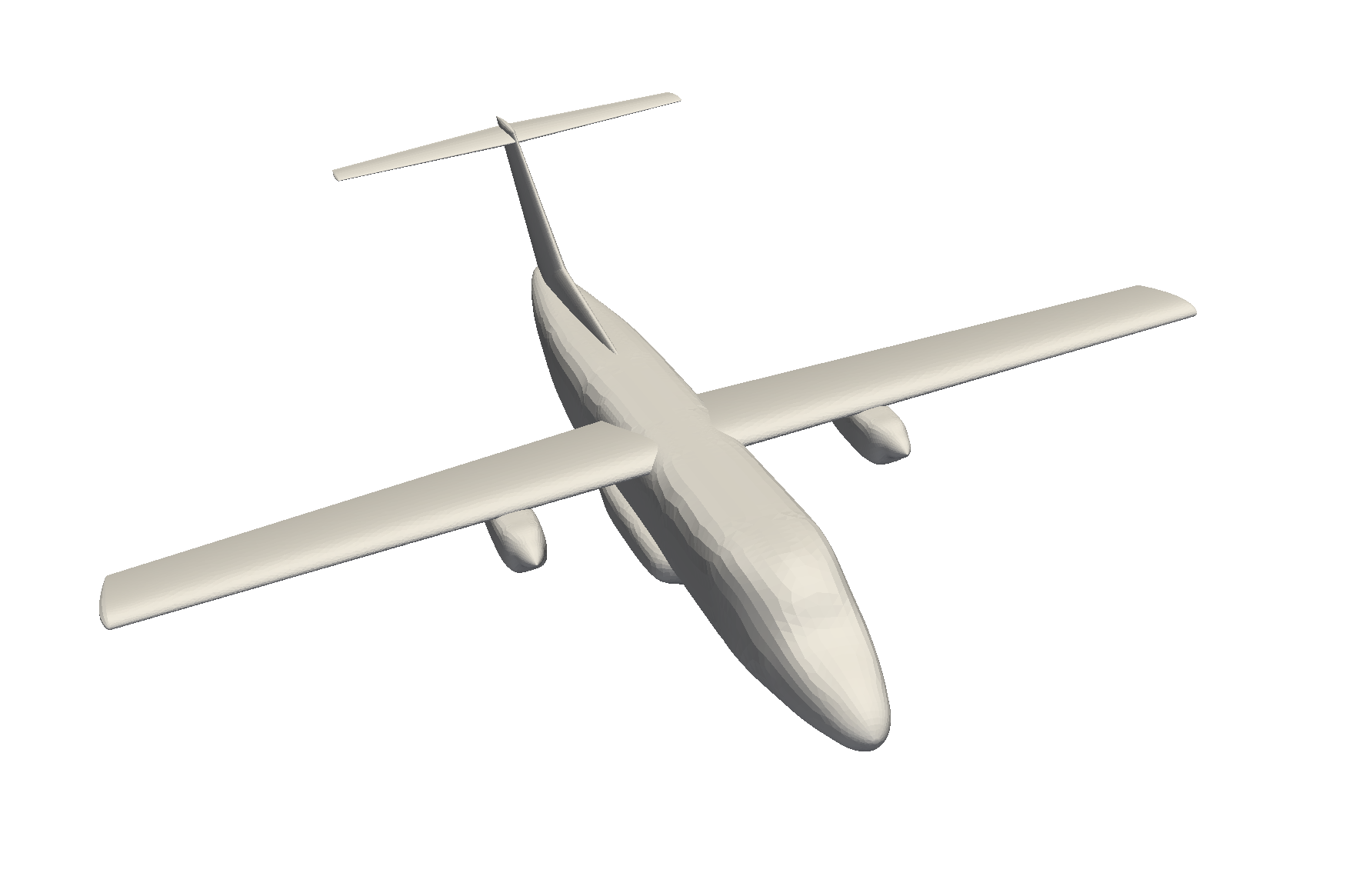}
	\caption{The geometry of the twin-engine utility aircraft without rotor blades.}
	\label{fig:plane}
\end{figure}

Of the seven uncertain parameters, three are geometrical and assumed to be normally distributed. The mean is the base geometry value and the standard deviation is defined to be 5\%. Four uncertain operational parameters are considered in addition, all modeled as $\beta(4,4)$ variables. See Table~\ref{tbl:acuparams} for details.

\begin{table}[t]
	\centering
	\caption{Uncertain parameters and their distribution as considered for the aircraft aerodynamics test case.}
	\label{tbl:acuparams}
	\begin{tabular}{ll}
		\textbf{Parameter} & \textbf{Distribution} \\
		\hline
		\hline
		Leading edge radius & $\mathcal{N}$ with 5\% standard deviation \\
		Maximum camber as percentage of the chord & $\mathcal{N}$ with 5\% standard deviation \\
		Distance of maximum camber from leading edge & $\mathcal{N}$ with 5\% standard deviation \\
		\hline
		Angle of incidence & $\beta(4, 4)$ with range $2.31^\circ \pm 5\%$ \\
		Side-slip angle & $\beta(4, 4)$ with range $0^\circ \pm 0.5^\circ$ \\
		Mach number & $\beta(4, 4)$ with range $0.72 \pm 5\%$ \\
		Free-stream pressure & $\beta(4,4)$ with range $101\,325 ~\textrm{N}/\textrm{m}^2 \pm 5\%$
	\end{tabular}
\end{table}

The geometrical uncertain parameters are specifically chosen such that the 4-digit NACA airfoil series can be used to parameterize them. The three parameters of the NACA series are essentially these parameters. The base geometry, which defines the mean of the distributions of these parameters, is that of the NACA2412 airfoil.

\begin{figure}
	\centering
	\includegraphics[width=.6\textwidth]{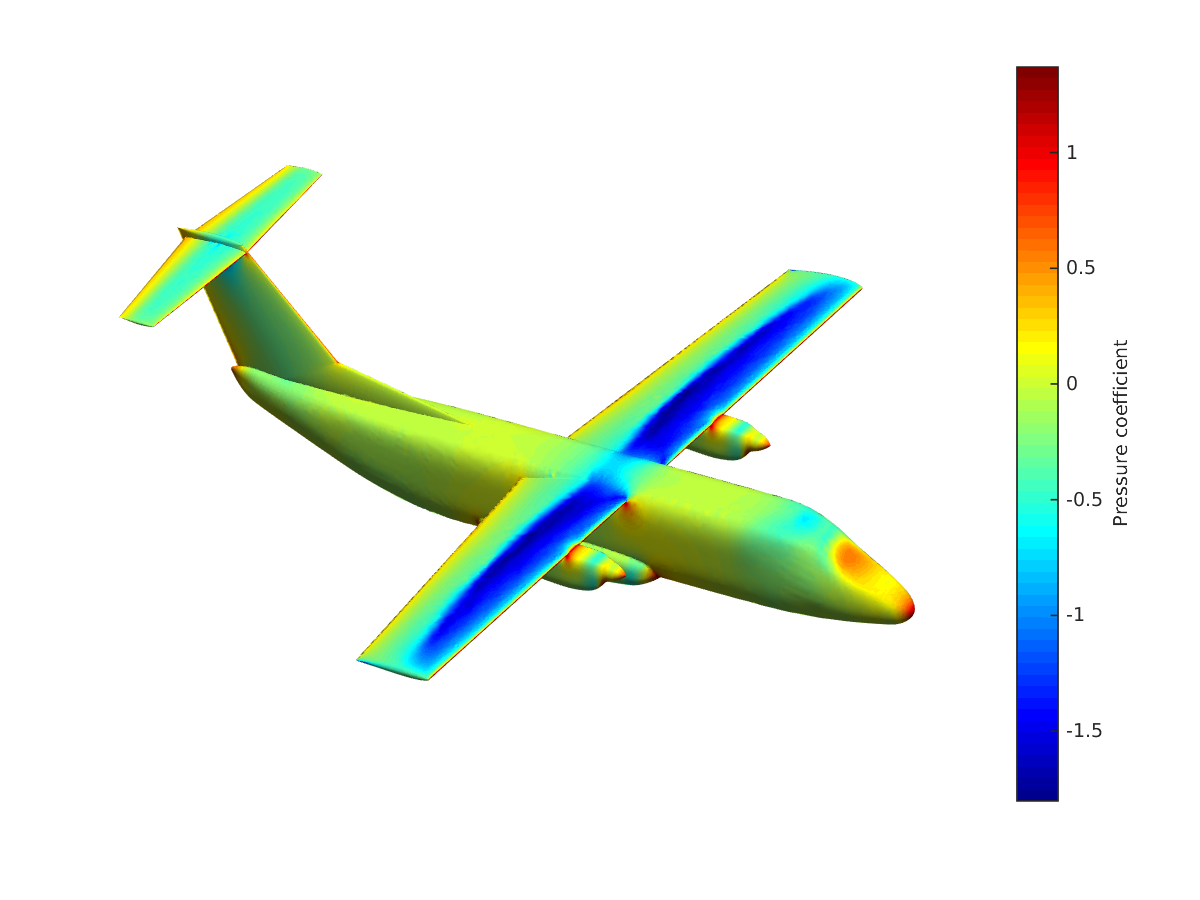}
	\caption{An example solution of the pressure coefficient at the wetted surface of the airplane. The uncertain inputs are fixed at their respective expected values.}
	\label{fig:plane_example}
\end{figure}

\subsubsection{Results}
A reduced tensor cubature rule with negative weights of degree 9 is generated, which yields a cubature rule of 1{,}293 nodes. A Smolyak sparse grid of the same degree consists of 2{,}465 nodes. If a tensor grid is used, then 78{,}125 simulations are necessary (if Gaussian rules are used). The reduced cubature rule with positive weights consists of 8{,}713 nodes in this case. This example shows that if time is an issue, allowing negative weights can indeed reduce the number of nodes significantly.

The lift, drag, and side-force coefficients are scalars obtained by integration over the wetted surface. Their moments are listed in Table~\ref{tbl:3dmoments} together with estimations based on least-squares regression. The degree of the polynomial fitted using least-squares is 5, which is the maximum number possible to keep the system determined, i.e.
\begin{equation}
	\dim \mathbb{P}(5, 7) < 1{,}293 < \dim \mathbb{P}(6, 7).
\end{equation}
Although the cubature rule has negative weights the variance is non-negative and the values are close to the least-squares estimates. Moreover, the order of magnitude of the lower-order moments seems physically reasonable (although no reference data is available for this case).

\begin{table}[t]
	\centering
	\caption{The first four non-central moments determined either using the cubature rule directly on the results (without hat) or using a high-degree cubature rule on the least-squares estimation (with hat). Empty places are values smaller than $10^{-5}$.}
	\label{tbl:3dmoments}
	\begin{tabular}{r || ll | ll | ll}
		\textbf{\#} & $\mathbf{c_l}$ & $\mathbf{\hat{c}_l}$ & $\mathbf{c_d}$ & $\mathbf{\hat{c}_d}$ & $\mathbf{c_\text{sf}}$ & $\mathbf{\hat{c}_\text{sf}}$ \\
		\hline
		\hline
		0 & 1.0000 & 1.0000 & 1.0000 & 1.0000 & 1.0000  & 1.0000  \\
		1 & 0.3640 & 0.3645 & 0.0226 & 0.0228 & -0.0015 & -0.0002 \\
		2 & 0.1326 & 0.1330 & 0.0005 & 0.0005 & &  \\
		3 & 0.0483 & 0.0485 &  &  &  &   \\
		4 & 0.0176 & 0.0177 &  &  &  & 
	\end{tabular}
\end{table}

We calculate moments of the pressure coefficient over the entire surface, and plot them in Figure~\ref{fig:planes}. Here, the four $\beta$-distributed uncertain parameters are taken into account, because geometrical uncertainties cannot be plotted. The mean looks very much like a typical pressure distribution. The variance (which is non-negative at all nodes on the geometry) highlights the location of the shocks, for which a small change in location leads to a large change in the pressure distribution. This result is consistent with UQ analyses of airfoil flows with shocks~\cite{Witteveen2009}. The higher-order moments (which have shown convergence due to the removal of the three geometrical uncertainties) are also large near the shock, indicating that it is unlikely that the resulting distribution is Gaussian.

\begin{figure}
	\centering
	\begin{minipage}{.45\textwidth}
		\includegraphics[width=\textwidth]{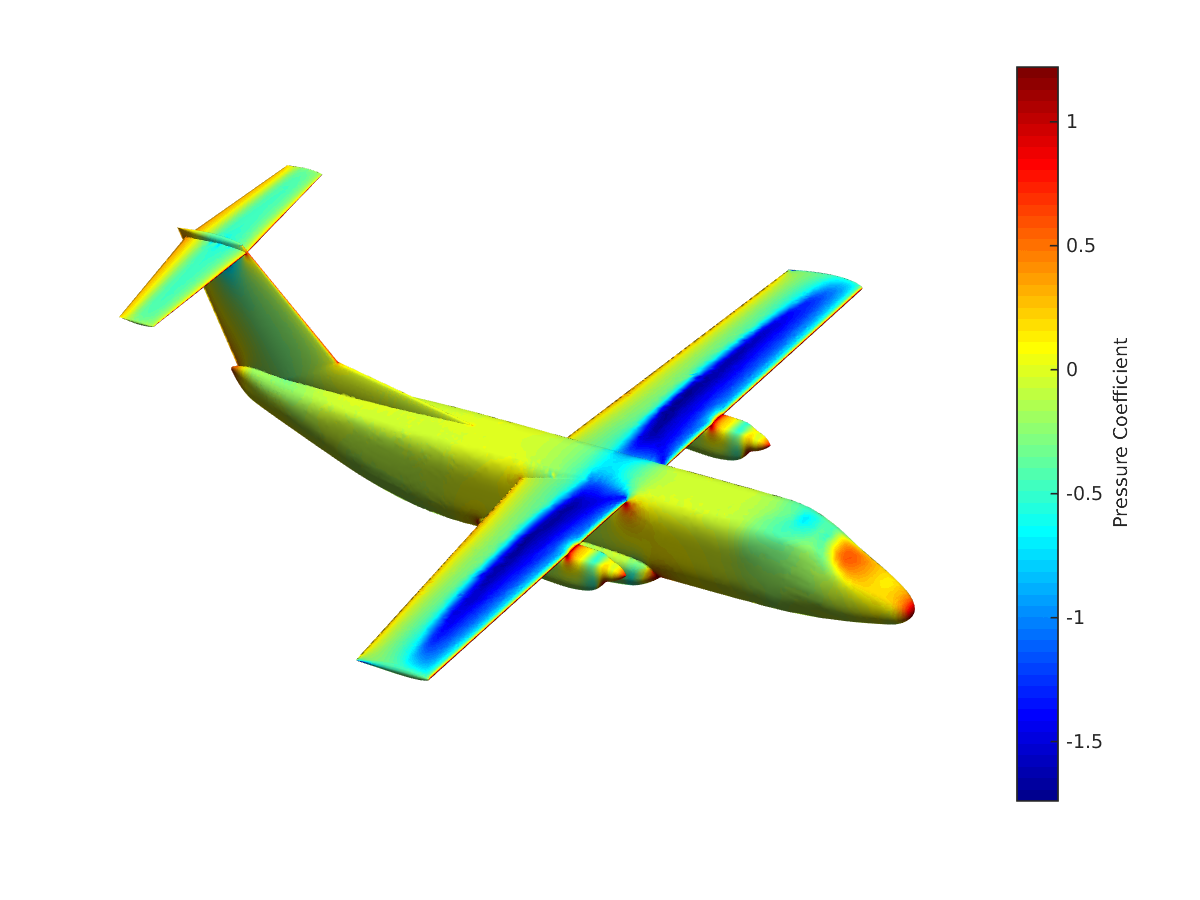}
		\subcaption{Mean}
	\end{minipage}
	\begin{minipage}{.45\textwidth}
		\includegraphics[width=\textwidth]{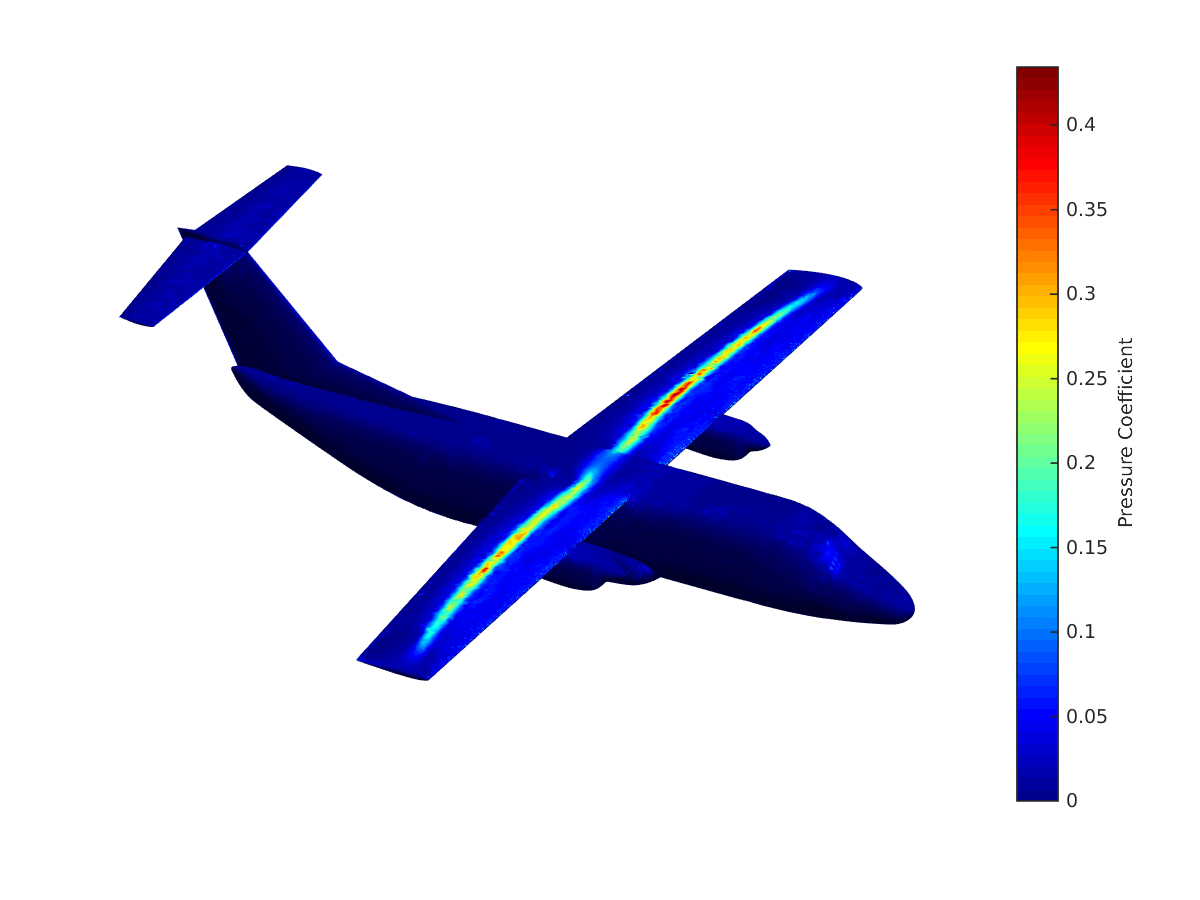}
		\subcaption{Variance}
	\end{minipage}
	\begin{minipage}{.45\textwidth}
		\includegraphics[width=\textwidth]{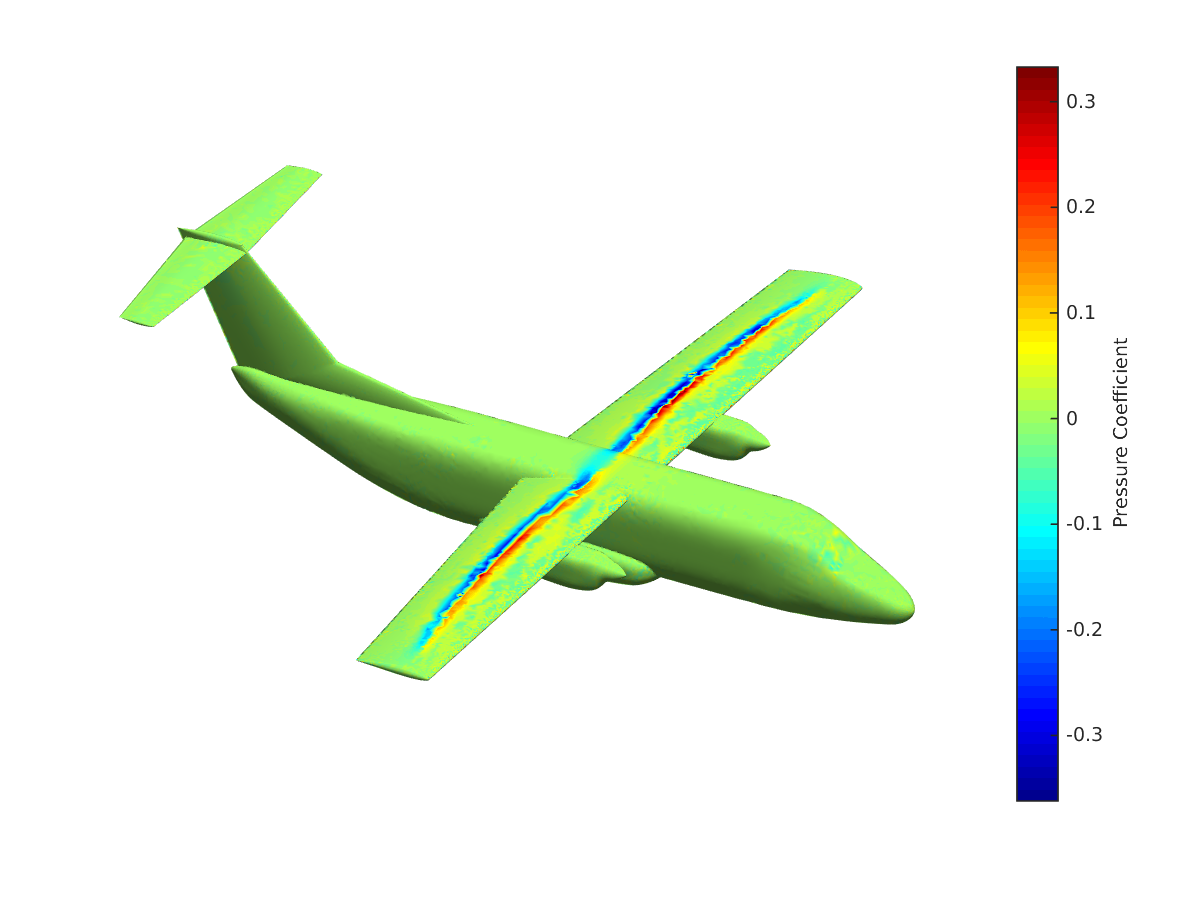}
		\subcaption{Third moment}
	\end{minipage}
	\begin{minipage}{.45\textwidth}
		\includegraphics[width=\textwidth]{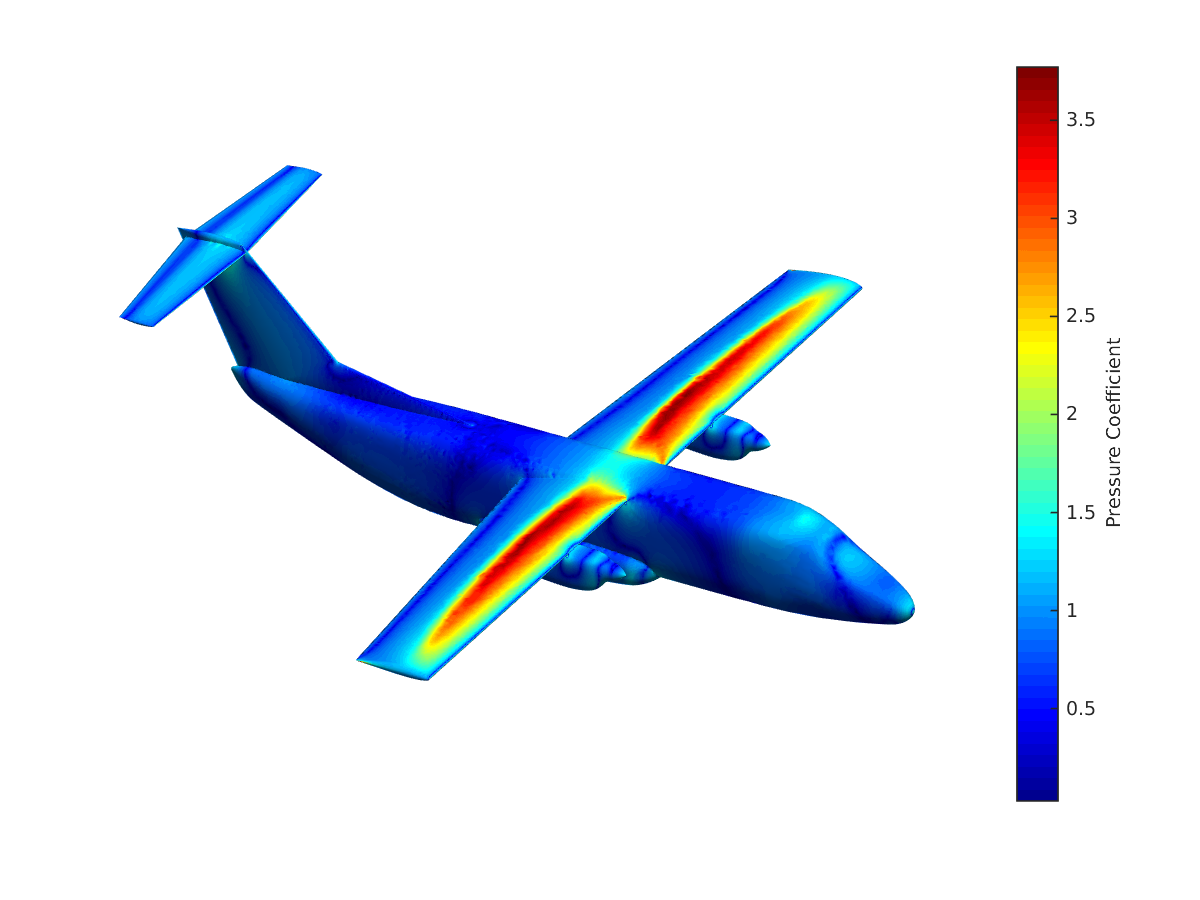}
		\subcaption{Fourth moment}
	\end{minipage}
	\caption{The first four central moments of the pressure coefficient at the wetted surface of the aircraft. Of the $k^\text{th}$ moment the $k^\text{th}$ root is taken such that the units are equal.}
	\label{fig:planes}
\end{figure}

\section{Conclusion}
Non-intrusive uncertainty quantification was studied using stochastic collocation methods. Three important properties are relevant for quadrature rules and cubature rules if they are used in stochastic collocation: nested, positive weights, and symmetry if the original distribution is symmetric. Existing quadrature rules and cubature rules do not have these three properties. The Clenshaw--Curtis quadrature rule is symmetric and nested, but does not have positive weights in general. The Gauss quadrature rule is symmetric and has positive weights, but is not nested.

The introduced quadrature rule performs well using a Smolyak grid. Any positive quadrature rule can be used to generate a set of nested quadrature rules, which can be used as input for a Smolyak procedure. Gaussian rules are quadrature rules which are always positive.

The proposed cubature rule comes in three variants: one which ignores symmetry, one which ignores positive weights, and one that has all properties. For low-dimensional problems, the cubature rule which satisfies all three properties has approximately the same nodes as a Smolyak sparse grid, but has positive weights. Convergence is also approximately equal, which can be seen in the results from the Genz test functions and the lid-driven cavity flow. If the uncertainty quantification problem is high-dimensional and computational efficiency is important, the positivity constraint can be relaxed to remove more nodes. The symmetric reduced cubature rule with possibly some negative weights yields less nodes than Smolyak grids and the reduced cubature rule with positive weights. In our example, higher-order moments remained positive, although the cubature rule has negative weights. This can become an issue though in applications where the response surface is more complex.

Although the cubature rules are nested, it is also important to note that the initial cubature rule still influences the result. If the initial cubature rule is too small, no general strategy exists to add nodes to the cubature rule and create a larger one, having the three properties. This is an option for further research.

\DeclareRobustCommand{\VAN}[1]{van #1}
\bibliographystyle{plainnatnourl}
\bibliography{literature}

\clearpage

\appendix

\section{Proofs of Lemmas~\ref{lmm:main1} and \ref{lmm:main2}}
\label{app:proofs}

\subsection{Proof of Lemma~\ref{lmm:main1}}
\begin{lemma*}
	Let $\{\boldsymbol \xi_1, \dots, \boldsymbol \xi_N\}$ be a type-1 symmetric cubature rule of degree $K$ with positive weights $\{w_1, \dots, w_N\}$. Let $Q$ be an orthant and let $N_{\bar{Q}}$ be the number of cubature nodes in $\bar{Q}$. Then there exists a symmetric null vector of $G_{-C}$ if
	\begin{equation}
		\binom{\lfloor \frac{K}{2} \rfloor + d}{d} < N_{\bar{Q}}.
	\end{equation}
\end{lemma*}
\begin{proof}
	\label{prf:lmm:main1}
	Without loss of generality, assume that the cubature rule is symmetric around 0. First we construct a suitable matrix $G'$.

	Let $Q$ be an orthant. Without loss of generality, let $\{\boldsymbol \xi_1, \dots, \boldsymbol \xi_{N_{\bar{Q}}}\} \in \bar{Q}$. Let $G'$ be the generalized Vandermonde-matrix of the nodes
	\begin{equation}
		\{2^{\|\boldsymbol \xi_1\|_0} \boldsymbol \xi_1, \dots, 2^{\|\boldsymbol \xi_{N_{\bar{Q}}}\|_0} \boldsymbol \xi_{N_{\bar{Q}}}\},
	\end{equation}
	omitting the monomials with an odd power. $G'$ is a $\binom{\lfloor \frac{K}{2} \rfloor + d}{d} \times N_{\bar{Q}}$-matrix.

	Assume that 
	\begin{equation}
		\binom{\lfloor \frac{K}{2} \rfloor + d}{d} < N_{\bar{Q}}
	\end{equation}
	holds. Let $\mathbf{c}'$ be a null vector of $G'$. Construct vector $\mathbf{c}$ as follows for $k = 1, \dots, N$:
	\begin{equation}
		c_k = c'_j,
	\end{equation}
	where $j$ is such that
	\begin{equation}
		|\boldsymbol \xi_k| = (|\xi^{(1)}_k|, |\xi^{(2)}_k|, \dots, |\xi^{(d)}_k|) = (|\xi^{(1)}_j|, |\xi^{(2)}_j|, \dots, |\xi^{(d)}_j|) = |\boldsymbol \xi_j|
	\end{equation}
	and $1 \leq j \leq N_{\bar{Q}}$. Due to the symmetry, such a $j$ always exists. Now $\mathbf{c}$ is a null vector of $G$. To see this, let a row index $i$ of $G$ be given, with row $\mathbf{r}_i$. A case distinction is made.

	\textbf{Case 1:} $m_i$ only contains even powers. Let $m_i(\boldsymbol \xi) = \boldsymbol \xi^\alpha$. Then
	\begin{align}
		\mathbf{r}_i \cdot \mathbf{c} &= \sum_{k=1}^N c_k m_i(\boldsymbol \xi_k) = \sum_{k=1}^N c_k m_i(|\boldsymbol \xi_k|) \\
		&= \sum_{k=1}^N c_k (-1)^\alpha |\boldsymbol \xi_k|^\alpha = \sum_{k=1}^N c_k |\boldsymbol \xi_k|^\alpha = \sum_{j=1}^{N_{\bar{Q}}} c'_j 2^{\|\boldsymbol \xi_j\|_0} |\boldsymbol \xi_j|^\alpha \\
		&= 0,
	\end{align}
	because this monomial was included in matrix $G'$.

	\textbf{Case 2:} $m_i$ contains an odd power. Without loss of generality, assume that the first power is odd, i.e., $m_i(\boldsymbol \xi) = \boldsymbol \xi^\alpha$, with $\alpha^{(1)}$ odd. Then let $I$ be the index set of cubature nodes with first element equal to 0, $J$ the index set of cubature nodes with first element larger than 0, and $K$ all other indices. Due to type-1 symmetry, the size of $J$ and $K$ is equal. Split the nodes along the plane of symmetry. Then
	\begin{align}
		\mathbf{r}_i \cdot \mathbf{c} &= \sum_{k=1}^N c_k m_i(\boldsymbol \xi_k) \\
		&= \sum_{k \in I} c_k m_i(\boldsymbol \xi_k) + \sum_{k \in J} c_k m_i(\boldsymbol \xi_k) + \sum_{k \in K} c_k m_i(\boldsymbol \xi_k) \\
		&= 0 + \sum_{k \in J} c_k \boldsymbol \xi_k^\alpha + \sum_{k \in K} c_k \boldsymbol \xi_k^\alpha \\
		&= \sum_{k \in J} c_k (\xi_k^{(1)})^{\alpha^{(1)}} (\xi_k^{(2 \dots d)})^{\alpha^{(2 \dots d)}} + \sum_{k \in K} c_k (\xi_k^{(1)})^{\alpha^{(1)}} (\xi_k^{(2 \dots d)})^{\alpha^{(2 \dots d)}}\\
		&= \sum_{k \in J} c_k (\xi_k^{(1)})^{\alpha^{(1)}} (\xi_k^{(2 \dots d)})^{\alpha^{(2 \dots d)}} + \sum_{k \in J} c_k (-\xi_k^{(1)})^{\alpha^{(1)}} (\xi_k^{(2 \dots d)})^{\alpha^{(2 \dots d)}} \\
		&= \sum_{k \in J} c_k (\xi_k^{(1)})^{\alpha^{(1)}} (\xi_k^{(2 \dots d)})^{\alpha^{(2 \dots d)}} - \sum_{k \in J} c_k (\xi_k^{(1)})^{\alpha^{(1)}} (\xi_k^{(2 \dots d)})^{\alpha^{(2 \dots d)}} \\
		&= 0. \qedhere
	\end{align}
\end{proof}

\subsection{Proof of Lemma~\ref{lmm:main2}}
\begin{lemma}
	\label{lmm:partitionnr}
	Let $\{\boldsymbol \xi_1, \dots, \boldsymbol \xi_N\}$ be a type-2 symmetric cubature rule of degree $K$ with positive weights $\{w_1, \dots, w_N\}$. Let $Q$ be an orthant after a rotation over $\frac{1}{4} \pi$ of all axes. Let $N_{\bar{Q}}$ be the number of cubature nodes in $\bar{Q}$. Then there exists a symmetric null vector of $G_{-C}$ if
	\begin{equation}
		1 + \sum_{l=1}^K p_d(l) < N_{\bar{Q}},
	\end{equation}
	where $p_d(l)$ is the restricted partition function.
\end{lemma}
\begin{proof}
	\label{prf:lmm:main2}
	Without loss of generality, assume that the cubature rule is symmetric around 0. The proof has the same structure as the proof of Lemma~\ref{lmm:symquad}, combined with Lemma~\ref{lmm:partitionnr}. Again, we first construct a suitable matrix $G'$.

	Let $Q$ be an orthant after $\frac{1}{4} \pi$ rotation of all basis vectors. Without loss of generality, let $\{\boldsymbol \xi_1, \dots, \boldsymbol \xi_{N_{\bar{Q}}}\} \in \bar{Q}$. Let $G'$ be the generalized Vandermonde-matrix of the nodes
	\begin{equation}
		\left\{\sum_{k=1}^{\#\sigma \boldsymbol \xi_1} \sigma_k \boldsymbol \xi_1, \dots, \sum_{k=1}^{\#\sigma \boldsymbol \xi_{N_{\bar{Q}}}} \sigma_k \boldsymbol \xi_{N_{\bar{Q}}}\right\},
	\end{equation}
	where $\#\sigma \boldsymbol \xi_k$ is the number of permutations of the elements of cubature node $\boldsymbol \xi_k$ and where $\sigma_k$ is the $k^\text{th}$ permutation operator, i.e., it is a sum over all permutations of $\boldsymbol \xi_k$. Omit all monomials $\boldsymbol \xi^\alpha$ with a power which is not \emph{sorted}. Due to Lemma~\ref{lmm:bound1inc}, $G'$ is a $\left(1 + \sum_{l=1}^K p_d(l)\right) \times N_{\bar{Q}}$-matrix. Assume that
	\begin{equation}
		1 + \sum_{l=1}^K p_d(l) < N_{\bar{Q}}
	\end{equation}
	holds. Let $\mathbf{c}'$ be a null vector of $G'$. Construct vector $\mathbf{c}$ as follows for $k = 1, \dots, N$:
	\begin{equation}
		c_k = c'_j,
	\end{equation}
	where $j$ is such that
	\begin{equation}
		\boldsymbol \xi_k = \sigma(\boldsymbol \xi_j),
	\end{equation}
	for a suitable permutation $\sigma$ with $1 \leq j \leq N_{\bar{Q}}$. This is well-defined because of the type-2 symmetry. Now $\mathbf{c}$ is a null vector of $G$. To see this, let a row index $i$ of $G$ be given, with row $\mathbf{r}_i$. Let $m_i(\boldsymbol \xi) = \boldsymbol \xi^\alpha$ be the respective monomial. Then:
	\begin{equation}
		\mathbf{r}_i \cdot \mathbf{c} = \sum_{k=1}^N c_k m_i(\boldsymbol \xi_k) = \sum_{k=1}^N c_k \boldsymbol \xi_k^\alpha.
	\end{equation}
	For each $\boldsymbol \xi_k$, there exists a permutation operator $\sigma_k$, such that $\sigma_k(\boldsymbol \xi_k) = \boldsymbol \xi_j$ for $1 \leq j \leq N_{\bar{Q}}$. Hence:
	\begin{align}
		\sum_{k=1}^N c_k \boldsymbol \xi_k^\alpha = \sum_{k=1}^N c_k (\sigma_k \boldsymbol \xi_{j_k})^\alpha
		= \sum_{j=1}^{N_{\bar{Q}}} c'_j \sum_{k=1}^{\#\sigma \boldsymbol \xi_j} (\sigma_k \boldsymbol \xi_j)^\alpha
		&= 0. \qedhere
	\end{align}
\end{proof}

\section{Algorithms for reduced cubature rule generation}
\label{app:algs}
The algorithms provided in this appendix are not yet efficient for high-dimensional cubature rules of high degree, and are merely given for sake of completeness. Examples of possible optimizations are the following: numerical issues can arise when determining the null vectors and time and memory issues can arise if the complete $G_{-C}$ is constructed. The first issue can be partially overcome by scaling all elements of the nodes onto the same interval and the second issue can be circumvented by constructing the matrix $G_{-C}$ column-wise and checking for existence of a null vector after each addition of a column.

\begin{algorithm}[H]
\caption{Determining the reduced cubature rule}
\label{alg:redcubrule}
\begin{algorithmic}[1]
\Require {Cubature rule nodes $\{\xib_1, \xib_2, \dots, \xib_N\}$ and weights $\{w_1, w_2, \dots, w_N\}$ of degree $K$, with $N = \dim \mathbb{P}(K, d)$.}
\Ensure {Non-negative weights $\{w_1^*, w_2^*, \dots, w_N^*\}$ having $\bigl(\dim \mathbb{P}(K-1, d)\bigr)$ number of non-zero entries, such that the resulting quadrature rule has degree $K-1$.}
\Statex ~
\State Construct $G_{-C}$ of the nodes using all monomials up to degree $K-1$ (see Section~\ref{subsec:redstep})
\State Determine $C$ null vectors $\mathbf{c}^{(1)}, \mathbf{c}^{(2)}, \dots, \mathbf{c}^{(C)}$ of $G_{-C}$
\For {$i = 1, \dots, C$} \label{alg:redcubrule:forstart}
\State $\mathbf{c} \gets \mathbf{c}^{(i)}$
\State $\alpha^{(1)} \gets \min_{k=1, \dots, N} \left\{\frac{w_k}{c_k} : c_k > 0\right\}$
\State $\alpha^{(2)} \gets \max_{k=1, \dots, N} \left\{-\frac{w_k}{c_k} : c_k < 0\right\}$
\State Let $k^{(1)}$ and $k^{(2)}$ be such that $\alpha^{(1)} = w_{k^{(1)}} / c_{k^{(1)}}$ and $\alpha^{(2)} = w_{k^{(2)}} / c_{k^{(2)}}$
\State $w^{(1)}_k \gets w_k - \alpha^{(1)} c_k$ and $w^{(2)}_k \gets w_k + \alpha^{(2)} c_k$ for $k = 1, \dots, N$.
\Statex ~
\State \emph{Here, a selection criterion can be applied:}
\State Pick $l = 1$ or 2 and let $\left\{w\right\} \gets \left\{w^{(l)}\right\}$
\Statex ~
\For {$j = i+1, \dots, C$}
\State $\mathbf{c}^{(j)} \gets \mathbf{c}^{(j)} - \mathbf{c}^{(i)} c^{(j)}_{k^{(l)}}/ c^{(i)}_{k^{(l)}}$
\State \emph{Now, $c^{(j)}_{k^{(l)}} = 0$}
\EndFor
\EndFor \label{alg:redcubrule:forend}
\State \textbf{return} $\{ w \}$
\end{algorithmic}
\end{algorithm}

\begin{algorithm}[H]
\caption{Determining the \emph{symmetric} reduced cubature rule}
\label{alg:symredcubrule}
\begin{algorithmic}[1]
\Require {Type-1 and type-2 symmetric cubature rule nodes $\{\xib_1, \xib_2, \dots, \xib_N\}$ and weights $\{w_1, w_2, \dots, w_N\}$ for degree $K$, with $N = \dim \mathbb{P}(K, d)$.}
\Ensure {Non-negative weights $\{w_1^*, w_2^*, \dots, w_N^*\}$ having at most $2^d \left(1 + \sum_{l=1}^{\lfloor K / 2 \rfloor} p_d(l)\right)$ non-zero entries, such that the resulting quadrature rule has degree $K-1$.}
\Statex ~
\State Let $G'$ be the generalized Vandermonde-matrix of the nodes (see \ref{app:proofs}):
\begin{equation}
		\left\{2^{\|\boldsymbol \xi_1\|_0} \sum_{k=1}^{\#\sigma \boldsymbol \xi_1} \sigma_k \boldsymbol \xi_1, \dots, 2^{\|\boldsymbol \xi_{N_{\bar{Q}}}\|_0} \sum_{k=1}^{\#\sigma \boldsymbol \xi_{N_{\bar{Q}}}} \sigma_k \boldsymbol \xi_{N_{\bar{Q}}}\right\}.
\end{equation}
\State Determine $K$ null vectors $\mathbf{c}'^{(1)}, \mathbf{c}'^{(2)}, \dots, \mathbf{c}'^{(K)}$ of $G'$
\State Determine $K$ symmetric null vectors $\mathbf{c}^{(1)}, \dots, \mathbf{c}^{(K)}$ of $G_{-C}$
\State Execute step \ref{alg:redcubrule:forstart} until \ref{alg:redcubrule:forend} of Algorithm~\ref{alg:redcubrule}, with $C \gets K$
\end{algorithmic}
\end{algorithm}

\begin{algorithm}[H]
\caption{Determining the \emph{negative symmetric} reduced cubature rule}
\label{alg:negsymredcubrule}
\begin{algorithmic}[1]
\Require {Type-1 and type-2 symmetric cubature rule nodes $\{\xib_1, \xib_2, \dots, \xib_N\}$ and weights $\{w_1, w_2, \dots, w_N\}$ for degree $K$, with $N = \dim \mathbb{P}(K, d)$.}
\Ensure {Non-negative weights $\{w_1^*, w_2^*, \dots, w_N^*\}$ having at most $2^d \left(1 + \sum_{l=1}^{\lfloor K / 2 \rfloor} p_d(l)\right)$ non-zero entries, such that the resulting quadrature rule has degree $K-1$.}
\Statex ~
\State Let $G'$ be the generalized Vandermonde-matrix of the nodes (see \ref{app:proofs}):
\begin{equation}
		\left\{2^{\|\boldsymbol \xi_1\|_0} \sum_{k=1}^{\#\sigma \boldsymbol \xi_1} \sigma_k \boldsymbol \xi_1, \dots, 2^{\|\boldsymbol \xi_{N_{\bar{Q}}}\|_0} \sum_{k=1}^{\#\sigma \boldsymbol \xi_{N_{\bar{Q}}}} \sigma_k \boldsymbol \xi_{N_{\bar{Q}}}\right\}.
\end{equation}
\State Determine $K$ null vectors $\mathbf{c}'^{(1)}, \mathbf{c}'^{(2)}, \dots, \mathbf{c}'^{(K)}$ of $G'$
\State Determine $K$ symmetric null vectors $\mathbf{c}^{(1)}, \dots, \mathbf{c}^{(K)}$ of $G_{-C}$ (see Section~\ref{subsec:redstep})

\For {$i = 1, \dots, K$} 
\State $\mathbf{c} \gets \mathbf{c}^{(i)}$
\State $S_K \gets \left\{k | c_k \neq 0\right\}$
\State $\gamma \gets \max_{ k \in S_K } \left\{ 2^{\| \xib_k \|_0} \# \sigma \xib_k \right\}$
\State Let $k_0$ be such that $\gamma = 2^{\| \xib_{k_0} \|_0} \# \sigma \xib_{k_0}$
\State $\alpha \gets w_{k_0} / c_{k_0}$
\State $w_k \gets w_k - \alpha c_k$ for $k = 1, \dots, N$
\For {$j = i+1, \dots, K$}
\State $\mathbf{c}^{(j)} \gets \mathbf{c}^{(j)} - \mathbf{c}^{(i)} c^{(j)}_{k_0}/ c^{(i)}_{k_0}$
\State \emph{Now, $c^{(j)}_{k_0} = 0$}
\EndFor
\EndFor
\State \textbf{return} $\{ w \}$

\end{algorithmic}
\end{algorithm}

\end{document}